\documentclass[11pt]{article}

\usepackage[utf8]{inputenc}
\usepackage[T1]{fontenc}
\usepackage{textcomp}
\usepackage{amsmath,amssymb,amsthm}
\usepackage{lmodern}
\usepackage[a4paper, margin=0.8in]{geometry}
\usepackage{xcolor, pict2e}
\usepackage{microtype}
\usepackage{listings}
\usepackage{moreverb}
\usepackage{hyperref}
\usepackage[sans]{dsfont}
\usepackage[font=sf, labelfont={sf,bf}, margin=1cm]{caption}

\usepackage{import}
\usepackage{xifthen}
\usepackage{pdfpages}
\usepackage{transparent}
\usepackage{chngcntr}
\counterwithin{figure}{section}
\usepackage{subcaption}

\definecolor{auburn}{rgb}{0.43, 0.21, 0.1}

\usepackage{hyperref}
\hypersetup{
    colorlinks=true,
    linkcolor=blue,
    citecolor=auburn,
    urlcolor=blue,
    pdfborder={0 0 0}
}

\newtheorem*{theorem*}{Theorem}
\newtheorem{theorem}{Theorem}[section]
\newtheorem{proposition}[theorem]{Proposition}
\newtheorem{lemma}[theorem]{Lemma}
\newtheorem{corollary}[theorem]{Corollary}
\newtheorem{notation}[theorem]{Notation}
\newtheorem{definition}[theorem]{Definition}
\newtheorem{claim}[theorem]{Claim}

\theoremstyle{remark}
\newtheorem{remark}[theorem]{Remark}

\newcommand{\R}{\mathbb{R}}
\newcommand{\N}{\mathbb{N}}
\newcommand{\Z}{\mathbb{Z}}

\newcommand{\C}{\mathbb{C}}
\newcommand{\T}{\mathbb{T}}
\newcommand{\D}{\mathbb{D}}

\newcommand{\Cc}{\mathcal{C}}

\newcommand{\Fc}{\mathcal{F}}
\newcommand{\Gc}{\mathcal{G}}

\newcommand{\Lc}{\mathcal{L}}

\newcommand{\expect}{\mathbb{E}}
\newcommand{\Expect}[1]{\mathbb{E} \left[ #1 \right] }
\newcommand{\EXPECT}[2]{\mathbb{E}_{#1} \left[ #2 \right] }

\newcommand{\Prob}[1]{\mathbb{P} \left( #1 \right) }
\newcommand{\PROB}[2]{\mathbb{P}_{#1} \left( #2 \right) }

\renewcommand{\P}{\mathbb{P}}
\newcommand{\E}{\mathbb{E}}

\newcommand{\abs}[1]{\left\vert #1 \right\vert}
\newcommand{\norme}[1]{\left\| #1 \right\| }

\newcommand{\floor}[1]{\left\lfloor #1 \right\rfloor}
\newcommand{\ceil}[1]{\left\lceil #1 \right\rceil}
\newcommand{\indic}[1]{ \mathbf{1}_{ \left\{ #1 \right\} } }
\newcommand{\eps}{\varepsilon}

\DeclareMathOperator{\CR}{CR}
\DeclareMathOperator{\Int}{Int}
\DeclareMathOperator{\Range}{Range}
\renewcommand{\d}{\mathrm{d}}

\renewcommand{\Cap}{\mathrm{Cap}}
\newcommand{\cyl}{\mathtt{C}}

\def \loopmeasure{\mu^{\rm loop}}
\def \bubmeasure{\mu^{\rm bub}}

\title{Crossing exponent in the Brownian loop soup}
\author{Antoine Jego\thanks{EPFL; antoine.jego@epfl.ch} \and Titus Lupu\thanks{Sorbonne Université and Université Paris Cité, CNRS, Laboratoire de Probabilités, Statistique et Modélisation, F-75005 Paris, France; titus.lupu@sorbonne-universite.fr} \and Wei Qian\thanks{City University of Hong Kong; weiqian@cityu.edu.hk}}
\date {}

\numberwithin{equation}{section}

\begin{document}

\maketitle

\abstract{
We study the clusters of loops in a Brownian loop soup in some bounded two-dimensional domain with subcritical intensity $\theta \in (0,1/2]$.
We obtain an exact expression for the asymptotic probability of the existence of a cluster crossing a given annulus of radii $r$ and $r^s$ as $r \to 0$ ($s >1$ fixed).
Relying on this result, we then show that the probability for a macroscopic cluster to hit a given disc of radius $r$ decays like $|\log r|^{-1+\theta+ o(1)}$ as $r \to 0$.
Finally, we characterise the polar sets of clusters, i.e. sets that are not hit by the closure of any cluster, in terms of $\log^\alpha$-capacity.

This paper reveals a connection between the 1D and 2D Brownian loop soups. This connection in turn implies the existence of a second critical intensity $\theta = 1$ that describes a phase transition in the percolative behaviour of large loops on a logarithmic scale targeting an interior point of the domain.
}

\setcounter{tocdepth}{2}
\tableofcontents

\section{Introduction}

For a planar domain $D \subset \C$, the Brownian loop soup $\Lc_D^\theta$ in $D$ with intensity $\theta >0$ is a random collection of Brownian loops lying in $D$ that is distributed according to a Poisson point process with intensity $\theta \loopmeasure_D$. The measure $\loopmeasure_D$ on Brownian loops is defined by
\begin{equation}
\loopmeasure_D( \d \wp ) = \int_D \d z \int_0^\infty \frac{\d t}{t}\, p_D(t,z,z)\, \P_D^{t,z,z}( \d \wp)
\end{equation}
where $p_D$ is the heat kernel in $D$ and $\P_D^{t,z,z}$ denotes the law of a Brownian bridge in $D$ from $z$ to $z$ of duration $t$. Because the total mass of $\loopmeasure_D$ is infinite, $\Lc_D^\theta$ contains countably infinitely many loops a.s.

First introduced by \cite{Lawler04}, the Brownian loop soup became a central object of study in 2D random conformal geometry, in particular due to its connections with the Gaussian free field (GFF) and with Conformal loop ensemble (CLE)/Schramm-Loewner evolution (SLE).
Note that a different choice of the intensity parameter $c=2\theta$ was used in \cite{Lawler04, SheffieldWernerCLE} and many other works, where $c$ corresponds to the ``central charge'' in conformal field theory. 
As established in the seminal work \cite{SheffieldWernerCLE}, the percolative property of the loop soup in the unit disk $\D$ (this can be easily extended to any domain $D$ such that $\C\setminus D$ is non polar for the Brownian motion) undergoes a phase transition:
\begin{itemize}
\item
Supercritical ($c>1, \theta>1/2$): there is a.s.\ a unique cluster of loops;
\item
Critical and subcritical ($c\le 1, \theta \leq 1/2$): there are a.s.\ infinitely many clusters. In this case, the outermost boundaries of the outermost clusters are distributed as a  CLE$_\kappa$ where $\kappa \in (8/3,4]$ and $\theta$ are related by
\begin{align}\label{theta-kappa}
2\theta=c=(6-\kappa)(3\kappa-8)/(2\kappa).
\end{align}
\end{itemize}
As a consequence of the relation to CLE, one can compute many exact quantities about the critical and subcritical loop soup. 
For example, for $\theta\in(0,1/2]$, the probability $p(\theta, n,r,R)$ of having $n\ge 2$ different clusters in $\Lc_\D^\theta$ that cross the annulus $R\D \setminus r\D$ for $R>r>0$ is the same as the probability that the (outer  or inner) boundaries of these $n$ clusters cross this annulus.
Even though \cite{SheffieldWernerCLE} only looked at the outer boundaries of the outermost clusters, both the inner and outer boundaries of any cluster can be transformed into outer boundaries of outermost clusters after some conditioning and conformal mapping (see e.g. \cite[Section 7]{GaoLiQian2022mutliple}).
We believe that with some additional work, it should be possible to show that 
\[p(\theta, n,r,R)=r^{\alpha_{2n}+o(1)} \text{ as } r\to 0, \quad \text{ for } \alpha_{2n}=(16n^2-(\kappa-4)^2)/(8\kappa)\]
where $\theta$ and $\kappa$ are related through \eqref{theta-kappa}, and $\alpha_{2n}$ is the $2n$-arm exponent of SLE (following from \cite{MR1879816,MR1879850,MR1879851,MR1899232}, see also \cite{wu2018arms}).

The main purpose of the current paper is to obtain the crossing probability for one cluster, i.e.\ $n=1$, for all intensities $\theta\in(0,1/2]$. This case is fundamentally different from the $n\ge 2$ case, because as a cluster crosses the annulus $R\D \setminus r\D$, its inner or outer boundaries typically do not cross this annulus. 
In particular, we do not expect to be able to express the crossing probability for $n=1$ naturally in terms of SLE or CLE. As we will reveal  below, the crossing probability for $n=1$ will decay as
$$|\log r|^{-1+\theta+o(1)},$$
much slower than the polynomial rate for $n\ge 2$. 
Heuristically, as a macroscopic cluster approaches a given point, it will contain infinitely many small loops that surround this point. In contrast, in the $n\ge 2$ case, we require each cluster not to contain small loops that surround the origin, in order to leave some space for the other cluster(s) to make the crossing(s), which is a much rarer event.

The value $c=1$ ($\theta=1/2$) is known to be critical for various reasons, and is the central charge associated with the Gaussian free field. As shown in  \cite{SheffieldWernerCLE}, the outer boundaries of the critical loop soup $(\theta=1/2)$ is distributed as an CLE$_4$. It is also known that SLE$_4$ or CLE$_4$ can be interpreted as the level lines of the Gaussian free field (see e.g.\ \cite{MR2486487,MR3101840,MR2525778,MR3708206,MR3936643}).
A more direct connection with the Gaussian free field was established by Le Jan \cite{LeJan2010LoopsRenorm, LeJan2011Loops} via the occupation time field.
Thanks to these connections to the GFF, the structure of the clusters at $\theta=1/2$ is very well understood 
\cite{LupuIsomorphism, Lupu19ConvCLE, QianWerner19Clusters, ALS1, ALS2}. However, the understanding of the clusters when $\theta < 1/2$ is still limited to their outer boundaries \cite{SheffieldWernerCLE, BCL16} and to the loops touching them \cite{Qian19BLportionboundary,Qian21DisconnectExponent,GaoLiQian2022mutliple}.

In this paper, we study in great details the clusters of the loop soup when $\theta \in (0,1/2]$, considerably refining the picture when $\theta \in (0,1/2)$. 
In a nutshell, we compute the crossing exponent of these clusters (Theorem \ref{T:large_crossing}) and give a characterisation of the polar sets of the clusters (Theorem \ref{T:polar}). We also reveal a connection between the 1D and 2D Brownian loop soups (Section \ref{S:intro_outline}) and show that it implies the existence of a second critical point $\theta =1$ (Section \ref{S:intro_2values}).

\subsection{Main results}\label{S:intro_crossing}

\begin{notation}
For any family $\Lc$ of loops and any sets $A, B \subset \C$, we will denote by $\{ A \overset{\Lc}{\leftrightarrow} B \}$ the event that there is a cluster of loops in $\Lc$ that intersects both $A$ and $B$.
\end{notation}

The first main result of this paper is:

\begin{theorem}\label{T:large_crossing}
Let $\theta \in (0,1/2]$ and let $R_0 \in (0,1)$ be a macroscopic radius. Then
\[
\P \Big( R_0 \partial \D \overset{\Lc_\D^\theta}{\longleftrightarrow} r \partial \D \Big)
= |\log r|^{-1+ \theta +o(1)}
\quad \quad \text{as~} \quad r \to 0.
\]
\end{theorem}


In the companion paper \cite{JLQ23b}, we build and study the properties of a field naturally associated to a subcritical loop soup. This field can be thought of as the generalisation of the Gaussian free field to subcritical intensities $\theta < 1/2$. Theorem \ref{T:large_crossing} above is a crucial input in \cite{JLQ23b} to show that the covariance of this field at points $x$ and $y$ blows up like $|\log |x-y||^{2(1-\theta) + o(1)}$ as $x-y \to 0$.

Theorem \ref{T:large_crossing} can be seen as an interpolation of the two known cases $\theta =1/2$ and $\theta \to 0^+$. Indeed, when the intensity $\theta = 1/2$, the connection with the Gaussian free field allows one to show that the crossing probability is asymptotic to a constant times $|\log r|^{-1/2}$ (see Lemma \ref{L:theta=1/2}).
When $\theta \to 0^+$, the crossing probability can be informally compared with the probability that a Brownian path starting somewhere in the unit disc (not at the origin)
hits $r \D$ before reaching the unit circle. This probability decays like a constant times $|\log r|^{-1}$.
We believe that the probability in Theorem~\ref{T:large_crossing} should also be asymptotic to a constant times $|\log r|^{-1+\theta}$ (see also \cite[Section 9.5]{JLQ23b} for strong heuristics in this direction), but we do not achieve this precision in this work. However, remarkably, we are able to obtain the following exact expression of the ``scaling limit'' of a certain crossing probability (see Theorem \ref{T:convergence_crossing} for a stronger statement).
As the logarithmic decay suggests, the relevant events to consider here are scaled versions of one another only on a logarithmic scale.

\begin{theorem}\label{T:intro_convergence_crossing}
Let $\theta \in (0,1/2]$.
For all $s >1$, the following crossing probability converges
\begin{align}\label{eq:convergence_crossing}
\P \Big( \delta \partial \D \overset{\Lc_\D^\theta}{\longleftrightarrow} \delta^s \partial \D \Big)
\xrightarrow[\delta \to 0]{} f_\infty(s)
\end{align}
where 
\begin{align}\label{eq:f_infty}
 f_\infty(s)= \frac{\sin(\pi \theta)}{\pi} \int_{s-1}^\infty t^{\theta-1} (t+1)^{-1} \d t.
\end{align}
\end{theorem}

See \eqref{E:arcsine_law} for the link with the generalised arcsine law with parameter $1-\theta$ and a discussion concerning regenerative sets.

Theorem~\ref{T:intro_convergence_crossing} is another demonstration of some form of ``exact solvability'' of the loop soup (in addition to the ``solvability'' coming from its connection with CLE). To prove Theorem~\ref{T:intro_convergence_crossing}, we first work out a functional equation that is satisfied by $f_\infty$, and then identify $f_\infty$ as the unique fixed point of the functional equation.
The outline of the proof of Theorem~\ref{T:intro_convergence_crossing} will be further explained in Section~\ref{S:intro_outline}, where we also point out the intimate connection between $f_\infty$ and the 1D loop soup.

One can observe from \eqref{eq:f_infty} that 
\begin{equation}
\label{E:f_infty_asymp}
f_\infty(s) \sim \frac{\sin(\pi \theta)}{\pi(1-\theta)} s^{\theta-1}, \quad \quad \text{as} \quad s \to \infty.
\end{equation}
We will first focus on showing Theorem~\ref{T:intro_convergence_crossing}, and then show (in Section \ref{S:large_crossing}) that the exponent $1-\theta$ indeed describes the decay of the probability of a macroscopic crossing, thus proving Theorem~\ref{T:large_crossing}.

\medskip

The crossing probability is an important quantity of interest in many other models in statistical mechanics. In the well-known model of Bernoulli site percolation on the triangular lattice, at criticality, the probability of one cluster crossing an annulus $\D\setminus \eps\D$ decays like $\eps^{5/48+o(1)}$ \cite{MR1887622}.  One can also look at the probability of having $k\ge 2$ clusters, that is, $2k$ disjoint ``arms'' with alternating colors (or more generally any given number of arms with prescribed colors). Many of these exponents have been computed \cite{MR1879816}, thanks to the relation to SLE in the scaling limit \cite{MR1851632}. 
In the subcritical regime, the crossing probability for a percolation cluster decays exponentially fast (instead of polynomially), marking a sharp phase transition \cite{men86, MR874906}.

In contrast, in our setting of the Brownian loop soup, the crossing probability decays like a power of $|\log\eps|$, for all intensities $\theta\le1/2$. This is closer to the setting of a system depending on two parameters (e.g. temperature and pressure) that is critical on a one-dimensional curve. The random cluster model and long range percolation on $\Z^d$ are two such examples. The latter has the interesting feature that there are two points of non-analyticity on the critical line for the one-arm exponent (see the introduction of \cite{hutchcroft2022sharp} for background on critical long range percolation). This shares some similarities with the Brownian loop soup: we show in Section \ref{S:intro_2values} that $\theta = 1/2$ and $\theta=1$ can be seen as two critical intensities for the percolation behaviour of the Brownian loop soup.


\medskip

We now discuss a consequence of Theorem \ref{T:large_crossing} on a characterisation of the sets that can be hit by a cluster of the loop soup. Recall that a set $A$ is said to be polar for planar Brownian motion $(B_t)_{t \geq 0}$ if for all starting point $x$,
$
\PROB{x}{\exists t>0: B_t \in A} = 0.
$
Since there are countably many loops in a Brownian loop soup, if a set $A$ is polar for 2D Brownian motion, then a.s. it will not be visited by any cluster. However, we are going to see that the \emph{closures} of the clusters can hit much thinner sets.
Identifying the polar sets of a random set provides a way of understanding how large and spread out that set is.

\begin{definition}
Let $D \subset \C$ be a bounded open domain and let $A \subset D$ be a Borel set. We will say that $A$ is polar for the clusters of $\Lc_D^\theta$ if
\[
\Prob{ \exists~ \Cc \text{~cluster~of~} \Lc_D^\theta: \overline{\Cc} \cap A \neq \varnothing } = 0.
\]
\end{definition}

For any measurable function $K : \C \times \C \to [0,\infty]$, let
\begin{equation}
\Cap_{K}(A) := \left( \inf \left\{ \int K(x,y) \mu(\d x) \mu(\d y): \mu \text{~probability~measure~on~} A \right\} \right)^{-1}
\end{equation}
be the capacity of $A$ with respect to the kernel $K$. These capacities can be thought of as a tool measuring precisely the size of set. For instance, by Frostman's lemma, the Hausdorff dimension of $A$ is the supremum of $\alpha \geq 0$ such that $\Cap_K(A) >0$ where $K(x,y) = |x-y|^{-\alpha}$. Here we will be interested in measuring the size of very thin sets and we will choose kernels $K$ of the form $K(x,y) = |\log |x-y||^\alpha$ for $\alpha >0$. We will denote $\Cap_{\log^\alpha}$ the capacity associated to that kernel.

\begin{theorem}\label{T:polar}
Let $\theta \in (0,1/2]$, $D \subset \C$ be a bounded open domain and let $A \subset D$ be a closed set.
\begin{itemize}
\item
If $A$ is polar for the clusters of $\Lc_D^\theta$, then $\Cap_{\log^\alpha}(A) = 0$ for all $\alpha > 1-\theta$.
\item
Conversely, if $A$ is not polar for the clusters of $\Lc_D^\theta$, then $\Cap_{\log^\alpha}(A) > 0$ for all $\alpha < 1-\theta$.
\end{itemize}
\end{theorem}

Recall that Kakutani's theorem states that a closed set $A$ is polar for planar Brownian motion if and only if $\Cap_{\log^\alpha}(A) = 0$ for $\alpha = 1$ (see e.g. \cite[Theorem 8.20]{morters2010brownian}). Informally, this well-known result is the limiting case $\theta \to 0$ of Theorem \ref{T:polar} above.

With this notion of polar sets, Theorem \ref{T:large_crossing} immediately implies a more general statement:

\begin{corollary}\label{C:generalisation}
If $D$ is a bounded open domain, $x \in D$ and $A$ is a closed set which does not contain $x$ and which is not polar for the clusters of $\Lc_D^\theta$, then
\[
\P \Big( \exists ~\Cc \text{~cluster~of~} \Lc_D^\theta: \overline{\Cc} \cap A \neq \varnothing \text{~and~} \Cc \cap D(x,r) \neq \varnothing \Big) = |\log r|^{-1+ \theta +o(1)}
\quad \quad \text{as~} \quad r \to 0.
\]
\end{corollary}

\paragraph*{Related works}
Let us mention that the clusters in loop soups have been also studied in dimensions $d\geq 3$, both on discrete lattices (random walk loop soups) and on metric graphs. 
Particularly investigated was the case $\theta=1/2$ due to the relation to the GFF
\cite{MR4112719, MR4557402, ChangDuLi23crit, CaiDing23HighD}.
See also the conjectures of Wendelin Werner about the scaling limits in continuum
for $\theta=1/2$ in dimensions $d\geq 3$ \cite{MR4237293}.
For studies of clusters of random walk loop soups for more general $\theta$,
we refer to \cite{MR3263044, MR3477785, MR3692311, vogel2020geometric}.
In dimensions $d\geq 3$ the one-arm probabilities of clusters decay as a power rather than as a power of $\log$ as in $d=2$.
More precisely, it can be shown that for every $\theta\in (0,1/2]$, the probability of a one-arm event is bounded between two powers, both on discrete lattice and on metric graph. 
At the metric graph level, 
the precise value of the exponent is known for $\theta=1/2$ in $d=3$ 
\cite{MR4112719, MR4557402} (this is a result on the metric graph GFF),
and in high dimensions $d>6$ \cite{MR4237293, CaiDing23HighD}.
At the discrete lattice level, the value of the exponent is known
in $d\geq 5$ for $\theta$ small enough
\cite{MR3477785}.

\subsection{Outline of the proof and further results}\label{S:intro_outline}
We now explain our approach to Theorem \ref{T:large_crossing}. Let us recall that we cannot rely on CLE tools since the event we consider is not a CLE-type event, meaning that this is not an event concerning the outermost boundaries of the clusters.
Obtaining a lower bound on the crossing probability should be in principle the easy direction. Indeed, this amounts to finding a good strategy to make the crossing and to be able to control the probability of this strategy. One such strategy could for instance consist in crossing successive annuli $r_{i+1} \partial \D \setminus r_{i} \D$ with one loop at a time ($r_0 = r < r_1 < \dots < r_n = e^{-1}$) and welding these successive crossings to ensure that a cluster has crossed the annuli $e^{-1} \D \setminus r \D$. Summing over all possible intermediate scales $r_1, \dots, r_n$, this should lead to a close-to-optimal strategy. Since the probability that there is a loop crossing a given annulus can be accurately estimated (see Lemma \ref{L:loopmeasure_annulus_quantitative}), one can lower bound the probability of this strategy (together with FKG inequality).
Although this computation can be explicitly written down, the analysis of the lengthy expressions one obtains when doing so seems far from simple.

More importantly, the above line of argument would not provide any upper bound on the crossing probability. This direction is indeed much more challenging since we need to control the probability of \emph{all} possible scenarios where the crossing event occurs. 
We therefore need to proceed differently. In our approach, we will obtain the upper and lower bounds (almost) at the same time.
Our main intermediate step towards proving Theorem \ref{T:large_crossing} is to prove a ``scaling limit'' result for the crossing probabilities, namely to prove \eqref{eq:convergence_crossing}.

Now, let's assume that the convergence \eqref{eq:convergence_crossing} has been established, without knowing yet that $f_\infty$ takes the simple form \eqref{eq:f_infty}, and let's see how we pursue.
A standard argument shows that $f_\infty$ is supermultiplicative: for all $s, t \geq 1$, $f_\infty(s) f_\infty(t) \leq f_\infty(st)$. See Lemma \ref{L:submultiplicativity} and Corollary \ref{C:submultiplicative}. In particular, Fekete's subadditive lemma implies the existence of a one-arm exponent $\xi_1 \geq 0$ such that
\begin{equation}
f_\infty(s) = s^{-\xi_1 + o(1)}
\quad \quad \text{as} \quad s \to \infty.
\end{equation}
We show in Section \ref{S:large_crossing} that this exponent indeed describes the decay of the probability of a macroscopic crossing, i.e. that
\[
\P \Big( r \partial \D \overset{\Lc_\D^\theta}{\longleftrightarrow} e^{-1} \partial \D \Big) = |\log r|^{- \xi_1 +o(1)}
\quad \quad \text{as~} \quad r \to 0.
\]
It then remains to compute the exponent $\xi_1$. In planar percolation, this step is usually done using SLE computations. Here, we exploit some of the exact solvability inherent to the Brownian loop soup by showing that $f_\infty$ satisfies some integral equation:

\begin{theorem}\label{T:integral}
The limiting function $f_\infty$ satisfies
\begin{equation}
\label{E:T_fixed_point}
f_\infty(s) = 1 - \Big( 1 - \frac{1}{s} \Big)^\theta + \theta (s-1)^\theta \int_1^\infty (s+t-1)^{-\theta-1} f_\infty(t) \d t,
\quad \quad s \geq 1.
\end{equation}
Moreover, for all $\alpha \in (0,1-\theta)$, $f_\infty$ is the only function belonging to $\Fc_\alpha$ \eqref{E:spaceF_alpha} satisfying \eqref{E:T_fixed_point}.
\end{theorem}

This integral equation is obtained by exploring the cluster from the outside and checking how deep inside a first layer of loops can go. The first term $1 - ( 1 - 1/s )^\theta$ corresponds to the case where the crossing of $\delta \D \setminus \delta^s \D$ has been made entirely by a single loop (which is unlikely when $s$ is large). The second term with the integral corresponds to the mixed case where the first layer of loops crosses exactly $\delta \D \setminus \delta^{1+(s-1)/t} \D$. The term $f_\infty(t)$ in the integral accounts for the remaining crossing of $\delta^{1+(s-1)/t} \D \setminus \delta^s \D$ by a cluster of loops staying in $\delta \D$.
This reasoning is made precise in Section \ref{S:crossing_integral}.

Theorem \ref{T:integral} provides a uniqueness result for the subsequential limits of the crossing probabilities appearing in Theorem \ref{T:intro_convergence_crossing}. These two theorems will therefore be proved at the same time.

\paragraph*{Identification of $f_\infty$ and connection to Bessel processes}

The Brownian loop soup can be also defined on the half line $(0,\infty)$ (or on more general intervals). In this one-dimensional setting, the total local time of the loop soup is well defined pointwise and finite a.s. For any intensity $\theta$, the local time profile has the law of a squared Bessel process $(R_t^{(\theta)})_{t \geq 0}$ of dimension $2\theta$, reflected at the origin, with $R_0^{(\theta)} = 0$ a.s. (the point 0 is the boundary of $(0,\infty)$). See \cite{Lupu18} for more details and precise statements. 
In dimension 1, the probability that a cluster of loops connects two points is simply equal to the probability that the local time stays positive on the interval joining these two points. The one-dimensional analogue of the function $f_\infty$ is therefore given by
\[
\lim_{\delta \to 0} \Prob{ \forall t \in [|\log \delta|, s |\log \delta|], R_t^{(\theta)} > 0} = \Prob{ \forall t \in [1,s], R_t^{(\theta)} >0 }, \quad \quad s \geq 1,
\]
by Brownian scaling. With Theorem \ref{T:integral} in hand, we can actually show that the function $f_\infty$ agrees with the above function:

\begin{theorem}\label{T:Bessel}
For all $s \geq 1$,
\begin{align*}
f_\infty(s) & = \P \big( \forall t \in [1,s], R_t^{(\theta)} >0 \big).
\end{align*}
In particular,
\begin{equation}
\label{E:f_infty_Bessel}
f_\infty(s)
= \frac{\sin(\pi \theta)}{\pi} \int_{s-1}^\infty t^{\theta-1} (t+1)^{-1} \d t,
\quad \quad s \geq 1.
\end{equation}
\end{theorem}

Theorem \ref{T:Bessel} in particular shows that $\xi_1 = 1-\theta$ concluding the proof of Theorem \ref{T:large_crossing}. Notice that $1-\theta$ is also the exponent describing the decay of $\P \big( \forall t \in [1,s], R_t^{(\theta)} >0 \big)$ as $s \to \infty$.

In several instances, there is a strong relationship between a 2D random conformally invariant object and its one-dimensional analogue. Particularly relevant to us, the circle average of the 2D Gaussian free field with fixed centre, seen as a function of the radius of the circle, has the law of a 1D Brownian motion (which is the GFF in dimension 1). See \cite[Section 2.8]{MR2322706} for a precise statement. Another important example concerns the local time of concentric circles of planar Brownian motion. As the radius of the circle becomes small, these local times become asymptotically like the local time time of 1D Brownian motion (which turns out to be a Bessel process). See \cite{jegoBMC, jegoCritical}.

We believe that such a relationship between the 1D and 2D local times of the Brownian loop soup should also exist. In a planar Brownian loop soup, the local time of a circle is infinite, so this relation is more subtle than in the case of a single Brownian trajectory.
The usual way to circumvent this issue is to define the local time via an approximation procedure where one subtracts a deterministic diverging quantity (Wick normalisation). However, this resulting local time is signed, i.e. takes negative values.
An alternative possibility is to restrict ourselves to loops that are not too small (specifically, when looking at the local time of a circle of radius $\eps$, we only keep loops with duration at least $\eps^{O(1)}$). As we show in this article (see Section \ref{S:intro_2values}), removing these small loops does not change the percolative properties of the loop soup when $\theta \in (0,1/2]$. We believe that with a cutoff of this type, the local time of concentric circles becomes asymptotically like the local time of a one-dimensional Brownian loop soup. Theorem \ref{T:Bessel} is a strong indication that such a relationship should hold. We now give some further heuristics in this direction.

\paragraph*{Relation between 1D and 2D loop soups}

The punctured unit disc $\D \setminus \{0\}$ and the half infinite cylinder $\cyl = (0,+\infty) \times \mathbb{S}^1$ are conformally equivalent via the map $r e^{i\theta} \mapsto (|\log r|, e^{i \theta})$. Because the Brownian loop soup is conformally invariant (up to time reparametrization, see \cite[Proposition 5.27]{LawlerConformallyInvariantProcesses}) and because $\{0\}$ is polar for Brownian motion, we may work on $\cyl$ instead of $\D$.
The loop measure in $\cyl$ is then given by
\[
\loopmeasure_\cyl(\d \wp) = \int_0^\infty \d t \int_{\cyl} \d z~ 
\dfrac{p_{\cyl}(t,z,z)}{t} \P_\cyl^{t,z,z}(\d \wp)
\]
where $p_{\cyl}$ and $\P_\cyl^{t,z,z}$ denote the heat kernel in $\cyl$ and the Brownian bridge probability measure of duration $t$ from $z$ to $z$ in $\cyl$.
We will denote by $\boldsymbol{\pi} : (x,e^{i\theta}) \in \cyl \mapsto x \in (0,+\infty)$ the longitudinal projection that induces a projection on paths in $\cyl$ onto paths in $(0,+\infty)$. 
In the cylinder, the two components of a Brownian bridge of duration $t$ from $(x, e^{i\theta})$ to $(x, e^{i\theta})$ are simply two independent Brownian bridges of duration $t$: one from $x$ to $x$ in $\R^+$ and the other one from $e^{i\theta}$ to $e^{i\theta}$ in $\mathbb{S}^1$. Similarly, the heat kernel $p_\cyl$ factorises: for $z = (x, e^{i\theta})$, $p_{\cyl}(t,z,z) = p_{\mathbb{S}^1}(t,e^{i\theta},e^{i\theta}) p_{\R^+}(t,x,x)$.
The push forward $\boldsymbol{\pi}_* \loopmeasure_\cyl$ of $\loopmeasure_\cyl$ by $\boldsymbol{\pi}$ is therefore equal to
\begin{equation}
\label{E:1D_2D}
\boldsymbol{\pi}_* \loopmeasure_\cyl(\d \wp^{\rm 1D}) = \int_0^\infty \d t \int_0^\infty \d x \left\{ 2\pi p_{\mathbb{S}^1}(t,1,1) \right\} \frac{p_{\R^+}(t,x,x)}{t} \P_{\R^+}^{t,x,x}( \d \wp^{\rm 1D} ),
\end{equation}
with the factor $2\pi$ coming from the integration over $\mathbb{S}^1$.
The upshot is that, since $2 \pi p_{\mathbb{S}^1}(t,1,1) \to 1$ as $t \to \infty$, if we restrict ourselves to long loops, then we approximately recover the Brownian loop measure in $(0,+\infty)$:
\[
\loopmeasure_{\R^+}(\d \wp^{\rm 1D}) = \int_0^\infty \d t \int_0^\infty \d x~ \frac{p_{\R^+}(t,x,x)}{t} \P_{\R^+}^{t,x,x}(\d \wp^{\rm 1D}).
\]
In particular, heuristically, the long crossings on $\cyl$, 
as in Theorems \ref{T:large_crossing} and \ref{T:intro_convergence_crossing},
are achieved by combining a small number of large loops rather than a large number of small loops. 
So essentially, we are in a regime where the 2D Brownian loop soup on $\cyl$
is well approximated by the 1D Brownian loop soup on $(0,+\infty)$.
Moreover, if two large Brownian loops $\wp_1$ and $\wp_2$ on $\cyl$
are such that their 1D projections under $\boldsymbol{\pi}$ intersect, 
then with high probability $\wp_{1}$ and $\wp_{2}$ themselves do intersect, 
the winding around the cylinder creating many opportunities for that.

\paragraph*{Random coverings of $\R^+$ and regenerative sets}

By forgetting the non trivial local time that the 1D Brownian loops carry, we can view the 1D Brownian loop soup as a random collection of intervals.
For any one-dimensional loop $\wp$, let $\boldsymbol{\tilde{\pi}}(\wp) = (\min \wp, \max \wp)$. By rooting the loops at their minimum (in a similar manner as in \eqref{E:decomposition_loopmeasure}; see \cite[Proposition 3.18]{Lupu18}), one can show that the push forward $\boldsymbol{\tilde{\pi}}_* \loopmeasure_{\R^+}$ of $\loopmeasure_{\R^+}$ by $\boldsymbol{\tilde{\pi}}$ is equal to
\[
\boldsymbol{\tilde{\pi}}_* \loopmeasure_{\R^+} (\d a, \d b) = (b-a)^{-2} \d a ~\d b \indic{0<a<b}.
\]
In other words, one can sample a Poisson point process $\{ (a_i,b_i), i \geq 1 \}$ on $(\R^+)^2$ with intensity measure $\theta (b-a)^{-2} \d a ~\d b \indic{0<a<b}$ and look at the random set $\bigcup_{i \geq 1} [a_i, b_i]$ which has the same law as the points visited by a loop in $\Lc_{\R^+}^\theta$.

This type of random coverings of $\R^+$ has been greatly studied; see e.g. \cite[Chapter 7]{Bertoin1999} for an introduction to this topic. One could consider intensity measures with more general forms $\d a ~ \d \nu(b-a)$ where $\nu$ is a measure on $\R^+$. The set of uncovered points of $\R^+$ is then a regenerative set and one can characterise whether this set is trivial or not in terms of $\nu$. In the concrete example $\nu(\d x) = \theta x^{-2} \d x$ that is of interest to us, it is known that the uncovered set is $\{0 \}$ if and only if $\theta \geq 1$ (which is consistent with the Bessel process approach). Let us however stress once more that the connectivity properties on the cylinder $\cyl$ cannot be immediately deduced from those in $\R^+$ (they are for instance drastically different if $\theta \in (1/2,1)$).


Theorem \ref{T:intro_convergence_crossing} can also be interpreted through the lens of regenerative sets. To this end, let us denote by
\[
g_{\delta,s} := \inf \{ g \in (0,s): \delta^g \overset{\Lc_\D^\theta}{\longleftrightarrow} \delta^s \}, \quad \delta \in (0,1), \quad s > 1.
\]
Performing a change of variable in the integral in \eqref{E:f_infty_Bessel}, Theorem \ref{T:intro_convergence_crossing} can be rephrased as saying that $g_{\delta,s} / s$ converges in distribution as $\delta \to 0$ to the generalised arcsine law with parameter $1-\theta$:
\begin{equation}
\label{E:arcsine_law}
\frac{\sin(\pi \theta)}{\pi} t^{-\theta} (1-t)^{\theta-1} \indic{0<t<1} \d t.
\end{equation}
This is reminiscent of the convergence of the last passage times in regenerative sets towards generalised arcsine laws. See \cite{Bertoin1999} for an introduction to this theory, especially Theorem 3.2 therein.

\subsection{\texorpdfstring{Two critical values: $\theta = 1/2$ and $\theta = 1$}{Two critical values: theta = 1/2 and theta = 1}}\label{S:intro_2values}
We finish this introduction by explaining why the exponent $1-\theta$ vanishes at $\theta = 1$ instead of vanishing at the critical point $\theta = 1/2$.
As we will see, this due to the existence of two different critical values of $\theta$ for the Brownian loop soup on the cylinder $\cyl$.
By conformal invariance of the Brownian loop soup, claims about infinite clusters in $\cyl$ can be translated to claims concerning clusters targeting a given point of a domain.

The fact that $\theta = 1$ is a critical value boils down to the following dichotomy (see Lemma \ref{L:fixed-points} and Figure \ref{fig1}):
\begin{itemize}
\item
$\theta \in (0,1)$: the functions $s \mapsto 1$ and $s \mapsto \P ( \forall t \in [1,s], R_t^{(\theta)} >0 )$ are two distinct solutions to the integral equation \eqref{E:T_fixed_point};
\item
$\theta>1$: the function $s \mapsto 1$ is the unique measurable function in $[0,1]^{[1,\infty)}$ solution to \eqref{E:T_fixed_point}.
\end{itemize}

For $\theta >0$ and $\tau >0$, let $\Lc_{\cyl,\tau}^\theta$ be the subset $\Lc_\cyl^\theta$ consisting of all loops of duration at least $\tau$.
Recall that $\tau \in (0,\infty) \mapsto 2\pi p_{\mathbb{S}^1}(\tau,1,1)$ is decreasing and converges to 1 as $\tau \to \infty$.
For $\theta \in (1/2,1)$, let
\begin{equation}
\label{E:t_theta}
\tau_\theta := \inf \{ \tau>0: 2\pi p_{\mathbb{S}^1}(\tau,1,1) \leq 1 / \theta \}.
\end{equation}

\begin{theorem}\label{T:perco_large_loops}
Let $\theta \in (0,1)$ and let $\tau \geq 0$. Assume that $\tau > \tau_\theta$ if $\theta \in (1/2,1)$.
For all $s > 1$,
\begin{equation}
\label{E:T_perco_large1}
\Prob{ \{T\} \times \mathbb{S}^1 \overset{\Lc_{\cyl,\tau}^\theta}{\longleftrightarrow} \{ s T \} \times \mathbb{S}^1 } \to \frac{\sin(\pi \theta)}{\pi} \int_{s-1}^\infty t^{\theta-1} (t+1)^{-1} \d t
\quad \quad \text{as} \quad T \to \infty.
\end{equation}
In particular, there is almost surely no infinite cluster in $\Lc_{\cyl,\tau}^\theta$.

Let $\theta \geq 1$. For any $\tau \geq 0$ and $s>1$,
\begin{equation}
\label{E:T_perco_large2}
\Prob{ \{T\} \times \mathbb{S}^1 \overset{\Lc_{\cyl,\tau}^\theta}{\longleftrightarrow} \{ s T \} \times \mathbb{S}^1 } \to 1
\quad \quad \text{as} \quad T \to \infty.
\end{equation}
\end{theorem}

To summarise,

\begin{enumerate}
\item
$\theta \in (0,1/2]$. It is known that all the clusters are bounded. 
In Theorem \ref{T:large_crossing}, we show that the crossing exponent is given by $1-\theta$.
Furthermore, if one removes the small loops (as described above), the value of the exponent is unchanged (Theorem \ref{T:perco_large_loops}). This is reminiscent of \cite{Lupu19ConvCLE}.
\item
$\theta \in (1/2,1)$. 
It is known that the loops form a unique cluster,
which is thus unbounded.
However, we show in Theorem \ref{T:perco_large_loops} that if one removes 
the small loops (at a sufficiently high threshold), 
each cluster formed by the remaining loops is bounded and the crossing exponent is again given by $1-\theta$.
\item
$\theta > 1$. Again, it is known that the loops form a unique cluster.
We further conjecture that if one removes the small loops, no matter the threshold, the remaining loops will still form one unbounded cluster, alongside infinitely many bounded clusters. Theorem \ref{T:perco_large_loops} is a partial result in this direction. 
\end{enumerate}

We do not state a conjecture for $\theta = 1$ (although Theorem \ref{T:perco_large_loops} also applies to this case).
The one-dimensional loop soup does percolate at this intensity since a two-dimensional Bessel process does not vanish. However, such a Bessel process goes very close to zero so it could be the case that the large loops in $\Lc_\cyl^{\theta=1}$ does not possess any infinite cluster but still percolates longitudinally. The situation is unclear to us.

So on the half-infinite cylinder, $\theta=1/2$ is the threshold for percolation of small loops, 
and $\theta=1$ is the threshold for percolation of large loops. Nevertheless, the appearance of the second critical point $\theta=1$ depends strongly on the domain,
since the size of the loops is not conformally invariant.
For instance on the half-plane, the two thresholds are the same and equal 
$\theta=1/2$, as can be deduced from a percolation by blocks construction similar to \cite{Lupu2014LoopsHalfPlane}.

Finally, let us mention that the intensity $\theta = 1$ was also known to be special due to its connection to the uniform spanning tree (such as the Wilson's algorithm \cite{Wilson1996GeneratingRS,PROPP1998170}). This is closely related to the spatial Markov property of the loop soup at $\theta=1$ when the loops  are viewed as oriented loops, as pointed out in \cite{Werner2016}.

\paragraph*{Acknowledgements}

The authors thank Juhan Aru, Nicolas Curien, Tom Hutchcroft, Pierre Nolin, Avelio Sepúlveda, Xin Sun and Zijie Zhuang
for discussions related to this work.
This material is based upon work supported by the National Science Foundation under Grant No. DMS-1928930 while AJ and WQ participated in a program hosted by the Mathematical Sciences Research Institute in Berkeley, California, during the Spring 2022 semester.
AJ is supported by Eccellenza grant 194648 of the Swiss National Science Foundation and is a member of NCCR SwissMAP. WQ was in CNRS and Laboratoire de Math\'ematiques d’Orsay, Universit\'e Paris-Saclay, when this work was initiated.

\section{Preliminaries}

Let $D \subset \C$ be an open set whose boundary is a finite collection of analytic curves. Let $x \in D$ and run a standard Brownian motion starting from $x$ until it reaches the boundary of $D$. Denote by $B_{\tau}$ the exit point. The law of $B_\tau$ has a density $H_D(x, \cdot)$, called Poisson kernel, with respect to the one-dimensional Hausdorff measure on $\partial D$.
The boundary Poisson kernel is defined by
\begin{equation}
\label{E:Poisson_boundary}
H_D(w,z) = \lim_{\eps \to 0} \frac{1}{\eps} H_D(w + \eps n_w, z), \quad \quad w,z \in \partial D,
\end{equation}
where $n_w$ denotes the inward normal at $w$.
When $D$ is the unit disc, $H_D$ has the following explicit expression:
\begin{equation}\label{E:Poisson_kernel_disc}
H_{\D}(x,z)
= \left\{ \begin{array}{cc}
\frac{1}{2\pi} \frac{1 - |x|^2}{|x - z|^2}, & x \in \D, z \in \partial \D, \\
\frac{1}{\pi} \frac{1}{|x - z|^2}, & x \in \partial \D, z \in \partial \D.
\end{array} \right.
\end{equation}

\subsection{FKG inequality in the loop soup}

We now recall the FKG inequality which is an import tool that we will use frequently in the paper.
A function $f$ is said to be increasing if for all $\Lc, \Lc'$ with $\Lc \subset \Lc'$, $f(\Lc) \leq f(\Lc')$.

\begin{lemma}[FKG inequality {\cite[Lemma 2.1]{Janson84}}]\label{L:FKG}
For all increasing bounded functions $f, g$,
\[
\Expect{ f(\Lc_D^\theta) g(\Lc_D^\theta) } \geq \Expect{ f(\Lc_D^\theta) } \Expect{ g(\Lc_D^\theta) }.
\]
\end{lemma}

In most places we will apply this result to indicator functions of increasing events.

\subsection{Crossing an annulus with loops}

The aim of this section is to prove the following estimates on the loop measure of some crossing event.

\begin{lemma}\label{L:loopmeasure_annulus_quantitative}
Let $0<r_1<r_2<1$. We have
\begin{equation}
\label{E:L_loopmeasure_annulus_quantitative1}
\loopmeasure_\D ( \{ \wp \mathrm{~crossing~} r_2 \D \setminus r_1 \D \} ) = \left( 1 + O(1)\frac{r_1}{r_2} \right) \log \frac{\log (1/r_1)}{\log (r_2/r_1)}.
\end{equation}
Moreover, the left hand side of \eqref{E:L_loopmeasure_annulus_quantitative1} is differentiable with respect to $r_1$ and
\begin{equation}
\label{E:L_loopmeasure_annulus_quantitative2}
\frac{\d}{\d r_1} \loopmeasure_\D ( \{ \wp \mathrm{~crossing~} r_2 \D \setminus r_1 \D \} ) = \Big( 1 + \frac{O(1)}{\log(r_2/r_1)} \Big) \frac{1}{r_1} \Big( \frac{1}{\log(r_2/r_1)} - \frac{1}{\log(1/r_1)} \Big).
\end{equation}
In \eqref{E:L_loopmeasure_annulus_quantitative1} and \eqref{E:L_loopmeasure_annulus_quantitative2}, $O(1)$ refers to a quantity uniformly bounded with respect to $r_1$ and $r_2$.
\end{lemma}

We will see in \eqref{E:loopmeasure_annulus} below that the left hand side of \eqref{E:L_loopmeasure_annulus_quantitative1} can be expressed in terms of a Poisson kernel in a disc and a boundary Poisson kernel in an annulus. Contrary to the case of a disc (see \eqref{E:Poisson_kernel_disc}), the Poisson kernel in an annulus does not have a simple explicit expression. Before proving Lemma \ref{L:loopmeasure_annulus_quantitative}, we first state a result concerning this latter Poisson kernel. This result will be useful in the proof of \eqref{E:L_loopmeasure_annulus_quantitative2}.

\begin{lemma}\label{L:Poisson_kernel}
Let $q \in (0,1)$. For all $\theta \in [0,2 \pi]$, $H_{\D \setminus q \D}(qe^{i\theta},1)$ is differentiable with respect to $q$ and
\begin{equation}
\frac{\d}{\d q} \left( q H_{\D \setminus q \D}(qe^{i\theta},1) \right) = \left( 1 + \frac{O(1)}{\log q} \right) \frac{1}{|\log q|} H_{\D \setminus q \D}(qe^{i\theta},1).
\end{equation}
\end{lemma}

We now assume that Lemma \ref{L:Poisson_kernel} holds and we prove Lemma \ref{L:loopmeasure_annulus_quantitative}. The proof of Lemma \ref{L:Poisson_kernel} is contained in Section \ref{S:Poisson}.

\begin{proof}[Proof of Lemma \ref{L:loopmeasure_annulus_quantitative}]
By rooting the loop at the unique point whose distance to the origin is maximal, we find that
\begin{equation}
\label{E:decomposition_loopmeasure}
\loopmeasure_\D = \frac{1}{\pi} \int_0^1 r \d r \int_0^{2\pi} \d \theta_2 ~\bubmeasure_{r \D} (r e^{i\theta_2}).
\end{equation}
where $\bubmeasure_{r\D}(r e^{i\theta_2})$ is the bubble measure defined in \cite{Lawler04}. This decomposition is similar to \cite[Proposition 8]{Lawler04} except that in our normalisation the loop measure is twice larger than in \cite{Lawler04} (equivalently, the critical intensity is 1/2 for us whereas it is 1 for them). Moreover, one can compute for all $r \geq r_1$ and $\theta_2 \in [0,2\pi]$,
\[
\bubmeasure_{r \D}(r e^{i \theta_2})( \{ \wp \text{~hit~} r_1 \D \} ) = \pi \int_0^{2\pi} \d \theta_1 r_1 H_{r \D \setminus r_1 \D}(r e^{i\theta_2}, r_1 e^{i \theta_1}) H_{r \D}(r_1 e^{i\theta_1}, r e^{i\theta_2}),
\]
leading to
\begin{align}
\label{E:loopmeasure_annulus}
& \loopmeasure_\D ( \{ \wp \mathrm{~crossing~} r_2 \D \setminus r_1 \D \} ) 
= \int_{r_2}^1 r \d r \int_0^{2\pi} \d \theta_2
\int_0^{2\pi} \d \theta_1 r_1 H_{r \D \setminus r_1 \D}(r e^{i\theta_2}, r_1 e^{i \theta_1}) H_{r \D}(r_1 e^{i\theta_1}, r e^{i\theta_2}).
\end{align}
Using the explicit expression \eqref{E:Poisson_kernel_disc} of the Poisson kernel in a disc, we see that for all $r \geq r_2$,
\begin{equation}
\label{E:proof_Poisson_disc1}
H_{r \D}(r_1 e^{i\theta_1},r e^{i\theta_2})
= \left( 1 + O(1) \frac{r_1}{r_2} \right) \frac{1}{2\pi r}
\end{equation}
and 
\begin{equation}
\label{E:proof_Poisson_disc2}
\frac{\d}{\d r_1} H_{r \D}(r_1 e^{i\theta_1},r e^{i\theta_2}) = - \frac{2 \cos(\theta_1 - \theta_2)}{2\pi r^2} + O(1) \frac{r_1}{r^3}.
\end{equation}
Plugging the estimate \eqref{E:proof_Poisson_disc1} in \eqref{E:loopmeasure_annulus}, we obtain that
\begin{align*}
\loopmeasure_\D ( \{ \wp \mathrm{~crossing~} r_2 \D \setminus r_1 \D \} )
= \Big( 1 + O(1)\frac{r_1}{r_2} \Big) \frac{1}{2\pi} \int_{r_2}^1 \d r \int_0^{2\pi}  \d \theta_2 \int_0^{2\pi} \d \theta_1 ~r_1 H_{r \D \setminus r_1 \D}(re^{i\theta_2},r_1 e^{i\theta_1}).
\end{align*}
Since for all $\theta_1 \in [0,2\pi]$ and $r \in [r_2,1]$ (see \eqref{E:Poisson_boundary})
\begin{align}
\label{E:proof_Poisson_disc3}
\int_{r \partial \D} H_{r \D \setminus r_1 \D}(z,r_1 e^{i\theta}) \d z
& = \lim_{\eps \to 0} \frac{1}{\eps} \int_{r \partial \D} H_{r \D \setminus r_1 \D}(z,(r_1 + \eps) e^{i\theta}) \d z \\
& = \lim_{\eps \to 0} \frac{1}{\eps} \PROB{(r_1 + \eps) e^{i\theta}}{ \tau_{r \partial \D} < \tau_{r_1 \partial \D} }
= \frac{1}{r_1 \log(r/r_1)},
\nonumber
\end{align}
we deduce that
\begin{align*}
\loopmeasure_\D ( \{ \wp \mathrm{~crossing~} r_2 \D \setminus r_1 \D \} )
= \left( 1 + O(1)\frac{r_1}{r_2} \right) \int_{r_2}^1 \frac{\d r}{r \log(r/r_1)}
= \left( 1 + O(1)\frac{r_1}{r_2} \right) \log \frac{\log(1/r_1)}{\log(r_2/r_1)}.
\end{align*}
This proves \eqref{E:L_loopmeasure_annulus_quantitative1}.

We now move on to the proof of \eqref{E:L_loopmeasure_annulus_quantitative2}.
In Lemma \ref{L:Poisson_kernel} we studied the boundary Poisson kernel in an annulus of the form $\D \setminus q \D$ which is related to the Poisson kernel we are interested in by
\[
H_{r \D \setminus r_1 \D}(r e^{i\theta_2},r_1 e^{i\theta_1}) = \frac{1}{r^2} H_{\D \setminus \frac{r_1}{r} \D}\Big(1, \frac{r_1}{r} e^{i(\theta_1-\theta_2)}\Big).
\]
By Lemma \ref{L:Poisson_kernel}, the above left hand side is therefore differentiable with respect to $r_1$ and
\[
\frac{\d}{\d r_1} \left( r_1 H_{r \D \setminus r_1 \D}(r e^{i\theta_2},r_1 e^{i\theta_1}) \right) = \left( 1 + \frac{O(1)}{\log(r/r_1)} \right) \frac{1}{\log(r/r_1)} H_{r \D \setminus r_1 \D}(r e^{i\theta_2},r_1 e^{i\theta_1}).
\]
Hence, $\loopmeasure_\D ( \{ \wp \mathrm{~crossing~} r_2 \D \setminus r_1 \D \} )$ is differentiable with respect to $r_1$ and using \eqref{E:proof_Poisson_disc1} and \eqref{E:proof_Poisson_disc2} we find that its derivative is equal to
\begin{align*}
& \int_{r_2}^1 r \d r \int_0^{2\pi}  \d \theta_2 \int_0^{2\pi} \d \theta_1 ~
\Big\{ \frac{\d}{\d r_1} \left( r_1 H_{r \D \setminus r_1 \D}(r e^{i\theta_2},r_1 e^{i\theta_1}) \right) H_{r \D}(r_1 e^{i\theta_1},r e^{i\theta_2}) \\
& ~~~~~~~~~~~~~~~~~~~~~~~~~~~~~~~~~~ +
r_1 H_{r \D \setminus r_1 \D}(r e^{i\theta_2},r_1 e^{i\theta_1}) \frac{\d }{\d r_1} H_{r \D}(r_1 e^{i\theta_1},r e^{i\theta_2})
\Big\} \\
& = \left( 1 + \frac{O(1)}{\log(r_2/r_1)} \right) \int_{r_2}^1 r \d r \int_0^{2\pi} \d \theta_2 \int_0^{2\pi} \d \theta_1 ~
\Big\{ \frac{1}{2\pi r} \frac{1}{\log(r/r_1)} + O(1) \frac{r_1}{r^2} \Big\} H_{r \D \setminus r_1 \D}(r e^{i\theta_2},r_1 e^{i\theta_1}) \\
& = \left( 1 + \frac{O(1)}{\log(r_2/r_1)} \right) \int_{r_2}^1 r \d r \int_0^{2\pi} \d \theta_2 \int_0^{2\pi} \d \theta_1 ~ \frac{1}{2\pi r} \frac{1}{\log(r/r_1)} H_{r \D \setminus r_1 \D}(r e^{i\theta_2},r_1 e^{i\theta_1}).
\end{align*}
We again use the fact that (see \eqref{E:proof_Poisson_disc3})
\[
\int_0^{2\pi} H_{r \D \setminus r_1 \D}(r e^{i\theta_2},r_1 e^{i\theta_1}) \d \theta_2
= \frac{1}{r r_1 \log(r/r_1)}
\]
to obtain that
\begin{align*}
\frac{\d}{\d r_1} \loopmeasure_\D ( \{ \wp \mathrm{~crossing~} r_2 \D \setminus r_1 \D \} )
& = \left( 1 + \frac{O(1)}{\log(r_2/r_1)} \right) \int_{r_2}^1 \frac{\d r}{r r_1 \log(r/r_1)^2} \\
& = \left( 1 + \frac{O(1)}{\log(r_2/r_1)} \right) \frac{1}{r_1} \left( \frac{1}{\log(r_2/r_1)} - \frac{1}{\log(1/r_1)} \right).
\end{align*}
This concludes the proof of \eqref{E:L_loopmeasure_annulus_quantitative2}, assuming Lemma \ref{L:Poisson_kernel}.
\end{proof}

\subsection{Boundary Poisson kernel in an annulus}\label{S:Poisson}

The goal of this section is to prove Lemma \ref{L:Poisson_kernel}. We start by stating and proving an intermediate result.

\begin{lemma}\label{lem:f_develop}
Let  $f(\theta):=H_{{\mathbb{D}}\setminus q{\mathbb{D}}} (q e^{i\theta}, 1)$.
For all $\theta\in[0,2\pi]$, we have
\begin{align}\label{eq:lem1}
f(\theta) = (1+o(1)) f(0)  \, \text{ as } \, q\to 0.
\end{align}
For all $\alpha, \theta \in[0, 2\pi]$, we have
\begin{align}\label{eq:lem2}
f (\alpha) - f(\theta) = f'(\theta) (\alpha -\theta) + \frac12 (\alpha -\theta)^2 f''(0)  \left(1+ o(1)\right)  \, \text{ as } \, q\to 0.
\end{align}
Moreover, 
\begin{align}\label{eq:lem3}
f''(0) =\frac{O(1)}{(\log q)^2} f(0)  \, \text{ as } \, q\to 0.
\end{align}
\end{lemma}

\begin{proof}
The multivalued conformal map $z\mapsto \log z/|\log q|$ sends the region ${\mathbb{D}}\setminus q {\mathbb{D}}$ to the vertical strip $V:=\{z = x+ i y: -1 \le x \le 0, y \in\R\}$. We have 
\begin{align}\label{eq:H_series}
H_{{\mathbb{D}} \setminus q {\mathbb{D}}} (1, q e^{i\theta})  =\frac{1}{q |\log q|^2} \sum_{n\in {\mathbb{Z}}} H_V (0, -1 + i (\theta+ 2n \pi)/|\log q|).
\end{align}
We can map the strip $V$ onto the upper half-plane ${\mathbb{H}}$ by the conformal map $\varphi: z\mapsto \exp(i z \pi)$, so that
\begin{align*}
|\varphi'(0)| |\varphi'( -1 + i (\theta+ 2n \pi)/|\log q|)| H_V (0, -1 + i (\theta+ 2n \pi)/|\log q|) =  &H_{\mathbb{H}} (1, -\exp((\theta+2n \pi) \pi/ \log q)),
\end{align*}
which tends to $H_{\mathbb{H}}(1, -1)$ as $q\to 0$.
Note that
\begin{align*}
|\varphi'(0)| =\pi, \quad  |\varphi'( -1 + i (\theta+ 2n \pi)/|\log q|)|= \pi \exp((\theta+2n \pi) \pi/ \log q).
\end{align*}
This implies that
\begin{align*}
H_V (0, -1 + i (\theta+ 2n \pi)/|\log q| )= \pi  \exp((\theta+2n \pi) \pi/ \log q) (1+ \exp((\theta+2n \pi) \pi/ \log q))^{-2}.
\end{align*}
Therefore
\begin{align}
f(\theta)=&\frac{\pi}{q |\log q|^2} \sum_{n\in {\mathbb{Z}}} \exp((\theta+2n \pi) \pi/ \log q)  (1+ \exp((\theta+2n \pi) \pi/ \log q))^{-2}\\
=&\frac{\pi}{q |\log q|^2} \sum_{n\in {\mathbb{Z}}} \cosh\left( \frac{(\theta+2n \pi) \pi}{2\log q}\right)^{-2}.
\end{align}
Therefore
\begin{align}
\label{eq:dif0}
&f'(\theta)  =-\frac{\pi}{q(\log q)^3} \sum_{n\in {\mathbb{Z}}} \tanh\left( \frac{(\theta+2n \pi) \pi}{2\log q}\right) \cosh\left( \frac{(\theta+2n \pi) \pi}{2\log q}\right)^{-2}, \\
&f''(\theta)  =-\frac{\pi}{q(\log q)^4} \sum_{n\in {\mathbb{Z}}} \left(1-3 \tanh\left( \frac{(\theta+2n \pi) \pi}{2\log q}\right)^2\right) \cosh\left( \frac{(\theta+2n \pi) \pi}{2\log q}\right)^{-2}
\end{align}
Using that
\begin{align*}
\exp (\theta \pi/ (2\log q) ) = 1+  o(1) \text{ as } q\to 0,
\end{align*}
we deduce that for all $\theta\in [0,2\pi]$,
\begin{align*}
f(\theta) = f( 0)  \left(1+  o(1)\right), \quad f'' (\theta) = f'' ( 0)  \left(1+  o(1)\right).
\end{align*}
This proves~\eqref{eq:lem1}.
Noting that $\tanh$ takes values in $(-1,1)$, we can also deduce that
\begin{align*}
 f''(0) =\frac{O(1)}{(\log q)^2} f(0),
\end{align*}
proving~\eqref{eq:lem3}.
Therefore for all $\alpha, \theta\in [0, 2\pi]$,
\begin{align*}
f(\alpha) -f(\theta)=&\int_\theta^\alpha f'(\beta) d\beta= \int_\theta^\alpha (f'(\beta) - f'(\theta)) d\beta + (\alpha-\theta) f'(\theta)\\
=& \int_\theta^\alpha \int_\theta^\beta f''(\gamma) d\gamma d\beta + (\alpha-\theta) f'(\theta)\\
=& \int_\theta^\alpha (\beta -\theta) d\beta  f'' ( 0)  \left(1+  \frac{O(1)}{\log q}\right) + (\alpha-\theta) f'(\theta)\\
=&\frac12(\alpha-\theta)^2 f''(0)  \left(1+  \frac{O(1)}{\log q}\right) + (\alpha-\theta) f'(\theta).
\end{align*}
This proves~\eqref{eq:lem2}, and completes the proof of the lemma.
\end{proof}

We are now ready to prove Lemma~\ref{L:Poisson_kernel}.

\begin{proof}[Proof of Lemma~\ref{L:Poisson_kernel}]
Fix $q_1 \in(0,1)$ and $q_2=q_1 -\eps$, where $\eps>0$ is small and will tend to $0$. We have
\begin{align*}
H_{{\mathbb{D}}\setminus q_2 {\mathbb{D}}} (1, q_2 e^{i\theta}) = &q_1 \int_0^{2\pi} H_{{\mathbb{D}}\setminus q_1 {\mathbb{D}}} (1, q_1 e^{i\alpha}) H_{{\mathbb{D}}\setminus q_2 {\mathbb{D}}} (q_1 e^{i\alpha}, q_2 e^{i\theta}) d \alpha.
\end{align*}
Therefore 
\begin{align}
\label{eq:H1}
\frac{1}{q_1}H_{{\mathbb{D}}\setminus q_2 {\mathbb{D}}} (1, q_2 e^{i\theta})
& =  \int_0^{2\pi} H_{{\mathbb{D}}\setminus q_1 {\mathbb{D}}} (1, q_1 e^{i\theta}) H_{{\mathbb{D}}\setminus q_2 {\mathbb{D}}} (q_1 e^{i\alpha}, q_2 e^{i\theta}) d \alpha \\
& + \int_0^{2\pi} ( H_{{\mathbb{D}}\setminus q_1 {\mathbb{D}}} (1, q_1 e^{i\alpha})  -H_{{\mathbb{D}}\setminus q_1 {\mathbb{D}}} (1, q_1 e^{i\theta}) ) H_{{\mathbb{D}}\setminus q_2 {\mathbb{D}}} (q_1 e^{i\alpha}, q_2 e^{i\theta}) d \alpha.
\nonumber
\end{align}
The first term on the right hand side of \eqref{eq:H1} is equal to
\begin{align}
\notag
& H_{{\mathbb{D}}\setminus q_1 {\mathbb{D}}} (1, q_1 e^{i\theta}) \int_0^{2\pi}  H_{{\mathbb{D}}\setminus q_2 {\mathbb{D}}} (q_1 e^{i\alpha}, q_2 e^{i\theta}) d \alpha =  H_{{\mathbb{D}}\setminus q_1 {\mathbb{D}}} (1, q_1 e^{i\theta}) \int_0^{2\pi}  H_{{\mathbb{D}}\setminus q_2 {\mathbb{D}}} (q_1, q_2 e^{i\alpha}) d \alpha \\
\label{eq:H3}
&=\frac{1}{q_2} H_{{\mathbb{D}}\setminus q_1 {\mathbb{D}}} (1, q_1 e^{i\theta}) \frac{|\log q_1|}{|\log q_2|} = \frac{1}{q_2} H_{{\mathbb{D}}\setminus q_1 {\mathbb{D}}} (1, q_1 e^{i\theta}) \left(1 + \frac{1}{q_1 \log q_1} \eps + O_\eps(\eps^2) \right).
\end{align}
Let $f(\theta):=H_{{\mathbb{D}}\setminus q_1 {\mathbb{D}}} (1, q_1 e^{i\theta})$. 
By Lemma~\ref{lem:f_develop}, the second term on the right hand side of \eqref{eq:H1} is equal to
\begin{align*}
 \int_{\theta-\pi}^{\theta+\pi} f'(\theta)(\alpha -\theta) H_{{\mathbb{D}}\setminus q_2 {\mathbb{D}}} (q_1 e^{i\alpha}, q_2 e^{i\theta}) d \alpha +\frac12  \int_{\theta-\pi}^{\theta+\pi} (\alpha-\theta)^2 f''(0) (1+o(1)) H_{{\mathbb{D}}\setminus q_2 {\mathbb{D}}} (q_1 e^{i\alpha}, q_2 e^{i\theta}) d \alpha.
\end{align*}
By symmetry, the first term in the above line is $0$, hence the second term on the right hand side of \eqref{eq:H1} is equal to
\begin{align}\label{eq:H22}
\frac12  \int_{\theta-\pi}^{\theta+\pi} (\alpha-\theta)^2 f''(0) (1+o(1)) H_{{\mathbb{D}}\setminus q_2 {\mathbb{D}}} (q_1 e^{i\alpha}, q_2 e^{i\theta}) d \alpha.
\end{align}
Dividing the term in \eqref{eq:H22} by $\eps$ and take the limit as $\eps\to 0$, we get
\begin{align}
\notag
&\frac12  f''(0) (1+o(1)) \int_{\theta-\pi}^{\theta+\pi} (\alpha-\theta)^2 H_{{\mathbb{D}}\setminus q_1 {\mathbb{D}}} (q_1 e^{i\alpha}, q_1 e^{i\theta}) d \alpha\\
\label{eq:H23}
& =\frac12 (1+o(1)) f''(0) \int_{-\pi}^\pi \beta^2  H_{{\mathbb{D}}\setminus q_1 {\mathbb{D}}} (q_1, q_1 e^{i\beta}) d \beta.
\end{align}
Note that
\begin{align*}
H_{{\mathbb{D}}\setminus q_1 {\mathbb{D}}} (q_1, q_1 e^{i\beta}) \le H_{\C\setminus q_1 {\mathbb{D}}} (q_1, q_1 e^{i\beta}) =\frac{1}{\pi q^2} \sin(\theta/2)^{-2}.
\end{align*}
Therefore, 
\begin{align*}
 \int_{-\pi}^\pi \beta^2  H_{{\mathbb{D}}\setminus q_1 {\mathbb{D}}} (q_1, q_1 e^{i\beta}) d \beta = O(1)q^{-2}.
\end{align*}
Therefore, combined with \eqref{eq:lem3},  \eqref{eq:H23} is equal to
$
 \frac{O(1)}{q_1^2 (\log q_1)^2} f(0),
$
which is, by \eqref{eq:lem1}, further equal to
\begin{align*}
 \frac{O(1)}{q_1^2 (\log q_1)^2}  H_{{\mathbb{D}}\setminus q_1 {\mathbb{D}}} (1, q_1 e^{i\theta}).
\end{align*}
Combined with \eqref{eq:H1} and \eqref{eq:H3}, we get that
\begin{align*}
\lim_{\eps\to 0} \eps^{-1}\left(q_2 H_{{\mathbb{D}}\setminus q_2 {\mathbb{D}}} (1, q_2 e^{i\theta}) - q_1 H_{{\mathbb{D}}\setminus q_1 {\mathbb{D}}} (1, q_1 e^{i\theta})\right) =
\frac{1}{\log q_1} H_{{\mathbb{D}}\setminus q_1 {\mathbb{D}}} (1, q_1 e^{i\theta}) \left(1 + \frac{O(1)}{\log q_1}\right).
\end{align*}
This completes the proof.
\end{proof}

\subsection{Surrounding a disk with loops}

We will say that a loop $\wp \in \Lc_\D^\theta$ surrounds the disc $r \D$ if $\wp$ does not intersect $r \D$ but disconnects it from $\partial \D$.

\begin{lemma}\label{L:surround}
There exists $c = c(\theta)>0$ such that for all $r \in (0,1/10)$,
\begin{equation}
\Prob{ \exists \wp \in \Lc_\D^\theta \mathrm{~surrounding~} r \D } \geq 1 - r^c.
\end{equation}
\end{lemma}

\begin{proof}
By monotonicity, it suffices to prove the lemma for $r$ of the form $r = e^{-n}$ for some integer $n \geq 1$.
By independence of the collections of loops $\{ \wp \in \Lc_\D^\theta, \wp \subset e^{-k} \D \setminus e^{-k-1} \D \}$, $k =0 \dots n-1$, and by scale invariance, the probability that no loop in $\Lc_\D^\theta$ surrounds $r \D$ is at most
\begin{align*}
\Prob{ \forall k=0 \dots n-1, \nexists \wp \in \Lc_\D^\theta \mathrm{~included~in~} e^{-k} \D \mathrm{~which~surrounds~} e^{-k-1} \D  }
= p^n
\end{align*}
where $p = \Prob{ \nexists \wp \in \Lc_\D^\theta \mathrm{~surrounding~} e^{-1} \D  }$. The fact that $p < 1$ concludes the proof.
\end{proof}

\subsection{A priori upper bound on the crossing probability of an annulus by a cluster}

We continue this preliminary section by recalling that when $\theta = 1/2$,

\begin{lemma}\label{L:theta=1/2}
When $\theta = 1/2$, the following limit exists and is nontrivial:
\[
\lim_{r \to 0} |\log r|^{1/2} \P \Big( e^{-1} \overset{\Lc_\D^{1/2}}{\longleftrightarrow} r \Big)
\in (0,\infty).
\]
\end{lemma}

\begin{proof}
First consider the event that there is a cluster $\Cc$ of
$\Lc_\D^{1/2}$ crossing from $e^{-1}\partial\D$
to $r\partial\D$ without surrounding the point $0$.
Such a cluster cannot exist if $\Lc_\D^{1/2}$ contains a Brownian loop $\wp$
contained in the annulus $e^{-1}\D\setminus (r\overline{\D})$
and disconnecting $e^{-1}\partial\D$ from $r\partial\D$.
So, by Lemma \ref{L:surround},
\[
\Prob{\exists \Cc 
\text{ cluster of }
\Lc_\D^{1/2},
e^{-1} \overset{\Cc}{\longleftrightarrow} r,~
\Cc \text{ does not surround } 0
}
\leq (e r)^{c},
\]
for some constant $c>0$ and $r$ small enough. 
So with the additional condition of not surrounding $0$,
the crossing probability decays much faster than logarithmically.
With this out of the way, let us concentrate on clusters of 
$\Lc_\D^{1/2}$ that surround $0$.

Let $(\Cc_{n})_{n\geq 1}$ be the infinite sequence of clusters of 
$\Lc_\D^{1/2}$ that surround $0$,
where the clusters are ordered from the most exterior in the nesting order
($\Cc_{n}$ surrounds $\Cc_{n+1}$).
We will denote by $\partial_{o}\Cc_{n}$ the outer boundary of
$\Cc_{n}$ and by $\partial_{i}\Cc_{n}$ the inner boundary component that surrounds $0$. 
Both $\partial_{o}\Cc_{n}$ and $\partial_{i}\Cc_{n}$ are known to be CLE$_{4}$-type
loops.
We will denote by $\Int(\partial_{o}\Cc_{n})$, resp.
$\Int(\partial_{i}\Cc_{n})$, the connected component of $0$ in
$\D\setminus \partial_{o}\Cc_{n}$, resp. $\D\setminus \partial_{i}\Cc_{n}$.
These are simply connected subdomains containing $0$.
We will denote
\[
\CR_{o}(n) = \CR(0,\Int(\partial_{o}\Cc_{n})),
\qquad
\CR_{i}(n) = \CR(0,\Int(\partial_{i}\Cc_{n})).
\]
For every $n\geq 1$, the r.v.
$\CR_{o}(n)/\CR_{i}(n)$ is independent from
$(\Cc_{1},\dots,\Cc_{n-1},\partial_{o}\Cc_{n})$.
Moreover, the family of r.v.s
$(\CR_{o}(n)/\CR_{i}(n))_{n\geq 1}$ is i.i.d.
The common law is known thanks to the relation between the clusters of 
$\Lc_\D^{1/2}$ and some local sets of a continuum GFF on $\D$
\cite{ALS2,ALS1,QianWerner19Clusters}.
We refer in particular to
\cite[Corollary 5.4]{ALS2} and \cite[Proposition 4.8]{ALS1}.
To express this law, consider $(B_{t})_{t\geq 0}$
a standard \textbf{one-dimensional} Brownian motion starting at $0$.
For $a>0$, denote
\begin{displaymath}
T_{-a} = \inf\{t\geq 0 : B_{t} = -a\}.
\end{displaymath}
Then $\log\big(\CR_{o}(n)/\CR_{i}(n)\big)$ has the same distribution as
$T_{-\pi}$.
Note that
\[
\Prob{T_{-\pi}\geq t}
=\Prob{\vert B_{t} \vert \leq \pi} \sim \sqrt{2\pi} t^{-1/2},
\quad \quad \text{as} \quad t \to \infty.
\]
We will also need the following comparison between conformal radius and the
Euclidean distance.
For every open simple connected domain $D\subset\C$, with $D\neq \C$,
and for every $z\in D$,
\begin{displaymath}
d(z,\partial D)\leq\CR(z,D)\leq 4 d(z,\partial D).
\end{displaymath}
The lower bound follows simply from the monotonicity of the
conformal radius.
The upper bound is the Koebe quarter theorem 
(Theorem 5-3 in \cite{Ahlfors2010ConfInv} and the subsequent corollary).

Denote
\[
n_{\rm cross}(r) = 
\inf\{n\geq 1 : e^{-1} \overset{\Cc_{n}}{\longleftrightarrow} r\}.
\]
If the crossing does not occur, we set $n_{\rm cross}(r) = +\infty$.
Let $A$ be the annulus
\[
A = \D\setminus(e^{-1}\overline{\D}) = 
\{z\in\C : e^{-1}<\vert z\vert <1\}.
\]
Denote
\[
\Sigma = 
\sum_{n\geq 1}\Prob{\Cc_{n}\cap A\neq\emptyset}.
\]
We claim that $\Sigma < +\infty$ and that
\begin{equation}
\label{E:pf_crossing_goal}
\Prob{n_{\rm cross}(r)<+\infty}
\sim \sqrt{2\pi}\Sigma \vert\log r\vert^{-1/2} \quad \quad \text{as} \quad r \to 0.
\end{equation}
This implies our lemma.

First, let us check that $\Sigma<+\infty$.
We have that
\[
\log(d(0,\Cc_{n-1})^{-1})
\geq
\log (\CR_{i}(n-1)^{-1}) - \log 4
\geq
\Big(\sum_{k=1}^{n-1}
\log (\CR_{o}(k)/\CR_{i}(k))\Big) - \log 4.
\]
Since the r.v.s $\log (\CR_{o}(k)/\CR_{i}(k))$ are i.i.d.
and
\[
\expect\Big[\log (\CR_{o}(k)/\CR_{i}(k))\Big]
=
\expect [T_{-\pi}]  = +\infty;
\]
the large deviation principle ensures that there is $u\in (0,1)$ such that
for every $n\geq 2$, 
\[
\Prob{
\sum_{k=1}^{n-1}
\log (\CR_{o}(k)/\CR_{i}(k))
\leq n-1 + \log 4}
\leq u^{n-1}.
\]
Thus,
\[
\Prob{d(0,\Cc_{n-1})\geq e^{-(n-1)}}\leq u^{n-1}.
\]
Further,
\begin{eqnarray*}
\Prob{\Cc_{n}\cap A\neq\emptyset}
&\leq &
\Prob{\Cc_{n}\cap A\neq\emptyset, d(0,\Cc_{n-1})< e^{-(n-1)}}
+
\Prob{d(0,\Cc_{n-1})\geq e^{-(n-1)}}
\\
&\leq & 
\Prob{\Cc_{n}\cap A\neq\emptyset, d(0,\Cc_{n-1})< e^{-(n-1)}}
+
u^{n-1}.
\end{eqnarray*}
For $n\geq 3$, on the event 
$\{\Cc_{n}\cap A\neq\emptyset, d(0,\Cc_{n-1})< e^{-(n-1)}\}$,
the loop soup $\Lc^{1/2}_{\D}$ cannot contain a Brownian loop $\wp$
with $\Range(\wp)\subset e^{-1}\D\setminus(e^{-(n-1)}\overline{\D})$
that surrounds $e^{-(n-1)}\overline{\D}$.
Indeed, such a loop would have to intersect $\Cc_{n-1}$,
thus be contained in $\Cc_{n-1}$,
thus surround $\Cc_{n}$, and thus prevent $\Cc_{n}$ from intersecting $A$,
which is a contradiction.
So, by Lemma \ref{L:surround}, there is a constant $c>0$,
such that for every $n\geq 3$,
$
\Prob{\Cc_{n}\cap A\neq\emptyset}
\leq e^{-c(n-2)} + u^{n-1}.
$
This ensures that $\Sigma<+\infty$.

We now turn to the proof of \eqref{E:pf_crossing_goal}. We will first establish the lower bound. 
Let $n\geq 1$. 
The conjunction of the following conditions is sufficient to ensure that 
$n_{\rm cross}(r) = n$:
\[
d(0,\Cc_{n-1})>r, \quad \partial_{o}\Cc_{n}\cap A\neq \emptyset \quad \text{and} \quad \CR_{o}(n)/\CR_{i}(n)\geq r^{-1}.
\]
For $n=1$ the first condition is irrelevant.
By using the independence of $\CR_{o}(n)/\CR_{i}(n)$ from
$(\Cc_{1},\dots,\Cc_{n-1},\partial_{o}\Cc_{n})$,
we get
\begin{eqnarray*}
\Prob{n_{\rm cross}(r) = n}
&\geq &\Prob{\partial_{o}\Cc_{n}\cap A\neq \emptyset, d(0,\Cc_{n-1})>r}
\Prob{\CR_{o}(n)/\CR_{i}(n)\geq r^{-1}}
\\& = &
\Prob{\partial_{o}\Cc_{n}\cap A\neq \emptyset, d(0,\Cc_{n-1})>r}
\Prob{T_{-\pi}\geq\vert\log r\vert}
\\& = &
\Prob{\Cc_{n}\cap A\neq \emptyset, d(0,\Cc_{n-1})>r}
\Prob{T_{-\pi}\geq\vert\log r\vert}.
\end{eqnarray*}
Thus,
\[
\Prob{n_{\rm cross}(r) <+\infty}
\geq \Prob{T_{-\pi}\geq\vert\log r\vert}
\sum_{n\geq 1}\Prob{\Cc_{n}\cap A\neq \emptyset, d(0,\Cc_{n-1})>r}.
\]
Further, as $r\to 0$,
\[
\Prob{T_{-\pi}\geq\vert\log r\vert}\sim
\sqrt{2\pi} \vert\log r\vert^{-1/2},
\qquad
\lim_{r\to 0}
\sum_{n\geq 1}\Prob{\Cc_{n}\cap A\neq \emptyset, d(0,\Cc_{n-1})>r}
=\Sigma,
\]
concluding the proof of the lower bound of \eqref{E:pf_crossing_goal}.

Let us now derive the upper bound of \eqref{E:pf_crossing_goal}.
Fix $v\in (0,1)$.
$\Prob{n_{\rm cross}(r) <+\infty}$ is equal to
\begin{align}
\nonumber
& \Prob{n_{\rm cross}(r) <+\infty,\CR_{i}(n_{\rm cross}(r))^{-1}\geq (4 r)^{-1}}
\\
\nonumber
&\leq 
\Prob{n_{\rm cross}(r) <+\infty,\CR_{o}(n_{\rm cross}(r))^{-1}\geq (4 r)^{-(1-v)}}
\\ \nonumber & +\Prob{n_{\rm cross}(r) <+\infty,
\CR_{o}(n_{\rm cross}(r))/\CR_{i}(n_{\rm cross}(r))\geq (4 r)^{-v}}
\\
\label{E:pf_crossing_99}
&\leq 
\Prob{n_{\rm cross}(r) <+\infty, 
d(0,\partial_{o}\Cc_{n_{\rm cross}(r)})\leq (4 r)^{1-v}}
\\
\label{E:pf_crossing_999} & +\Prob{n_{\rm cross}(r) <+\infty,
\CR_{o}(n_{\rm cross}(r))/\CR_{i}(n_{\rm cross}(r))\geq (4 r)^{-v}}
\end{align}
On the event $\{n_{\rm cross}(r) <+\infty, 
d(0,\partial_{o}\Cc_{n_{\rm cross}(r)})\leq (4 r)^{1-v}\}$,
the loop soup $\Lc^{1/2}_{\D}$ cannot contain a Brownian loop $\wp$
with
$\Range(\wp)\subset e^{-1}\D\setminus((4 r)^{1-v}\overline{\D})$
that surrounds $(4 r)^{1-v}\overline{\D}$.
Indeed, then $\wp$ would have to intersect $\partial_{o}\Cc_{n_{\rm cross}(r)}$,
which is impossible by construction.
Thus, by Lemma \ref{L:surround}, there is a constant $c>0$ such that for every
$r$ small enough, \eqref{E:pf_crossing_99} is at most $e^{c} (4 r)^{c(1-v)}$. On the other hand, \eqref{E:pf_crossing_999} is at most
\begin{align*}
& \sum_{n\geq 1}
\Prob{\partial_{o}\Cc_{n}\cap A\neq \emptyset,
\CR_{o}(n)/\CR_{i}(n)\geq (4 r)^{-v}}, 
\\
& = 
\sum_{n\geq 1}
\Prob{\partial_{o}\Cc_{n}\cap A\neq \emptyset}
\Prob{\CR_{o}(n)/\CR_{i}(n)\geq (4 r)^{-v}}
=
\Sigma
\Prob{T_{-\pi}\geq v\log ((4 r)^{-1})}. 
\end{align*}
Therefore,
\begin{displaymath}
\limsup_{r\to 0} \vert\log r\vert^{1/2}\Prob{n_{\rm cross}(r) <+\infty}
\leq \sqrt{2\pi}\Sigma v^{-1/2}.
\end{displaymath}
By letting $v\to 1$, we get the desired result.
\end{proof}

In particular, this provides a first rough upper bound on the probability of crossing for any $\theta \leq 1/2$:

\begin{corollary}[A priori upper bound]\label{C:a_priori_u_bound}
For all $\theta \leq 1/2$ and $s \geq 1$,
\begin{equation}
\label{E:L_a_priori_u_bound}
\limsup_{\eps \to 0} \P \Big( \eps \overset{\Lc_\D^\theta}{\longleftrightarrow} \eps^s \Big) \leq s^{-1/2}.
\end{equation}
\end{corollary}

\begin{proof}
By monotonicity of the left hand side of \eqref{E:L_a_priori_u_bound} w.r.t. $\theta$, we can assume that $\theta =1/2$. Let $\eta >0$ be a small parameter and let $s > 1$ (the case $s=1$ is trivial).
Conditioned on $\{ e^{-1} \overset{\Lc_\D^{\theta}}{\longleftrightarrow} \eta \eps \}$ and $\{ \eps \overset{\Lc_\D^{\theta}}{\longleftrightarrow} \eps^s \}$, to have a cluster crossing the annulus $e^{-1} \D \setminus \eps^s \D$, it is enough to have a loop included in $\eps \D$ that surrounds the disc $\eta \eps \D$. By FKG inequality and scale invariance, we obtain that
\begin{align*}
\P \Big( e^{-1} \overset{\Lc_\D^{\theta}}{\longleftrightarrow} \eps^s \Big)
\geq \P \Big( e^{-1} \overset{\Lc_\D^{\theta}}{\longleftrightarrow} \eta \eps \Big) \P \Big( \eps \overset{\Lc_\D^{\theta}}{\longleftrightarrow} \eps^s \Big) \Prob{ \exists \wp \in \Lc_\D^\theta \text{~surrounding~} \eta \D }.
\end{align*}
By Lemma \ref{L:theta=1/2}, this leads to
\[
\limsup_{\eps \to 0} \P \Big( \eps \overset{\Lc_\D^\theta}{\longleftrightarrow} \eps^s \Big)
\leq s^{-1/2} \Prob{ \exists \wp \in \Lc_\D^\theta \text{~surrounding~} \eta \D}^{-1}.
\]
As $\eta \to 0$, the probability on the right hand side converges to 1. This concludes the proof.
\end{proof}

\subsection{Generalisation of Theorem \ref{T:large_crossing}}\label{SS:generalisation}

We conclude this preliminary section by showing that Corollary \ref{C:generalisation} is a quick consequence of Theorem \ref{T:large_crossing} and FKG inequality.

\begin{proof}[Proof of Corollary \ref{C:generalisation}, assuming Theorem \ref{T:large_crossing}]
Let $R_1, R_2 >0$ be such that $D(x,R_1) \subset D \subset D(x,R_2)$ and let $d$ be the distance between $x$ and $A$. The probability that there is a cluster $\Cc$ of $\Lc_D^\theta$ such that $\overline{\Cc} \cap A \neq \varnothing$ and $\Cc \cap D(x,r) \neq \varnothing$ is at most
\begin{align*}
\P \Big( \partial D(x,r) \overset{\Lc_{D(x,R_2)}^\theta}{\longleftrightarrow} \partial D(x,d) \Big)
= \P \Big( \frac{r}{R_2} \partial \D \overset{\Lc_\D^\theta}{\longleftrightarrow} \frac{d}{R_2} \partial \D \Big)
\end{align*}
by invariance under scaling and translation. Since $A$ is closed, $d$ is positive and Theorem \ref{T:large_crossing} yields the upper bound stated in Corollary \ref{C:generalisation}. A similar argument shows that if $r_0 \in (0,R_1)$ is fixed, then
\[
\P \Big( \partial D(x,r) \overset{\Lc_D^\theta}{\longleftrightarrow} \partial D(x,r_0) \Big) = |\log r|^{-1+ \theta +o(1)}.
\]
Moreover, because $A$ is not polar, the probability that there is a cluster in $\Lc_D^\theta$ whose closure hits $A$ and that disconnects $\partial D(x,r_0/2)$ and $\partial D(x,r_0)$ is positive. On the intersection of this event and the event $\{ \partial D(x,r) \overset{\Lc_D^\theta}{\longleftrightarrow} \partial D(x,r_0) \}$, there must be a cluster whose closure joins $D(x,r)$ to $A$. By FKG inequality (see Lemma \ref{L:FKG}), the probability of the intersection of these two events is at least the product of the probabilities. This shows that
\[
\P \Big( \exists ~\Cc \text{~cluster~of~} \Lc_D^\theta: \overline{\Cc} \cap A \neq \varnothing \text{~and~} \Cc \cap D(x,r) \neq \varnothing \Big) \geq |\log r|^{-1+ \theta +o(1)}
\]
as claimed.
\end{proof}

\section{One-arm event in the Brownian loop soup}\label{S:crossing}

The goal of this section is to prove Theorem \ref{T:large_crossing}. We start by introducing a few notations and state a stronger form of Theorem \ref{T:intro_convergence_crossing}.
We will denote by
\[
p_{\eps, \delta} = \Prob{ \eps \partial \D \overset{\Lc_\D^\theta}{\longleftrightarrow}  \delta \partial \D }.
\]
Let
\begin{equation}
\label{E:setI}
I_\leq^* := \{ (s_1,s_2) \in (0,\infty)^2: s_1 \leq s_2 \}
\end{equation}
and for all $\delta>0$, let $F_\delta$ be the following function
\[
F_\delta : (s_1,s_2) \in I_\leq^* \mapsto p_{\delta^{s_1}, \delta^{s_2}} \in [0,1].
\]
We will show in this section that:

\begin{theorem}\label{T:convergence_crossing}
$F_\delta$ converges as $\delta \to 0$ to some function $F_\infty$. The convergence is uniform on each compact subset of $I_\leq^*$.
\end{theorem}

As a consequence of this convergence, the limiting function $F_\infty$ is actually only a function of the ratio $s_2/s_1$, i.e. for all $(s_1,s_2) \in I_\leq$, $F_\infty(s_1,s_2) = F_\infty(1,s_2/s_1)$. The function $f_\infty$ appearing in Theorem \ref{T:intro_convergence_crossing} is then simply given by: $f_\infty(s) = F_\infty(1,s)$ for all $s \geq 1$.

\medskip

The proofs of Theorems \ref{T:large_crossing} and \ref{T:convergence_crossing} are then divided as follows:
\begin{itemize}
\item
Section \ref{S:crossing_tightness}: 
the sequence $(F_\delta)_{\delta >0}$ is tight;
\item
Section \ref{S:crossing_integral}: any subsequential limit of $(F_\delta)_{\delta >0}$ satisfies some integral equation;
\item
Section \ref{S:fixed-point}: uniqueness of the fixed-points of this integral equation. This will finish the proof of Theorem \ref{T:convergence_crossing};
\item
Section \ref{S:Bessel}: identification of $f_\infty$ in terms of Bessel processes as stated in Theorem \ref{T:Bessel}. This will in particular nail down the decay of $f_\infty(s)$ as $s \to \infty$;
\item
Section \ref{S:large_crossing}: proof of Theorem \ref{T:large_crossing} by relating the decay of the probability of large crossings to the decay of $f_\infty$ at infinity. 
\end{itemize}

\subsection{Tightness}\label{S:crossing_tightness}

Recall the definition \eqref{E:setI} of $I_\leq^*$.

\begin{lemma}[Tightness]\label{L:crossing_tightness}
For every decreasing sequence $(\delta_n)_{n \geq 1}$ converging to zero, we can extract a subsequence $(\delta_{n_k})_{k \geq 1}$ such that $(F_{\delta_{n_k}})_{k \geq 1}$ converges uniformly on every compact subsets of $I_\leq^*$ towards a continuous function.
\end{lemma}

\begin{proof}[Proof of Lemma \ref{L:crossing_tightness}]
It is enough to prove the following uniform equicontinuity property: there exist $C,c>0$ such that for all $\delta >0$, for all $(s_1,t_1), (s_2,t_2) \in I_\leq^*$,
\begin{equation}
\label{E:proof_tightness1}
|F_\delta(s_1,t_1) - F_\delta(s_2,t_2)| \leq C \frac{|s_2-s_1|^c + |t_2-t_1|^c}{\max(s_1,s_2)^\theta} + C \frac{1}{|\log \delta|^c}.
\end{equation}
Indeed, let $(\delta_n)_{n \geq 1}$ be a decreasing sequence converging to zero. By a slight variant of the Arzel\`{a}--Ascoli theorem (see e.g. \cite[Theorem 6.2]{MR3474488}), the above equicontinuity property and the fact that $(F_{\delta_n})_{n \geq 1}$ is bounded ($0 \leq F_\delta \leq 1$ for all $\delta$) imply that for every compact subset of $I_\leq^+$, we can extract a subsequence that converges uniformly on that compact set to a continuous function. We then obtain the result stated in Lemma \ref{L:crossing_tightness} by a diagonalisation argument (so that the convergence is uniform on every compact subset simultaneously).

It remains to prove the equicontinuity property \eqref{E:proof_tightness1}. Let $\delta >0$ and $(s_1,t_1), (s_2,t_2) \in I_\leq^*$. By symmetry, we can assume that $s_1 \leq s_2$. By the triangle inequality,
\begin{equation}
\label{E:proof_tightness}
|F_\delta(s_1,t_1) - F_\delta(s_2,t_2)|
\leq |F_\delta(s_1,t_1) - F_\delta(s_1,t_2)| + |F_\delta(s_1,t_2) - F_\delta(s_2,t_2)|.
\end{equation}
Let us start by bounding the last term on the right hand side. Because $s_1 \leq s_2$, 
$F_\delta(s_1,t_2) \leq F_\delta(s_2,t_2)$. 
Obtaining an inequality in the other direction amounts to bounding from below the probability to cross $\delta^{s_1} \D \setminus \delta^{t_2} \D$ conditioned on the fact that there is a crossing of $\delta^{s_2} \D \setminus \delta^{t_2} \D$. Let $\eta>0$ be a small parameter that we fix later. Conditionally on the smaller crossing, there will be a larger crossing as soon as the following two events hold:
\[
\{ \exists \wp \in \Lc_\D^\theta \text{~crossing~} \delta^{s_1} \D \setminus \delta^{s_2 + \eta} \D \}
\quad \text{and} \quad
\{ \exists \wp \in \Lc_\D^\theta \text{~surrounding~} \delta^{s_2 + \eta} \D \text{~while~staying~in~} \delta^{s_2} \D\}.
\]
By FKG inequality we obtain that $F_\delta(s_1,t_2) / F_\delta(s_2,t_2)$ is at least the product of the probabilities of each of the two above events. By \eqref{E:L_loopmeasure_annulus_quantitative1}, there exists $C>0$ such that the probability of the first event is at least $1 - C (s_2-s_1+\eta)^\theta/s_2^\theta$. By Lemma \ref{L:surround}, the probability of the second event is at least $1- \delta^{c\eta}$ for some $c>0$, giving that
\begin{align*}
F_\delta(s_1,t_2) \geq F_\delta(s_2,t_2) ( 1- \delta^{c\eta}) \left(1- C \frac{(s_2-s_1+\eta)^\theta}{s_2^\theta} \right).
\end{align*}
We choose
\[
\eta = \left\{ \begin{array}{ll}
\frac{|\log (s_2-s_1)|}{|\log \delta|} & \text{if}~ \frac{s_2-s_1}{|\log (s_2-s_1)|} > |\log \delta|, \\
\frac{\log |\log \delta|}{|\log \delta|} & \text{if}~ \frac{s_2-s_1}{|\log (s_2-s_1)|} \leq |\log \delta|.
\end{array} \right.
\]
This choice leads to the following estimate:
\[
\abs{F_\delta(s_1,t_2) - F_\delta(s_2,t_2) } \leq
\left\{ \begin{array}{ll}
C \frac{(s_2-s_1)^c}{s_2^\theta} & \text{if}~ \frac{s_2-s_1}{|\log (s_2-s_1)|} > |\log \delta|, \\
C |\log \delta|^{-c} & \text{if}~ \frac{s_2-s_1}{|\log (s_2-s_1)|} \leq |\log \delta|,
\end{array} \right.
\]
for some $C,c>0$. In particular,
\[
\abs{F_\delta(s_1,t_2) - F_\delta(s_2,t_2) } \leq C \frac{(s_2-s_1)^c}{s_2^\theta} + C |\log \delta|^{-c}
\]
for all $s_1 \leq s_2$.
The first term on the right hand side of \eqref{E:proof_tightness} can be bounded in a similar fashion:
\[
|F_\delta(s_1,t_1) - F_\delta(s_1,t_2)| \leq C \frac{|t_2-t_1|^c}{\max(t_1,t_2)^\eta} + C |\log \delta|^{-c}.
\]
Putting things together, we obtain \eqref{E:proof_tightness1}. This concludes the proof of Lemma \ref{L:crossing_tightness}.
\end{proof}

\subsection{Integral equation}\label{S:crossing_integral}

\begin{lemma}\label{L:crossing_equation_sub}
Let $u \in (0,1)$. For any $\delta \in (0,1)$ and $s \geq 1$,
\begin{align*}
p_{\delta,\delta^s}
& = 1 - \left( \frac{s-1+u}{s} \right)^\theta + \frac{O(1)}{s \log \delta}
+ \left( u^\theta + \frac{O(1)}{\log \delta} \right) p_{\delta^u, \delta^{s-1+u}} + \left( 1 + O(1)\frac{(\log |\log \delta|)^C}{|\log \delta|^\theta} \right) \times \\
& \times \theta (1-u) (s-1+u)^\theta
\int_1^{1 + \frac{s-1}{u}} (s + (1-u)(t-1))^{-\theta-1} p_{(\delta^{s-1+u})^{1/t},\delta^{s-1+u}} \d t.
\end{align*}
The $O(1)$ appearing above stands for a quantity uniformly bounded with respect to $\delta$ and $s$ (it may depend on $u$).
\end{lemma}

\begin{proof}[Proof of Lemma \ref{L:crossing_equation_sub}]
Let $u \in (0,1)$, $s>1$ and $\delta >0$. By scaling,
\[
p_{\delta,\delta^s} = \Prob{ \delta^u \overset{\Lc_{\delta^{-1+u} \D}}{\longleftrightarrow} \delta^{s-1+u} }.
\]
Let
\begin{equation}
t_* = \min \{ t \geq 1: \exists \wp \in \Lc_{ \delta^{-1+u} \D}^\theta \text{~crossing~} \D \setminus ( \delta^{s-1+u})^{1/t} \D \}.
\end{equation}
If $t_* =1$, then it means that the crossing that we are interested in is entirely made by a single loop which visits $\D^c$. If $t_* >1$, then we need loops included in $\D$ to finish the crossing. Notice further that if $t \geq 1 + \frac{s-1}{u}$, then $(\delta^{s-1+u})^{1/t}$ is larger than $\delta^u$ and the crossing of the annulus $\delta^u \D \setminus \delta^{s-1+u} \D$ is entirely made by loops staying in $\D$. Putting these remarks together, we obtain that
\begin{align*}
p_{\delta,\delta^s}
& = \Prob{ t_* = 1} + \Prob{t_* \geq 1 + \frac{s-1}{u} } p_{\delta^u, \delta^{s-1+u}} \\
& + \int_1^{1 + \frac{s-1}{u}} \Prob{ t_* \in \d t}
\Prob{ \delta^{u} \overset{\Lc_{\delta^{-1+u} \D}^\theta}{\longleftrightarrow} \delta^{s-1+u} \Big\vert t_* = t}.
\end{align*}
We are now going to estimate precisely each term appearing in the above display. These terms are divided into two different types of probabilities: probabilities that either a \emph{loop} or a \emph{cluster} crosses some annulus. Terms of the first type can be explicitly computed thanks to Lemma \ref{L:loopmeasure_annulus_quantitative}, whereas terms of the second type will expressed in terms of $F_\delta$.
In the sequel, the $O(1)$ terms might depend on $u$. In particular, the fact that $u$ is positive will be important for the estimates we are going to write.

By scaling and by Lemma \ref{L:loopmeasure_annulus_quantitative},
\begin{align*}
& \Prob{t_* = 1} =  \Prob{ \exists \wp \in \Lc_{\delta^{-1+u} \D}^\theta \text{~crossing~} \D \setminus \delta^{s-1+u} \D}
= \Prob{ \exists \wp \in \Lc_{\D}^\theta \text{~crossing~} \delta^{1-u} \D \setminus \delta^s \D} \\
& = 1 - \exp \left( - \theta \loopmeasure_\D (\{\wp \in \Lc_{\D}^\theta \text{~crossing~} \delta^{1-u} \D \setminus \delta^s \D\}) \right)
= 1 - \left( \frac{s-1+u}{s} \right)^\theta + \frac{O(1)}{s \log \delta}.
\end{align*}
Similarly,
\begin{align*}
\Prob{t_* \geq 1 + \frac{s-1}{u} } = \Prob{ \nexists \wp \in \Lc_{\delta^{-1+u}}^\theta \text{~crossing~} \D \setminus \delta^u \D }
= u^\theta + \frac{O(1)}{\log \delta}.
\end{align*}
Finally, writing $r_t = \delta^{1-u + \frac{s-1+u}{t}}$,
\begin{align*}
& \Prob{t_* \in \d t} = \frac{\d}{\d t} \Prob{ \exists \wp \in \Lc_{\delta^{-1+u} \D}^\theta \text{~crossing~} \D \setminus (\delta^{s-1+u})^{1/t} \D} \d t \\
& = \theta \left( \frac{\d}{\d t} \loopmeasure_\D ( \{ \wp \in \Lc_{\D}^\theta \text{~crossing~} \delta^{1-u} \D \setminus r_t \D \} ) \right) e^{ - \theta \loopmeasure_\D ( \{ \wp \in \Lc_{\D}^\theta \text{~crossing~} \delta^{1-u} \D \setminus r_t \D \} ) } \d t.
\end{align*}
By Lemma \ref{L:loopmeasure_annulus_quantitative},
\begin{align*}
& \frac{\d}{\d t} \loopmeasure_\D ( \{ \wp \in \Lc_{\D}^\theta \text{~crossing~} \delta^{1-u} \D \setminus r_t \D \} ) = \left( \frac{\d r_t}{\d t} \right) \frac{\d}{\d r} \loopmeasure_\D ( \{ \wp \in \Lc_{\D}^\theta \text{~crossing~} \delta^{1-u} \D \setminus r \D \} ) \Big\vert_{r = r_t} \\
& = \left( 1 + \frac{O(1)}{\frac{s-1+u}{t} \log \delta} \right)  \left( \frac{s-1+u}{t^2} \log (\delta) r_t \right) \frac{1}{r_t} \frac{1}{\log \delta} \left( \frac{1}{\frac{s-1+u}{t}} - \frac{1}{1-u + \frac{s-1+u}{t}} \right) \\
& = \left( 1 + \frac{O(1)}{\log \delta} \right) \frac{1-u}{t(1-u)+s-1+u}.
\end{align*}
Together with \eqref{E:L_loopmeasure_annulus_quantitative1}, it gives that
\begin{align*}
& \Prob{t_* \in \d t}
= \left( 1 + \frac{O(1)}{\log \delta} \right) \theta \frac{1-u}{t(1-u)+s-1+u}
\left( \frac{(1-u)t + s-1+u}{s-1+u} \right)^{-\theta} \d t \\
& = \left( 1 + \frac{O(1)}{\log \delta} \right) \theta (1-u) (s-1+u)^\theta (s + (1-u)(t-1))^{-\theta-1} \d t.
\end{align*}
Putting things together, we have obtained that
\begin{align*}
p_{\delta,\delta^s}
& = 1 - \left( \frac{s-1+u}{s} \right)^\theta + \frac{O(1)}{s \log \delta}
+ \left( u^\theta + \frac{O(1)}{\log \delta} \right) p_{\delta^u, \delta^{s-1+u}} + \left( 1 + \frac{O(1)}{\log \delta} \right) \times \\
& \times \theta (1-u) (s-1+u)^\theta
\int_1^{1 + \frac{s-1}{u}} (s + (1-u)(t-1))^{-\theta-1} \Prob{ \delta^u \overset{\Lc_{\delta^{-1+u} \D}^\theta}{\longleftrightarrow} \delta^{s-1+u} \Big\vert t_* = t} \d t.
\end{align*}

It remains to estimate the conditional probability appearing in the integral above.
Let $t \in [1, 1 + \frac{s-1}{u}]$.
We are going to conclude the proof by showing that
\begin{equation}
\label{E:proof_1234}
\Prob{ \delta^u \overset{\Lc_{\delta^{-1+u} \D}^\theta}{\longleftrightarrow} \delta^{s-1+u} \Big\vert t_* = t}
= \left( 1 - O(1) \frac{\log |\log \delta|^C}{|\log \delta|^\theta} \right) p_{(\delta^{s-1+u})^{1/t},\delta^{s-1+u}}.
\end{equation}
The upper bound is clear. Indeed,
conditionally on $t_* = t$, to make the crossing from $\delta^{s-1+u}$ to $\delta^u$, the loops which stay in $\D$ have to cross at least the annulus $(\delta^{s-1+u})^{1/t} \D \setminus \delta^{s-1+u} \D$. This is saying that
\[
\Prob{ \delta^u \overset{\Lc_{\delta^{-1+u} \D}^\theta}{\longleftrightarrow} \delta^{s-1+u} \Big\vert t_* = t}
\leq p_{(\delta^{s-1+u})^{1/t}, \delta^{s-1+u}}.
\]
We will now provide a lower bound. The idea is to ``weld'' the crossing from $\delta^u \D$ to $(\delta^{s-1+u})^{1/t} \D$, and the crossing from $(\delta^{s-1+u})^{1/t} \D$ to $(\delta^{s-1+u}) \D$. Let $t_+$ and $t_-$ be close to $t$ and be such that $t_- < t < t_+$.
Conditionally on $t_* = t$, to create a crossing of the annulus $\delta^u \D \setminus \delta^{s-1+u} \D$, it is enough for the loops staying in $\D$ to do the following four events:
\begin{itemize}
\item
Event $E_1$: There is a cluster crossing the annulus $(\delta^{s-1+u})^{1/t} \D \setminus \delta^{s-1+u} \D$;
\item
Event $E_2$: There is a loop staying in $(\delta^{s-1+u})^{1/t} \D$ which surrounds $(\delta^{s-1+u})^{1/t_-} \D$;
\item
Event $E_3$: There is a loop staying in $(\delta^{s-1+u})^{1/t_+} \D$ which surrounds $(\delta^{s-1+u})^{1/t} \D$;
\item
Event $E_4$: There is a loop crossing the annulus $(\delta^{s-1+u})^{1/t_+} \D \setminus (\delta^{s-1+u})^{1/t_-} \D$
\end{itemize}
By FKG inequality, the probability of the intersection of these four events is at least the product of each probability. This gives
\begin{align*}
& \Prob{ \delta^{s-1+u} \overset{\Lc_{\delta^{-1+u} \D}^\theta}{\longleftrightarrow} \delta^u \Big\vert t_* = t}
\geq p_{(\delta^{s-1+u})^{1/t},\delta^{s-1+u}} \Prob{E_2} \Prob{E_3} \Prob{E_4}.
\end{align*}
Let us denote $\eps = \delta^{s-1+u}$.
By Lemma \ref{L:surround}, there exists $c>0$ such that
$
\Prob{ E_2 } \geq 1 - \eps^{c (\frac{1}{t_-} - \frac1t)}.
$
We choose $t_+$ and $t_-$ precisely so that
\[
\frac{1}{t_-} - \frac{1}{t} = \frac1t - \frac{1}{t_+} = \frac{\theta}{c} \frac{\log \log \frac1\eps}{\log \frac1\eps}.
\]
This choice leads to:
$\Prob{ E_2 } = \Prob{E_3} \geq 1 - |\log \eps|^{-\theta}.$
On the other hand, by Lemma \ref{L:loopmeasure_annulus_quantitative},
\[
\Prob{ E_4 } = 1 - \left( \frac{\frac{1}{t_-} - \frac{1}{t_+}}{ \frac{1}{t_-} } \right)^{\theta + o(1/\log \delta)} = 1 - O(1) t^{\theta} |\log \eps|^{-\theta} \log |\log \delta|^C = 1 - O(1) \frac{\log |\log \delta|^C}{|\log \delta|^\theta}.
\]
Overall, we have obtained 
\[
\Prob{ \delta^u \overset{\Lc_{\delta^{-1+u} \D}^\theta}{\longleftrightarrow} \delta^{s-1+u} \Big\vert t_* = t}
\geq \left( 1 - O(1) \frac{\log |\log \delta|^C}{|\log \delta|^\theta} \right) p_{(\delta^{s-1+u})^{1/t},\delta^{s-1+u}}.
\]
This concludes the proof of \eqref{E:proof_1234} and the proof of Lemma \ref{L:crossing_equation_sub}.
\end{proof}

\begin{corollary}\label{C:crossing}
Let $(\delta_n)_{n \geq 1}$ be a decreasing sequence converging to zero such that $(F_{\delta_n})_{n \geq 1}$ converges uniformly on every compact subsets of $I_\leq^*$ to some continuous function $F_\infty$. $F_\infty$ satisfies for all $(s_1,s_2) \in I_\leq^*$,
\begin{align}\label{E:C_crossing}
& F_\infty(s_1,s_2) =  1 - \left( 1- \frac{s_1}{s_2} \right)^\theta 
+ \theta \left( \frac{s_2}{s_1}-1 \right)^\theta \int_1^\infty \left( \frac{s_2}{s_1} + t-1 \right)^{-\theta-1} F_\infty \left( \frac{s_2-s_1}{t}, s_2-s_1 \right) \d t.
\end{align}
\end{corollary}

\begin{proof}[Proof of Corollary \ref{C:crossing}]
Because the convergence is uniform on each compact set and by Lemma \ref{L:crossing_equation_sub}, we have for all $(s_1,s_2) \in I_\leq^*$, for all $u \in (0,1]$,
\begin{align*}
F_\infty(s_1,s_2) & =  1 - \left( 1 - (1-u) \frac{s_1}{s_2} \right)^\theta + u^\theta F_\infty(us_1,s_2-(1-u)s_1) + \theta (1-u) (s_2/s_1-1+u)^\theta \times \\
& \times \int_1^{1+\frac{s_2/s_1-1}{u}} (s_2/s_1 + (1-u)(t-1))^{-\theta-1} F_\infty \left( \frac{s_2-(1-u)s_1}{t}, s_2-(1-u)s_1 \right) \d t.
\end{align*}
We will get \eqref{E:C_crossing} by sending $u \to 0$. Indeed, the term $u^\theta F_\infty(us_1,s_2-(1-u)s_1)$ goes to zero simply because $F_\infty$ is bounded by 1. The integral is handled by continuity of $F_\infty$ and dominated convergence theorem.
\end{proof}

\subsection{Fixed-points of the integral equation and Proof of Theorem \ref{T:convergence_crossing}}\label{S:fixed-point}

In this section we study the solutions of the integral equation that appeared in Corollary \ref{C:crossing}. We start by introducing the functional spaces we will be working with.
For $\alpha \in \R$, we define
\begin{equation}
\label{E:spaceF_alpha}
\Fc_\alpha := \Big\{ f: [1,\infty) \to [0,1] ~\text{measurable}, ~\norme{f}_\alpha := \sup_{s \geq 1} s^\alpha |f(s)| < \infty \Big\}
\end{equation}
and for all measurable function $f:[1,\infty) \to [0,1]$, let
\begin{equation}
\label{E:mapT}
T(f) : s \in [1,\infty) \mapsto 1 - \left( 1 - \frac1s \right)^\theta + \theta (s-1)^\theta \int_1^\infty \left( s + t-1 \right)^{-\theta-1} f(t) \d t.
\end{equation}
We now introduce the two-point versions (recall the definition \eqref{E:setI} of $I_\leq^*$): for $\alpha \in \R 0$, let
\begin{equation}
\label{E:spaceG_alpha}
\Gc_\alpha := \Big\{
F : I_\leq^* \to [0,1] \text{~measurable}, ~\norme{F}_{\Gc_\alpha} := \sup_{(s_1,s_2) \in I_\leq^*} (s_2/s_1)^\alpha  F(s_1,s_2) < \infty
\Big\}
\end{equation}
and for all measurable function $F: I_\leq^* \to [0,1]$, let
\begin{equation}
\label{E:mapS}
S (F) : (s_1,s_2) \in I_\leq^* \mapsto 1 - \Big( 1- \frac{s_1}{s_2} \Big)^\theta 
+ \theta \Big( \frac{s_2}{s_1}-1 \Big)^\theta \int_1^\infty \Big( \frac{s_2}{s_1} + t-1 \Big)^{-\theta-1} F\Big( \frac{s_2-s_1}{t}, s_2-s_1 \Big) \d t.
\end{equation}

\begin{lemma}\label{L:fixed-points}
Let $\theta >0$.
For all $\alpha \in (-\theta,1)$, $\Fc_\alpha$ is stable under $T$ and $T : \Fc_\alpha \to \Fc_\alpha$ is a Lipschitz map with Lipschitz constant $c(\alpha,\theta)$ where
\begin{equation}\label{E:L_fixed_point2}
c(\alpha,\theta) = \frac{\Gamma(1-\alpha) \Gamma(\alpha+\theta)}{\Gamma(\theta)}.
\end{equation}
In particular,
\begin{itemize}
\item
If $\theta \in (0,1)$ and $\alpha \in (0,1-\theta)$, then $c(\alpha,\theta) < 1$ and $T : \Fc_\alpha \to \Fc_\alpha$ is a contraction mapping. Therefore, for all $\alpha \in (0,1-\theta)$, there exists a unique function $f_\alpha \in \Fc_\alpha$ satisfying $f_\alpha = T(f_\alpha)$;
\item
If $\theta > 1$ and $\alpha \in (1-\theta, 0)$, then $c(\alpha,\theta) < 1$ and $s \mapsto 1$ is the unique fixed point of $T$ in $\Fc_\alpha$.
\end{itemize}
The same statements hold with the map $T$ and the spaces $\Fc_\alpha$ replaced by $S$ and $\Gc_\alpha$ respectively. For $\theta \in (0,1)$ and $\alpha \in (0,1-\theta)$, let $F_\alpha$ denote the unique fixed-point of $S$ in $\Gc_\alpha$.
For all $\theta \in (0,1)$ and $\alpha, \beta \in (0,1-\theta)$,
\begin{equation}
\label{E:L_fixed_point3}
F_\alpha(s_1,s_2) = f_\beta(s_2/s_1), \quad \quad (s_1,s_2) \in I_\leq^*.
\end{equation}
\end{lemma}

\begin{figure}
   \centering
   \begin{subfigure}{.3\columnwidth}
    \def\svgwidth{\columnwidth}
\begingroup%
  \makeatletter%
  \providecommand\color[2][]{%
    \errmessage{(Inkscape) Color is used for the text in Inkscape, but the package 'color.sty' is not loaded}%
    \renewcommand\color[2][]{}%
  }%
  \providecommand\transparent[1]{%
    \errmessage{(Inkscape) Transparency is used (non-zero) for the text in Inkscape, but the package 'transparent.sty' is not loaded}%
    \renewcommand\transparent[1]{}%
  }%
  \providecommand\rotatebox[2]{#2}%
  \newcommand*\fsize{\dimexpr\f@size pt\relax}%
  \newcommand*\lineheight[1]{\fontsize{\fsize}{#1\fsize}\selectfont}%
  \ifx\svgwidth\undefined%
    \setlength{\unitlength}{332.62002321bp}%
    \ifx\svgscale\undefined%
      \relax%
    \else%
      \setlength{\unitlength}{\unitlength * \real{\svgscale}}%
    \fi%
  \else%
    \setlength{\unitlength}{\svgwidth}%
  \fi%
  \global\let\svgwidth\undefined%
  \global\let\svgscale\undefined%
  \makeatother%
  \begin{picture}(1,0.54236057)%
    \lineheight{1}%
    \setlength\tabcolsep{0pt}%
    \put(0,0){\includegraphics[width=\unitlength,page=1]{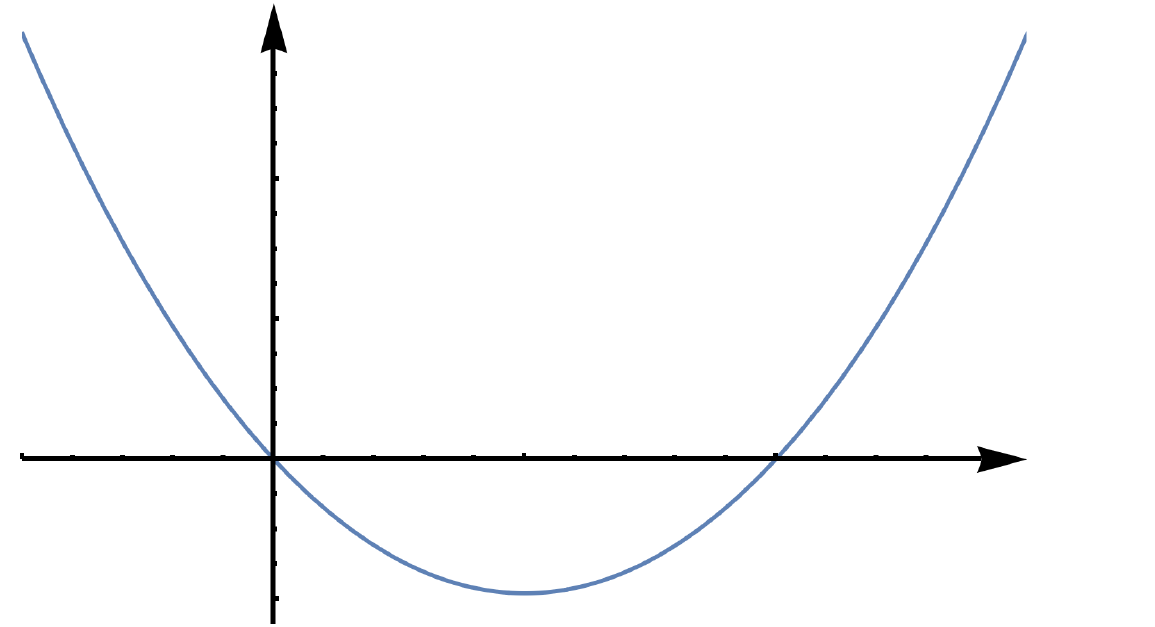}}%
    \put(0.25810658,0.5093119){\makebox(0,0)[lt]{\lineheight{1.25}\smash{\begin{tabular}[t]{l}$c(\alpha,\theta)$\end{tabular}}}}%
    \put(0.84152807,0.08833438){\makebox(0,0)[lt]{\lineheight{1.25}\smash{\begin{tabular}[t]{l}$\alpha$\end{tabular}}}}%
    \put(0.06364081,0.08425444){\makebox(0,0)[lt]{\lineheight{1.25}\smash{\begin{tabular}[t]{l}$(0,1)$\end{tabular}}}}%
    \put(0,0){\includegraphics[width=\unitlength,page=2]{theta08.pdf}}%
    \put(0.49800941,0.17799784){\makebox(0,0)[lt]{\lineheight{1.25}\smash{\begin{tabular}[t]{l}$1-\theta$\end{tabular}}}}%
  \end{picture}%
\endgroup%

   \caption{$\theta=0.8$}
   \end{subfigure}
   \begin{subfigure}{.3\columnwidth}
    \def\svgwidth{\columnwidth}
\begingroup%
  \makeatletter%
  \providecommand\color[2][]{%
    \errmessage{(Inkscape) Color is used for the text in Inkscape, but the package 'color.sty' is not loaded}%
    \renewcommand\color[2][]{}%
  }%
  \providecommand\transparent[1]{%
    \errmessage{(Inkscape) Transparency is used (non-zero) for the text in Inkscape, but the package 'transparent.sty' is not loaded}%
    \renewcommand\transparent[1]{}%
  }%
  \providecommand\rotatebox[2]{#2}%
  \newcommand*\fsize{\dimexpr\f@size pt\relax}%
  \newcommand*\lineheight[1]{\fontsize{\fsize}{#1\fsize}\selectfont}%
  \ifx\svgwidth\undefined%
    \setlength{\unitlength}{332.62002321bp}%
    \ifx\svgscale\undefined%
      \relax%
    \else%
      \setlength{\unitlength}{\unitlength * \real{\svgscale}}%
    \fi%
  \else%
    \setlength{\unitlength}{\svgwidth}%
  \fi%
  \global\let\svgwidth\undefined%
  \global\let\svgscale\undefined%
  \makeatother%
  \begin{picture}(1,0.58304148)%
    \lineheight{1}%
    \setlength\tabcolsep{0pt}%
    \put(0,0){\includegraphics[width=\unitlength,page=1]{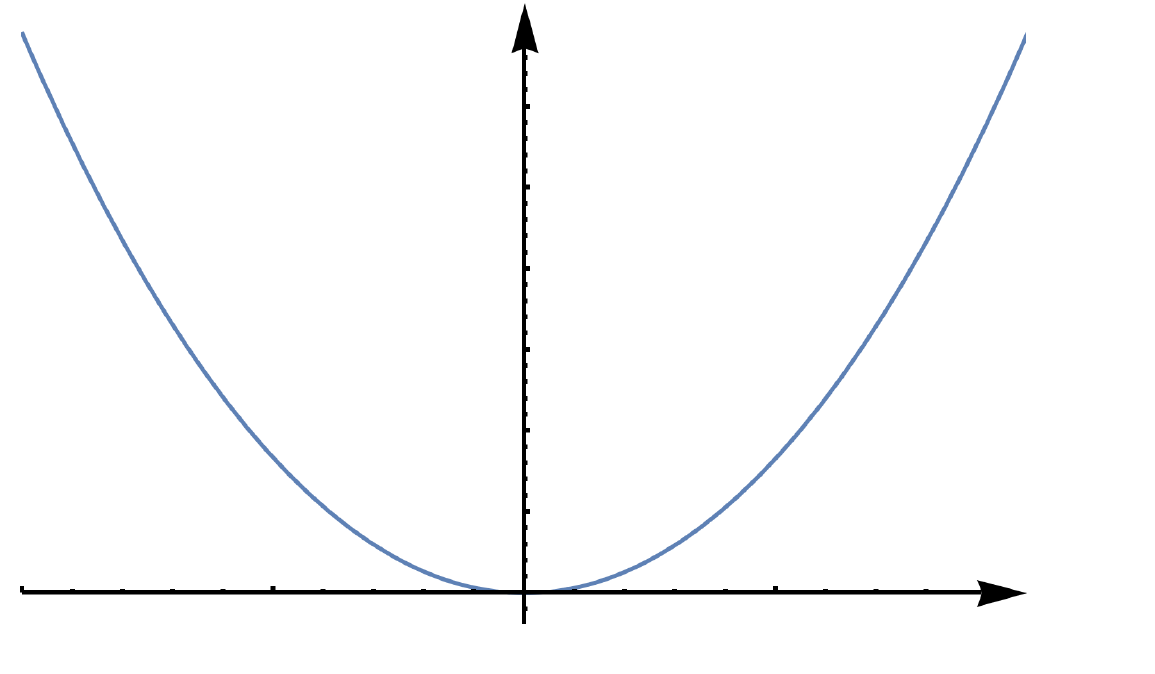}}%
    \put(0.22202941,0.54999282){\makebox(0,0)[lt]{\lineheight{1.25}\smash{\begin{tabular}[t]{l}$c(\alpha,\theta)$\end{tabular}}}}%
    \put(0.84152807,0.01627397){\makebox(0,0)[lt]{\lineheight{1.25}\smash{\begin{tabular}[t]{l}$\alpha$\end{tabular}}}}%
    \put(0.28823715,0.00725772){\makebox(0,0)[lt]{\lineheight{1.25}\smash{\begin{tabular}[t]{l}$(0,1)$\end{tabular}}}}%
  \end{picture}%
\endgroup%

   \caption{$\theta=1$}
   \end{subfigure}
   \begin{subfigure}{.3\columnwidth}
    \def\svgwidth{\columnwidth}
\begingroup%
  \makeatletter%
  \providecommand\color[2][]{%
    \errmessage{(Inkscape) Color is used for the text in Inkscape, but the package 'color.sty' is not loaded}%
    \renewcommand\color[2][]{}%
  }%
  \providecommand\transparent[1]{%
    \errmessage{(Inkscape) Transparency is used (non-zero) for the text in Inkscape, but the package 'transparent.sty' is not loaded}%
    \renewcommand\transparent[1]{}%
  }%
  \providecommand\rotatebox[2]{#2}%
  \newcommand*\fsize{\dimexpr\f@size pt\relax}%
  \newcommand*\lineheight[1]{\fontsize{\fsize}{#1\fsize}\selectfont}%
  \ifx\svgwidth\undefined%
    \setlength{\unitlength}{335.38705322bp}%
    \ifx\svgscale\undefined%
      \relax%
    \else%
      \setlength{\unitlength}{\unitlength * \real{\svgscale}}%
    \fi%
  \else%
    \setlength{\unitlength}{\svgwidth}%
  \fi%
  \global\let\svgwidth\undefined%
  \global\let\svgscale\undefined%
  \makeatother%
  \begin{picture}(1,0.53788595)%
    \lineheight{1}%
    \setlength\tabcolsep{0pt}%
    \put(0,0){\includegraphics[width=\unitlength,page=1]{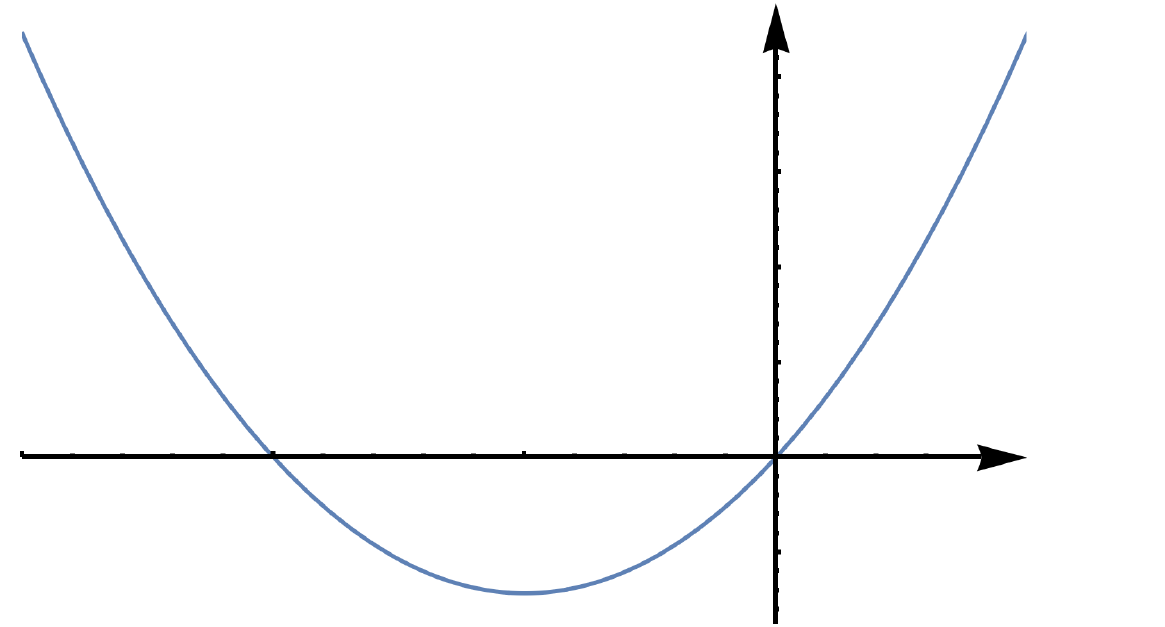}}%
    \put(0.44312518,0.50510995){\makebox(0,0)[lt]{\lineheight{1.25}\smash{\begin{tabular}[t]{l}$c(\alpha,\theta)$\end{tabular}}}}%
    \put(0.84283551,0.0876056){\makebox(0,0)[lt]{\lineheight{1.25}\smash{\begin{tabular}[t]{l}$\alpha$\end{tabular}}}}%
    \put(0.4965469,0.18464178){\makebox(0,0)[lt]{\lineheight{1.25}\smash{\begin{tabular}[t]{l}$(0,1)$\end{tabular}}}}%
    \put(0,0){\includegraphics[width=\unitlength,page=2]{theta12.pdf}}%
    \put(0.11870723,0.05210055){\makebox(0,0)[lt]{\lineheight{1.25}\smash{\begin{tabular}[t]{l}$1-\theta$\end{tabular}}}}%
  \end{picture}%
\endgroup%

   \caption{$\theta=1.2$}
   \end{subfigure}
\caption{Plot of the function $\alpha \mapsto c(\alpha,\theta)$ from Lemma \ref{L:fixed-points} for three values of $\theta$. The origin of the axes is fixed at $(0,1)$. If $\theta < 1$ (resp. $\theta > 1$), $c(\alpha,\theta)$ takes values less than 1 for some positive (resp. negative) values of $\alpha$. This reflects a phase transition at $\theta = 1$ in the uniqueness of the fixed-points of $T$ \eqref{E:mapT}, which in turn reflects a phase transition of the percolative behaviour of large loops in the loop soup (see Section \ref{S:intro_2values}).}\label{fig1}
\end{figure}

The difference of behaviour of $\alpha \mapsto c(\alpha, \theta)$ depending on whether $\theta < 1$, $\theta =1$ or $\theta >1$ is depicted in Figure \ref{fig1}.

The transition at $\theta = 1$ is another manifestation of the phase transition of the percolative behaviour of large loops in the loop soup (see Section \ref{S:intro_2values}). At criticality $\theta = 1$, the constant $c(\alpha, \theta)$ is equal to $\pi \alpha / \sin (\pi \alpha) \geq 1$ by Euler's reflection formula. Uniqueness of fixed-points cannot be inferred in that case using this approach. In Theorem \ref{T:perco_large_loops}, we will nevertheless be able to cover the case $\theta = 1$ by noticing that the unique fixed point $f_\alpha^\theta \in \Fc_\alpha$ tends to 1 as $\theta \to 1^-$.

Notice that it is not enough to work with the $L^\infty$-norm since it corresponds to the limiting case where the Lipschitz constant $c(0,\theta)$ equals 1. And indeed, when $\theta \in (0,1)$, the function $s \mapsto 1$ is another fixed point of $T$ that belongs to $\Fc_0$.

\begin{proof}
We are going to prove the statements concerning $T$ and $\Fc_\alpha$. The exact same arguments will also prove the statements concerning $S$ and $\Gc_\alpha$. The equality \eqref{E:L_fixed_point3} will then follow by uniqueness.

Let $\theta >0$ and $\alpha \in (-\theta,1)$.
We start by showing that $T(\Fc_\alpha) \subset \Fc_\alpha$. Let $f \in \Fc_\alpha$. The fact that $T(f) \geq 0$ is clear. $T(f)$ does not exceed 1 because for all $s \geq 1$,
\[
T(f)(s) \leq 1 - \left( \frac{s-1}{s} \right)^\theta + \frac{\theta}{s-1} \int_1^\infty \left( 1 + \frac{t}{s-1} \right)^{-\theta-1} \d t = 1.
\]
Concerning the $\alpha$-norm of $T(f)$, one can show that (see also below for details)
\begin{align*}
& \frac{\theta}{s-1} \int_1^\infty \left( 1 + \frac{t}{s-1} \right)^{-\theta -1} \frac{1}{t^\alpha} \d t
\sim \frac{c(\alpha,\theta)}{s^{\alpha}}
\quad \quad \text{as~} s \to \infty
\end{align*}
where $c(\alpha,\theta)$ is as in \eqref{E:L_fixed_point2}. This proves that $T$ maps $\Fc_\alpha$ to a subset of $\Fc_\alpha$.

We now study the Lipschitz property of $T$.
Let $f_1, f_2 \in \Fc_\alpha$. By definition of $T$, $\norme{\cdot}_\alpha$ and by the triangle inequality, we have
\begin{align*}
\norme{T(f_1) - T(f_2)}_\alpha & \leq \sup_{s \geq 1} s^\alpha \frac{\theta}{s-1} \int_1^\infty \left( 1 + \frac{t}{s-1} \right)^{-\theta-1} |f_1(t) - f_2(t)| \d t \\
& \leq \norme{f_1 - f_2}_\alpha \sup_{s \geq 1} s^\alpha \frac{\theta}{s-1} \int_1^\infty \left( 1 + \frac{t}{s-1} \right)^{-\theta-1} t^{-\alpha} \d t.
\end{align*}
Doing the change of variable $x = 1/t$ and then using the integral representation \eqref{E:Hypergeometric_integral} of the hypergeometric function ${}_2 F_1(a,b,c,\cdot)$ whose definition is recalled in \eqref{E:Hypergeometric}, we obtain that
\begin{align*}
s^\alpha \frac{\theta}{s-1} \int_1^\infty \left( 1 + \frac{t}{s-1} \right)^{-\theta-1} t^{-\alpha} \d t
& = \theta s^\alpha (s-1)^\theta \int_0^1 x^{\theta+\alpha-1} (1 + (s-1)x)^{-\theta-1} \d x \\
& = \frac{\theta}{\theta+\alpha} s^\alpha (s-1)^\theta {}_2F_1(\theta+1,\theta+\alpha,\theta+\alpha+1,1-s).
\end{align*}
The Pfaff transformation \eqref{E:Hypergeometric_Pfaff} then gives that
\begin{equation}
\label{E:Pfaff}
s^\alpha \frac{\theta}{s-1} \int_1^\infty \left( 1 + \frac{t}{s-1} \right)^{-\theta-1} t^{-\alpha} \d t
= \frac{\theta}{\theta+\alpha} \left( \frac{s-1}{s} \right)^\theta {}_2 F_1 \left( \alpha + \theta, \alpha, 1+\alpha+\theta, \frac{s-1}{s} \right).
\end{equation}
In particular, we see that seen as a function of $s \geq 1$, the above expression is increasing and converges to (using the fact that $\alpha <1$ and by \eqref{E:Hypergeometric_z=1})
\[
\frac{\theta}{\theta+\alpha} {}_2 F_1 \left( \alpha + \theta, \alpha, 1+\alpha+\theta, 1 \right) = \frac{\theta}{\theta+\alpha} \frac{\Gamma(1+\alpha+\theta) \Gamma(1-\alpha)}{\Gamma(1+\theta)} = c(\alpha,\theta)
\]
where $c(\alpha,\theta)$ is the constant defined in \eqref{E:L_fixed_point2}. Wrapping things up, we have proved that
\[
\norme{T(f_1) - T(f_2)}_\alpha \leq c(\alpha,\theta) \norme{f_1 - f_2}_\alpha.
\]
This concludes the proof that $T : \Fc_\alpha \to \Fc_\alpha$ is $c(\alpha,\theta)$-Lipschitz.

The existence of a unique fixed point of $T$ then follows from the Banach fixed-point theorem and from the fact that $\Fc_\alpha$ is a complete metric space (which in turn follows from the classical fact that $L^\infty$ is a complete space). This proves the lemma.
\end{proof}

We now have all the ingredients to prove Theorem \ref{T:convergence_crossing}.

\begin{proof}[Proof of Theorem \ref{T:convergence_crossing}]
Lemma \ref{L:crossing_tightness} proves tightness of $(F_\delta)_{\delta >0}$. Corollary \ref{C:crossing} proves that any subsequential limit $F_\infty$ satisfies $S(F_\infty) = F_\infty$ where we recall that $S$ is defined in \eqref{E:mapS}. Moreover, any such subsequential limit belongs to $\Gc_{1/2}$ since it can be bounded by the analogous quantity with intensity $\theta = 1/2$ for which the decay rate is known (see Corollary \ref{C:a_priori_u_bound}). Finally, Lemma \ref{L:fixed-points} shows that $F_\infty = F_\alpha$ for any $\alpha < 1-\theta$ where $F_\alpha$ is (as in Lemma \ref{L:fixed-points}) the unique fixed-point of $S$ which belongs to $\Gc_\alpha$. This concludes the proof.
\end{proof}

Let $f_\infty$ be the limiting function appearing in Theorem \ref{T:intro_convergence_crossing}.
As a direct consequence of Lemma \ref{L:fixed-points}, $f_\infty$ belongs to $\bigcap_{\alpha < 1-\theta} \Fc_\alpha$, that is to say $f_\infty$ decays faster $s^{-\alpha}$ for all $\alpha < 1-\theta$. In the next section, we will identify precisely $f_\infty$ and we will in particular obtain that $f_\infty$ decays exactly like a constant times $s^{\theta-1}$. In Appendix \ref{S:app_lower}, we give an independent short proof of the fact that $f_\infty$ does not decay much faster than $s^{-1+\theta}$ that does not rely on this explicit expression.

\subsection{\texorpdfstring{$f_\infty$ in terms of Bessel processes}{f infinity in terms of Bessel processes}}\label{S:Bessel}

In this section we prove Theorem \ref{T:Bessel}. Recall that we denote by $(R_t^{(\theta)})_{t \geq 0}$ a squared Bessel process of dimension $2\theta$, reflected at the origin, with $R_0^{(\theta)} = 0$ a.s. Denote by
\[
f_{\rm Bes} : s \in [1,\infty) \mapsto \Prob{ \forall t \in [1,s], R_t^{(\theta)} >0 }.
\]
We first recall the explicit expression of the above probability and  show that $f_{\rm Bes}$ is also a fixed-point of $T$.

\begin{lemma}\label{L:Bessel_explicit}
Let $\theta \in (0,1)$.
For all $s \geq 1$,
\begin{equation}
\label{E:L_Bessel_explicit}
f_{\rm Bes}(s)
= \frac{\sin(\pi \theta)}{\pi} \int_{s-1}^\infty t^{\theta-1} (t+1)^{-1} \d t.
\end{equation}
Moreover, $T f_{\rm Bes} = f_{\rm Bes}$.
\end{lemma}

\begin{proof}
The last hitting time of 0 of $R^{(\theta)}$ before $s$, divided by $s$, follows the generalised arcsine law \eqref{E:arcsine_law} with parameter $1-\theta$; see e.g. \cite[Proposition 3.1]{Bertoin1999}. The left hand side of \eqref{E:L_Bessel_explicit} corresponds to the probability that this last hitting time is smaller than 1. Integrating \eqref{E:arcsine_law} and then performing a change of variable concludes the proof of \eqref{E:L_Bessel_explicit}.

We now prove that $T f_{\rm Bes} = f_{\rm Bes}$.
It is likely that this can be proved directly from the explicit expression of $f_{\rm Bes}$ and certain identities between special functions. We prefer to take a different route that explains what this integral equation means probabilistically for the Bessel process.
We can decompose (see \cite{shiga1973bessel})
\[
(R_t^{(\theta)})_{t \geq 1}
\quad \overset{\mathrm{(d)}}{=} \quad
(R_{t-1}^{(\theta)})_{t \geq 1} + (R^{(0)}_{t-1})_{t \geq 1}
\]
where on the right hand side $R^{(\theta)}$ and $R^{(0)}$ are independent and are distributed as follows: $R^{(\theta)}$ is as before a squared Bessel process of dimension $2\theta$ with $R_0^{(\theta)} = 0$ a.s., and $R^{(0)}$ is a zero-dimensional squared Bessel process with $R_0^{(0)} \sim R_1^{(\theta)}$.
Elaborating on the discussion that follows Theorem \ref{T:integral}, the first term on the right hand side can be thought of as the local time of the (large) loops that stay inside $\delta \D$, whereas the second term corresponds to the local time of the loops that hit $\delta \partial \D$ (excursions).
Let $\tau := \inf \{ t>0: R^{(0)}_t = 0 \}$ and fix $s \geq 1$. Using the above decomposition, we can write
\begin{equation}
\label{E:Bessel1}
f_{\rm Bes}(s) = \Prob{ \tau \geq s-1 } + \int_0^{s-1} \Prob{ \tau \in \d u } \Prob{ \forall v \in [u,s-1], R_v^{(\theta)} >0 }.
\end{equation}
By Brownian scaling,
the probability $\Prob{ \forall v \in [u,s-1], R_v^{(\theta)} >0 }$ is equal to $f_{\rm Bes}((s-1)/u)$. To conclude that $Tf_{\rm Bes} = f_{\rm Bes}$, we need to compute the density of $\tau$. Conditionally on $R_0^{(0)} = x^2$, the density of $\tau$ is given by (see \cite[Proposition 2.9]{LawlerBessel})
\[
\Prob{ \tau \in \d u \vert R_0^{(0)} = x^2 }
= \frac12 x^2 u^{-2} e^{-x^2 / (2u) } \indic{u >0} \d u.
\]
Using the explicit density of $R_0^{(0)} \sim R_1^{(\theta)}$ (see \cite[Equation (37)]{LawlerBessel}), we obtain that
\begin{align*}
\Prob{ \tau \in \d u }
& = \frac{1}{\Gamma(\theta)} 2^{-\theta} u^{-2} \int_0^\infty \d x ~x^{2\theta + 1} e^{- \frac{u+1}{u}\frac{x^2}{2}} \indic{u >0} \d u \\
& = \frac{1}{\Gamma(\theta)} u^{-1+\theta} (u+1)^{-1-\theta} \int_0^\infty \d y ~ y^\theta e^{-y} \indic{u >0} \d u
= \theta u^{-1+\theta} (u+1)^{-1-\theta} \indic{u >0} \d u
\end{align*}
thanks to the change of variable $y = \frac{u+1}{u}\frac{x^2}{2}$.
In particular,
\[
\Prob{ \tau \geq s-1} = \theta \int_{s-1}^\infty u^{-1+\theta} (u+1)^{-1-\theta} \d u = 1 - \Big( \frac{s-1}{s} \Big)^\theta.
\]
Going back to \eqref{E:Bessel1} and then performing the change of variable $t = (s-1)/u$, we obtain that
\begin{align*}
f_{\rm Bes}(s) & = 1 - \Big( \frac{s-1}{s} \Big)^\theta
+ \theta \int_0^{s-1} u^{-1+\theta} (u+1)^{-1-\theta} f_{\rm Bes} \Big( \frac{s-1}{u} \Big) \d u \\
& = 1 - \Big( \frac{s-1}{s} \Big)^\theta + \theta (s-1)^\theta \int_1^\infty (s+t-1)^{-1-\theta} f_{\rm Bes}(t) \d t.
\end{align*}
This proves that $Tf_{\rm Bes} = f_{\rm Bes}$.
\end{proof}

Theorem \ref{T:Bessel} follows:

\begin{proof}[Proof of Theorem \ref{T:Bessel}]
By Lemma \ref{L:Bessel_explicit}, $f_{\rm Bes}$ is a fixed-point of $T$. The explicit expression \eqref{E:L_Bessel_explicit} of $f_{\rm Bes}$ also shows that it belongs to $\Fc_\alpha$ \eqref{E:spaceF_alpha} for all $\alpha < 1-\theta$. But by Lemma \ref{L:fixed-points}, $f_\infty$ is the only fixed point of $T$ that belongs to these spaces. Hence $f_\infty = f_{\rm Bes}$.
\end{proof}

\subsection{Probability of large crossings -- Proof of Theorem \ref{T:large_crossing}}\label{S:large_crossing}

This section is dedicated to the proof of Theorem \ref{T:large_crossing}. It will be divided into two parts. The first part will prove the lower bound and the second part will be focused on the upper bound. 

\paragraph{Lower bound} We start by showing some supermultiplicativity-type estimate. This estimate is a quick consequence of the FKG inequality.

\begin{lemma}\label{L:submultiplicativity}
There exists $c>0$ such that for all $\eta >0$, $\eps$ small enough and $s, t>1$,
\[
p_{\eps, \eps^{s + \eta}} ( p_{\eps^{s}, \eps^{s+t}} - \eps^{c\eta} )
\leq p_{\eps, \eps^{s + t}}.
\]
\end{lemma}

\begin{proof}
Let us denote by $E$ the event that there exists a loop $\wp \in \Lc_\D^\theta$ which surrounds $\eps^{s+\eta} \D$ while staying in $\eps^{s} \D$. By scaling and by Lemma \ref{L:surround}, there exists $c>0$ such that for all $\eps$ small enough $\Prob{E} \geq 1 - \eps^{c\eta}$.
We have
\begin{align*}
\frac{p_{\eps,\eps^{s+t}}}{p_{\eps, \eps^{s+\eta}}}
& = \Prob{ \eps \partial \D \overset{\Lc_\D^\theta}{\longleftrightarrow} \eps^{s+t} \partial \D \Big\vert \eps \partial \D \overset{\Lc_\D^\theta}{\longleftrightarrow} \eps^{s+\eta} \partial \D } \geq \Prob{ E, \eps^{s} \partial \D \overset{\Lc_\D^\theta}{\longleftrightarrow} \eps^{s+t} \partial \D \Big\vert \eps \partial \D \overset{\Lc_\D^\theta}{\longleftrightarrow} \eps^{s+\eta} \partial \D }.
\end{align*}
Now, by FKG inequality, the above probability is larger than the same one without conditioning. Therefore,
\[
\frac{p_{\eps,\eps^{s+t}}}{p_{\eps, \eps^{s+\eta}}} \geq p_{\eps^{s}, \eps^{s+t}} - \Prob{E^c}
\]
which concludes the proof.
\end{proof}

\begin{corollary}\label{C:submultiplicative}
The function $f_\infty :[1,\infty) \to (0,1]$ is supermultiplicative: for all $s, t \geq 1$,
$
f_\infty(s)f_\infty(t) \leq f_\infty(st).
$
Therefore, there exists $\xi_1 >0$ such that
$
f_\infty(s) = s^{-{\xi_1} + o(1)}
$
as $s \to \infty$.
\end{corollary}

\begin{proof}
By Lemma \ref{L:submultiplicativity}, for all $s, t \geq 1$, $t'>t$,
$
f_\infty(s) f_\infty(t') \leq f_\infty(s t).
$
We then obtain the statement by continuity of $f_\infty$.
The existence of the limit of
$
\log f_\infty(s) / \log s
$
as $s \to \infty$
is a consequence of Fekete's subadditive lemma.
\end{proof}

Note that thanks to Theorem \ref{T:Bessel} we already know that $\xi_1 = 1-\theta$. In the remaining of this section we will still write $\xi_1$ instead of $1-\theta$ to emphasise that the argument does not rely on the exact value of the exponent $\xi_1$.

We are now ready to prove the lower bound part of Theorem \ref{T:large_crossing}.

\begin{proof}[Proof of Theorem \ref{T:large_crossing} -- Lower bound]
Let $\eta >0$ be a small parameter.
Let $s>1$ be large enough so that for all $t > s$, $f_\infty(t) \geq t^{-{\xi_1} - \eta}$.
Let $c>0$ be the constant appearing in Lemma \ref{L:submultiplicativity}.
By Theorem \ref{T:convergence_crossing},
\[
p_{\eps,\eps^{s(1+\eta)}} - \eps^{c\eta} \xrightarrow[ \eps \to 0]{} f_\infty(s(1+\eta)).
\]
We can therefore pick $\eps_0$ small enough so that for all $\eps \in (0,\eps_0)$, $p_{\eps, \eps^{s(1+\eta)}} - \eps^\eta \geq (1-\eta) f_\infty(s(1+\eta)) $. Let $\eps>0$ be much smaller than $\eps_0$. Let $k$ be the largest integer so that
$\tilde{\eps}_0 := \eps^{1/s^k} \leq \eps_0$, i.e. let
\[
k = \floor{ \frac{1}{\log s} \log \frac{|\log \eps|}{|\log \eps_0|} }.
\]
We apply successively Lemma \ref{L:submultiplicativity}:
\begin{align*}
p_{\eps^{1/s^k},\eps} & \geq \left( p_{\eps^{1/s^k},\eps^{(1+\eta)/s^{k-1}}} - \eps^{\eta/s^{k-1}} \right) p_{\eps^{1/s^{k-1}},\eps} \geq  \dots \geq \prod_{i=1}^{k} \left( p_{\eps^{1/s^i},\eps^{(1+\eta)/s^{i-1}}} - \eps^{\eta/s^{i-1}} \right) \\
& \geq \left( (1-\eta) f_\infty((1+\eta)s) \right)^k
\geq \left( (1-\eta) \left\{ (1+\eta)s \right\}^{-{\xi_1} - \eta} \right)^k.
\end{align*}
The main contribution comes from
\[
s^{-{\xi_1} k } \geq \left( \frac{|\log \eps|}{|\log \eps_0|} \right)^{-{\xi_1}}
\]
and we have shown that for some $\delta = \delta(\eta)$ going to zero as $\eta \to 0$,
$
p_{\tilde{\eps}_0,\eps} \geq |\log \eps|^{-{\xi_1} - \delta}.
$
Since $\tilde{\eps}_0 \geq (\eps_0)^{s}$ is macroscopic, we obtain by gluing a few extra loops and FKG inequality that
\[
p_{e^{-1},\eps} \geq |\log \eps|^{-{\xi_1} - \delta'}
\]
where $\delta'$ is arbitrarily small. It concludes the proof.
\end{proof}

\paragraph{Upper bound}

To obtain an upper bound we need to establish a version of Lemma \ref{L:submultiplicativity} in the other direction. To achieve this, we need to find two independent collections of loops whose clusters make non trivial crossings. This is very much similar in spirit to the BK inequality in percolation and we state this reverse inequality as:

\begin{lemma}\label{L:BK_inequality}
Let $s>1$ and $k_0 \geq 1$. For any $n \geq k_0+1$ and $\eps >0$,
\begin{align}
\label{E:L_BK}
p_{\eps, \eps^{s^n}} & \leq \sum_{k=1}^{k_0} p_{\eps,\eps^{s^k}} p_{\eps^{s^{k+1}},\eps^{s^n}} + \sum_{k=k_0+1}^{n-2} \Prob{ \exists \wp \in \Lc_\D^\theta \text{~crossing~} \eps^s \D \setminus \eps^{s^k} } p_{\eps^{s^{k+1}}, \eps^{s^n}} \\
& \quad + \Prob{ \exists \wp \in \Lc_\D^\theta \text{~crossing~} \eps^s \D \setminus \eps^{s^{n-1}} \D }.
\nonumber
\end{align}
\end{lemma}

\begin{proof}
We explore the cluster of the circle $\eps \partial \D$ as follows. We start with $\Cc_0 = \eps \partial \D$ and for all $i \geq 1$, we explore the loops that touch $\Cc_{i-1}$ and we let $\Cc_i$ be the union of $\Cc_{i-1}$ and the newly revealed loops.
We stop the exploration at the first time $i_*$ that the cluster hits $\eps^{s} \partial \D$:
\[
i_* = \inf \{ i \geq 1: \Cc_i \cap \eps^{s} \partial \D \neq \varnothing \}.
\]
Let
\[
k_* = \max \{ k \geq 1: \Cc_{i_*} \cap \eps^{s^k} \partial \D \neq \varnothing \}.
\]
By definition, we have not revealed the loops that hit the disc of radius $\eps^{s^{k_*+1}}$. Therefore, for all $k = 1 \dots n-2$,
\[
\Prob{ \eps \overset{\Lc_\D^\theta}{\longleftrightarrow} \eps^{s^n} \vert k_* = k }
\leq p_{\eps^{s^{k+1}},\eps^{s^n}}.
\]
With a union bound this gives,
\begin{align*}
p_{\eps, \eps^{s^n}} & \leq 
\sum_{k=1}^{n-2} \Prob{k_* = k} p_{\eps^{s^{k+1}},\eps^{s^n}}
+ \Prob{k_* \geq n-1}.
\end{align*}
We bound $\Prob{ k_* = k} \leq \Prob{ k_* \geq k}$ and when $k \leq k_0$ we bound this probability by $p_{\eps, \eps^{s^k}}$. When $k \geq k_0 +1$, we note that $k_*$ can be larger or equal than $k$ only if there is a \emph{loop} crossing the annulus $\eps^s \D \setminus \eps^{s^k} \D$. This gives \eqref{E:L_BK}.
\end{proof}

We can now conclude this section with a proof of the upper bound in Theorem \ref{T:large_crossing}.

\begin{proof}[Proof of Theorem \ref{T:large_crossing} -- Upper bound]
Let $\delta >0$.
Let $k_0$ be a large integer and $s>1$. Let $\eps_0$ be small enough so that for all $k=1 \dots k_0$, for all $\eps \in (0,\eps_0)$,
\begin{equation}
\label{E:proof_upper_bound3}
p_{\eps, \eps^{s^k}} \leq (s^k)^{-{\xi_1} + \delta}.
\end{equation}
By Lemma \ref{L:loopmeasure_annulus_quantitative},
\[
\Prob{ \exists \wp \in \Lc_\D^\theta \text{~crossing~} \eps^{s} \D \setminus \eps^{s^{k}} \D }
= (1+o(1)) \theta s^{-(k-1)},
\]
where $o(1) \to 0$ as $\eps \to 0$, uniformly in $k \geq 2$. We can assume that $\eps_0$ has been chosen small enough so that for all $k \geq 2$, for all $\eps \in (0,\eps_0)$,
\[
\Prob{ \exists \wp \in \Lc_\D^\theta \text{~crossing~} \eps^{s} \D \setminus \eps^{s^k} \D }
\leq 2 \theta s^{-(k-1)}.
\]
We now apply Lemma \ref{L:BK_inequality} with $s$, $k_0$, $\eps_0$ (instead of $\eps$) and some integer $n \geq k_0+2$:
\[
p_{\eps_0, \eps_0^{s^n}} \leq \sum_{k=1}^{k_0} s^{-({\xi_1}-\delta) k} p_{\eps_0^{s^{k+1}},\eps_0^{s^n}} + 2\theta \sum_{k=k_0+1}^{n-2} s^{-(k-1)} p_{\eps_0^{s^{k+1}},\eps_0^{s^n}} + 2\theta s^{-(n-2)}.
\]
We assume that $k_0$ is large enough to ensure that for all $k \geq k_0+1$, $2\theta s^{-(k-1)} \leq s^{-({\xi_1} - \delta)k}$. The above inequality then implies that for all $n \geq k_0+2$,
\begin{equation}
\label{E:proof_upper_bound1}
p_{\eps_0, \eps_0^{s^n}} + s^{-({\xi_1} - \delta)n} \leq \sum_{k=1}^{n-2} s^{-({\xi_1}-\delta) k} \left( p_{\eps_0^{s^{k+1}},\eps_0^{s^n}} + s^{-({\xi_1} - \delta) (n-k-1)} \right).
\end{equation}
We now iterate this inequality until only terms of the form $p_{\eps,\eps^{s^k}}$ for some $\eps \in (0,\eps_0]$ and $k \in \{1, \dots k_0 \}$ remain. Thanks to \eqref{E:proof_upper_bound3}, we can bound these final terms by $s^{-({\xi_1} - \delta)k}$ and we obtain an estimate of the form
\begin{equation}
\label{E:proof_upper_bound2}
p_{\eps_0, \eps_0^{s^n}} + s^{-({\xi_1} - \delta)n} \leq 2 \sum \prod_i s^{-({\xi_1}-\delta) k_i} 
\end{equation}
Instead of describing precisely this summation, we make the following two key observations:
\begin{itemize}
\item
The total number of terms involved in the sum is at most $C^n$ for some $C>0$ independent of $\delta$, $s$ and $k_0$. This is due to the fact that if a sequence $(S_n)_{n \geq 1}$ of integers satisfies $S_n \leq \sum_{k=1}^{n-1} S_k$ for all $n \geq 2$, then $(S_n)$ grows at most exponentially in $n$.
\item
Each time that we use the inequality \eqref{E:proof_upper_bound1}, the sum of the ratios of the logarithms of the radii decreases by one. In \eqref{E:proof_upper_bound1}, this is saying that $k + (n-k-1) = n-1$. Because we do not use this inequality more than $n/k_0$ times, in each term of the sum in \eqref{E:proof_upper_bound2}, we have $\sum_i k_i \geq n - n/k_0$.
\end{itemize}
Putting things together, we obtain that
\[
p_{\eps_0, \eps_0^{s^n}} \leq 2 C^n s^{-({\xi_1} - \delta)(1-1/k_0) n} = 2 s^{-\{({\xi_1} - \delta)(1-1/k_0) - \frac{\log C}{\log s} \} n}.
\]
The probability $p_{e^{-1}, \eps_0^{s^n}}$ is smaller than $p_{\eps_0, \eps_0^{s^n}}$ (if $\eps_0 \leq e^{-1}$) and we have proven that for all $s>1$, $\delta >0$, $k_0 \geq 1$ large enough, for all $\eps_0$ small enough and $n \geq k_0+2$, letting $r = \eps_0^{s^n}$,
\[
p_{e^{-1}, r} \leq 2 \left( |\log \eps_0| |\log r| \right)^{-({\xi_1} - \delta)(1-1/k_0) + \frac{\log C}{\log s} }.
\]
It concludes the proof.
\end{proof}

\subsection{Percolation of large loops}

We now turn to the proof of the phase transition at $\theta = 1$ of the percolative behaviour of large loops, i.e. the proof of Theorem \ref{T:perco_large_loops}.

\begin{proof}[Proof of Theorem \ref{T:perco_large_loops}]
The proof of this result follows the exact same strategy as the proof of Theorem \ref{T:intro_convergence_crossing}. Let $\theta >0$ and $\tau \geq 0$. We consider the functions
\[
\hat{F}_T : (s,t) \in I_\leq^* \mapsto \Prob{ \{s T\} \times \mathbb{S}^1 \overset{\Lc_{\cyl,\tau}^\theta}{\longleftrightarrow} \{ t T \} \times \mathbb{S}^1 }.
\]
We show that $(\hat{F}_T)_{T \geq 1}$ is tight (in the same sense as Lemma \ref{L:crossing_tightness}) and that every subsequential limit $\hat{F}$ satisfies $S \hat{F} = \hat{F}$ where $S$ is defined in \eqref{E:mapS}.
These results are proved exactly like Lemma \ref{L:crossing_tightness} and Corollary \ref{C:crossing}. Crucially, the proofs of Lemma \ref{L:crossing_tightness} and Corollary \ref{C:crossing} only use large loops to create additional connections.
What remains to be discussed is the identification of the subsequential limit $\hat{F}$. This is where we need to distinguish different cases.

$\bullet~ \theta \in (0,1/2]$: By Corollary \ref{C:a_priori_u_bound}, $\hat{F}$ belongs to $\Gc_{1/2}$ \eqref{E:spaceG_alpha}. Lemma \ref{L:fixed-points} states that there is a unique fixed point of $S$ in $\Gc_{1/2}$. Lemma \ref{L:Bessel_explicit} then concludes the identification of the limit with the right hand side of \eqref{E:T_perco_large1}.

$\bullet~ \theta \in (1/2,1)$: To use Lemma \ref{L:fixed-points}, we need to guarantee a certain decay of $\hat{F}$ and this is where we need to assume that the threshold $\tau$ is larger than $\tau_\theta$ \eqref{E:t_theta}. We cannot use the comparison with the case $\theta = 1/2$ any more. However, the restriction to large loops allows us to use the comparison \eqref{E:1D_2D} with the 1D loop soup. Indeed, if there is a crossing of $[sT, tT] \times \mathbb{S}^1$ by a cluster of $\Lc_{\cyl,\tau}^\theta$, the longitudinal projection of the loops have to cross $[sT,tT]$. Let
\[
\lambda :=
2\pi \theta p_{\mathbb{S}^1}(\tau,1,1).
\]
By definition of $\tau_\theta$, $\lambda < 1$ and by \eqref{E:1D_2D},
\[
\theta \boldsymbol{\pi}_* \loopmeasure_\cyl(\d \wp^{\rm 1D}) \indic{T(\wp) > \tau} \leq \lambda \loopmeasure_{\R^+}(\d \wp^{\rm 1D}).
\]
Hence, the probability that there is a cluster of $\Lc_{\cyl,\tau}^\theta$ crossing $[sT, tT] \times \mathbb{S}^1$ is bounded by the probability that a cluster of $\Lc_{\R^+}^\lambda$ crosses $[sT,tT]$. The local time of a 1D loop soup having the law of a Bessel process \cite{Lupu18}, the latter is equal to
\begin{align*}
\Prob{ \forall x \in [sT, tT], R^{(\lambda)}_x > 0}
= \Prob{ \forall x \in [1,t/s], R^{(\lambda)}_x > 0}
= \frac{\sin(\pi \lambda)}{\pi} \int_{t/s-1}^\infty u^{\lambda-1} (u+1)^{-1} \d u,
\end{align*}
by Brownian scaling and Lemma \ref{L:Bessel_explicit}. In particular, any subsequential limit $\hat{F}$ of $(\hat{F}_T)_{T \geq 1}$ belongs to $\Gc_{\alpha}$ for all $\alpha \in (0,1-\lambda)$. We can then conclude the proof as in the case $\theta \in (0,1/2]$.

$\bullet~ \theta \geq 1$: If $\theta > 1$, Lemma \ref{L:fixed-points} shows that $(s,t) \mapsto 1$ is the unique fixed-point of $S$ in $[0,1]^{I_\leq^*}$. This identifies the limit. Alternatively, and in order to cover the critical case $\theta = 1$, we can proceed as follows. Let $\tau \geq 0$ be any threshold, $\theta' \in (1/2,1)$ be very close to one and $\tau' > \tau \vee \tau_{\theta'}$. We have
\begin{align*}
\liminf_{\T \to \infty} \Prob{ \{s T\} \times \mathbb{S}^1 \overset{\Lc_{\cyl,\tau}^{\theta=1}}{\longleftrightarrow} \{ t T \} \times \mathbb{S}^1 }
& \geq \lim_{\T \to \infty} \Prob{ \{s T\} \times \mathbb{S}^1 \overset{\Lc_{\cyl,\tau'}^{\theta'}}{\longleftrightarrow} \{ t T \} \times \mathbb{S}^1 }\\
& = \frac{\sin(\pi \theta')}{\pi} \int_{t/s-1}^\infty u^{\theta'-1} (u+1)^{-1} \d u.
\end{align*}
We then conclude by noticing that the above display converges to 1 as $\theta' \to 1^-$.
\end{proof}

\section{Polar sets and capacity}

This section is dedicated to the proof of Theorem  \ref{T:polar}. By Theorem \ref{T:large_crossing}, we already know with great accuracy the probability that a large cluster gets $\eps$-close to a given point. In order to prove Theorem \ref{T:polar}, we will also need a two-point estimate that we encapsulate in the following result. We will denote by
\begin{equation}
Z_\eps := \P \Big( e^{-1} \partial \D \overset{\Lc_\D^\theta}{\longleftrightarrow} \eps \D  \Big).
\end{equation}

\begin{proposition}\label{P:two-point}
Let $r_0 > 0$ be a fixed macroscopic radius.
For all $\eta>0$, there exists $C>0$ such that for all $x, y \in \D$, $\eps >0$,
\begin{equation}
\label{E:p1}
\frac{1}{Z_\eps^2} \Prob{ D(x,\eps) \overset{\Lc_\D^\theta}{\longleftrightarrow} D(y,\eps) } \leq C |\log |x-y||^{2(1-\theta) + \eta}
\end{equation}
and
\begin{equation}
\label{E:p2}
\frac{1}{Z_\eps^2} \Prob{ D(x,\eps) \overset{\Lc_\D^\theta}{\longleftrightarrow} D(y,\eps), D(x,\eps) \overset{\Lc_\D^\theta}{\longleftrightarrow} \partial D(x,r_0) }
\leq C |\log |x-y||^{1-\theta + \eta}.
\end{equation}
\end{proposition}

Note that if $x$ and $y$ are in the bulk of $\D$ (i.e. at a macroscopic distance to $\partial \D$), then the reverse inequalities also hold. The lower bounds are much simpler to establish by FKG-inequality. We do not write these details since we will only need the upper bounds.

In order to prove this proposition, we will need precise estimates on the probability of crossing two annuli with loops. These are the two-point analogues of the estimate \eqref{E:L_loopmeasure_annulus_quantitative1} and are contained in Section \ref{SS:two_annuli}. In Section \ref{SS:p} we will then prove Proposition \ref{P:two-point}. Finally, we will prove Theorem \ref{T:polar} in Section \ref{SS:polar}. The proof of Theorem \ref{T:polar} is fairly standard once Theorem \ref{T:large_crossing} and Proposition \ref{P:two-point} are established.

\subsection{Crossing two annuli with loops}\label{SS:two_annuli}

\begin{lemma}\label{L:two_annuli}
Let $A>0$ be large.
Let $x,y \in \C$ and $r_x, r_y,R >0$ such that $A |x-y| \leq R \leq 1$ and $\max(r_x,r_y) \leq |x-y|^A$. Then
\begin{equation}
\label{E:L_polar00}
\loopmeasure_{D(x,1)} ( \{ \wp \mathrm{~visiting~} D(x,r_x) \mathrm{~and~} D(y,r_y) \} )
= \left( 1 + o(1) \right) \frac{(\log |x-y|)^2}{\log \frac{1}{r_x} \log \frac{1}{r_y}}
\end{equation}
and
\begin{align}
\label{E:L_polar0}
\loopmeasure_{D(x,1)} ( \{ \wp \mathrm{~visiting~} D(x,r_x), D(y,r_y) \mathrm{~and~} \partial D(x,R) \} )
= (1+o(1)) \frac{\log \frac{1}{R} \log \frac{R}{|x-y|^2}}{\log \frac{1}{r_x} \log \frac{1}{r_y}}
\end{align}
where $o(1) \to 0$ as $A \to \infty$, uniformly in $x,y, r_x, r_y, R$ as above.
\end{lemma}

\begin{proof}
By translation invariance we can assume that $x=0$.
Rooting the loops at the closest point to $x$, we have
\begin{equation}
\label{E:L_polar5}
\loopmeasure_{\D} = \frac{1}{\pi} \int_0^1 r \d r \int_0^{2\pi} \d \theta \bubmeasure_{\D \setminus r \bar{\D}}(r e^{i\theta}).
\end{equation}
To justify this decomposition, we can for instance first root the loops in the domain $\C \setminus \bar{\D}$ at the point whose distance to the origin is maximal:
\[
\loopmeasure_{\C \setminus \bar{\D}} = \frac{1}{\pi} \int_1^\infty R \d R \int_0^{2\pi} \d \theta' \bubmeasure_{R \D \setminus \bar{\D}}(r e^{i\theta'}).
\]
This decomposition follows from the whole plane version of \eqref{E:decomposition_loopmeasure} and then by the restriction property of the loop and bubble measures.
Let $f: z \mapsto -1/z$. By conformal invariance of the loop measure,
\[
\loopmeasure_\D = f \circ \loopmeasure_{\C \setminus \bar{\D}} = \frac{1}{\pi} \int_1^\infty R \d R \int_0^{2\pi} \d \theta' f \circ \bubmeasure_{R \D \setminus \D}(r e^{i\theta'}).
\]
The bubble measure is conformally covariant: $f \circ \bubmeasure_{R \D \setminus \bar{\D}}(r e^{i\theta'}) = \frac{1}{R^4} \bubmeasure_{\D \setminus \frac{1}{R} \bar{D}}(- \frac{1}{R} e^{-i\theta'})$. The change of variable $r = 1/R$ and $\theta = \pi - \theta'$ finishes the proof of \eqref{E:L_polar5}.

Let $r \in (0,r_x)$ and $\theta \in [0,2\pi]$. We compute
\begin{align*}
\bubmeasure_{\D \setminus r \bar{\D}}(r e^{i\theta})( \{ \wp \text{~intersecting~} D(y,r_y) \} )
= \pi \int_{\partial D(y,r_y)} \d z_y
H_{\D \setminus (r \D \cup D(y,r_y)}(r e^{i \theta}, z_y)
H_{\D \setminus r \D}(z_y, re^{i \theta}).
\end{align*}
For all $z_y \in \partial D(y, r_y)$,
\[
H_{\D \setminus r \D}(z_y, re^{i \theta})
= (1+o(1)) \frac{1}{2\pi r} \frac{\log |x-y|}{\log r}.
\]
We can then integrate
\begin{align}
\label{E:L_polar1}
\int_{\partial D(y,r_y)} \d z_y
H_{\D \setminus (r \D \cup D(y,r_y)}(r e^{i \theta}, z_y)
& = \lim_{\eps \to 0} \frac{1}{\eps} \PROB{(r+\eps)e^{i \theta}}{ \tau_{D(y,r_y)} < \tau_{r \D} \wedge \tau_{\partial \D} }.
\end{align}
To compute the above hitting probability, we first stop the Brownian path at the first exit time of $e r \D$:
\begin{align*}
\PROB{(r+\eps)e^{i \theta}}{ \tau_{D(y,r_y)} < \tau_{r \D} \wedge \tau_{\partial \D} }
= \EXPECT{(r+\eps)e^{i \theta}}{ \indic{ \tau_{er \partial \D} < \tau_{r \D} } \PROB{B_{\tau_{e r \partial \D}}}{ \tau_{D(y,r_y)} < \tau_{r \D} \wedge \tau_{\partial \D} } }.
\end{align*}
Starting from any point $w \in er \partial \D$, the probability that a Brownian path hits the small disc $D(y,r_y)$ before exiting the domain $\D$ and before hitting $r \D$ can be accurately estimated:
\begin{equation}
\label{E:L_polar3}
\PROB{w}{ \tau_{D(y,r_y)} < \tau_{r \D} \wedge \tau_{\partial \D}}
= (1+o(1)) \frac{\log \frac{1}{|x-y|}}{\log \frac{1}{r_y} \log \frac{1}{r} - \left( \log \frac{1}{|x-y|} \right)^2}.
\end{equation}
See \cite[Lemma 2.3]{jegoBMC}.
This implies that the integral in \eqref{E:L_polar1} is equal to
\begin{align*}
& (1+o(1)) \frac{\log \frac{1}{|x-y|}}{\log \frac{1}{r_y} \log \frac{1}{r} - \left( \log \frac{1}{|x-y|} \right)^2}
\lim_{\eps \to 0} \frac{1}{\eps} \PROB{(r+\eps)e^{i \theta}}{ \tau_{er \partial \D} < \tau_{r \D} } \\
& = (1+o(1)) \frac{1}{r} \frac{\log \frac{1}{|x-y|}}{\log \frac{1}{r_y} \log \frac{1}{r} - \left( \log \frac{1}{|x-y|} \right)^2}.
\end{align*}
Putting everything together, we obtain that
\begin{align*}
& \loopmeasure_\D( \{ \wp \text{~intersecting~} r_x \D \text{~and~} D(y,r_y) \} ) \\
& = (1+o(1)) \Big( \log \frac{1}{|x-y|} \Big)^2 \int_0^{r_x} \frac{\d r}{r}
\frac{1}{\log \frac{1}{r} \big( \log \frac{1}{r_y} \log \frac{1}{r} - \big( \log \frac{1}{|x-y|} \big)^2 \big) } \\
& = - (1+o(1)) \log \left( 1 - \frac{(\log |x-y|)^2}{\log \frac{1}{r_x} \log \frac{1}{r_y}} \right)
= (1+o(1)) \frac{(\log |x-y|)^2}{\log \frac{1}{r_x} \log \frac{1}{r_y}}.
\end{align*}
This concludes the proof of \eqref{E:L_polar00}.

\eqref{E:L_polar0} follows from the observation that
\begin{align*}
& \loopmeasure_{\D} ( \{ \wp \mathrm{~visiting~} r_x \D, D(y,r_y) \mathrm{~and~} R \partial \D \} ) \\
& = \loopmeasure_{\D} ( \{ \wp \mathrm{~visiting~} r_x \D, D(y,r_y) \} ) - \loopmeasure_{R\D} ( \{ \wp \mathrm{~visiting~} r_x \D, D(y,r_y) \} ).
\end{align*}
We have just computed the first right hand side term. By scaling, the second term is equal to
\[
(1+o(1)) \frac{\big(\log \frac{R}{|x-y|} \big)^2}{\log \frac{R}{r_x} \log \frac{R}{r_y}}
= (1+o(1)) \frac{\big(\log \frac{R}{|x-y|} \big)^2}{\log \frac{1}{r_x} \log \frac{1}{r_y}},
\]
implying that 
\begin{align*}
& \loopmeasure_{\D} ( \{ \wp \mathrm{~visiting~} r_x \D, D(y,r_y) \mathrm{~and~} R \partial \D \} )
= (1+o(1)) \frac{\log \frac{1}{R} \log \frac{R}{|x-y|^2}}{\log \frac{1}{r_x} \log \frac{1}{r_y}}.
\end{align*}
This concludes the proof of \eqref{E:L_polar0}.
\end{proof}

\begin{corollary}\label{C:3crossings}
Let $A>0$ large enough and $x,y, r_x, r_y, R$ be as in the statement of Lemma \ref{L:two_annuli}.
There exists $C= C(A)>0$ such that
the probability that there are loops crossing $D(x,|x-y|/2) \setminus D(x,r_x)$, $D(y,|x-y|/2) \setminus D(y,r_y)$ and $D(x,R) \setminus (D(x,|x-y|/2) \cup D(y,|x-y|/2))$ is bounded by
\begin{equation}
\label{E:3crossings}
C \frac{\log \frac{1}{R} \log \frac{R}{|x-y|^2}}{\log \frac{1}{r_x} \log \frac{1}{r_y}}.
\end{equation}
\end{corollary}

We emphasise that the three crossings described above can be realised by a single loop, or by two or three different loops.

\begin{proof}
We handle each of the following scenarios separately:

$\bullet$ A single loop makes the three crossings.
By \eqref{E:L_polar0}, the probability of that event is at most \eqref{E:3crossings} for some $C>0$.

$\bullet$ One loop crosses $D(x,|x-y|/2) \setminus D(x,r_x)$ and $D(y,|x-y|/2) \setminus D(y,r_y)$ and another loop crosses $D(x,R) \setminus (D(x,|x-y|/2) \cup D(y,|x-y|/2))$.
In that case, there must be a loop visiting both $D(x,r_x)$ and $D(y,r_y)$ and another loop crossing $D(x,R) \setminus D(x,3|x-y|/2)$.
Using that for any (crossing) events $A, B$,
\[
\Prob{ \exists \wp \neq \wp' \in \Lc_\D^\theta: \wp \in A, \wp' \in B }
\leq \E \Big[ \sum_{\wp \neq \wp'} \indic{ \wp \in A} \indic{\wp' \in B} \Big]
= \theta^2 \loopmeasure_D(A) \loopmeasure_D(B),
\]
we can bound the probability of that event by
\begin{align*}
\theta^2 \loopmeasure_D ( \{ \wp \text{~visiting~}  D(x,r_x) \text{~and~} D(y,r_y) \} )
\loopmeasure_D ( \{ \wp \text{~crossing~} D(x,R) \setminus D(x,3|x-y|/2) \} ).
\end{align*}
By \eqref{E:L_loopmeasure_annulus_quantitative1} and \eqref{E:L_polar00}, this is at most
\[
C \frac{(\log|x-y|)^2}{\log \frac{1}{r_x} \log \frac{1}{r_y}} \times \frac{\log \frac{1}{R}}{\log \frac{1}{|x-y|}} \leq C \frac{\log \frac{1}{R} \log \frac{R}{|x-y|^2}}{\log \frac{1}{r_x} \log \frac{1}{r_y}}.
\]

$\bullet$ One loop crosses $D(x,|x-y|/2) \setminus D(x,r_x)$ and $D(x,R) \setminus (D(x,|x-y|/2) \cup D(y,|x-y|/2))$ and another loop crosses $D(y,|x-y|/2) \setminus D(y,r_y)$. As above, the probability of that event can be bounded by
\begin{align*}
& \theta^2 \loopmeasure_D (\{ \wp \text{~crossing~} D(x,R) \setminus D(x,r_x) \})
\loopmeasure_D (\{ \wp \text{~crossing~} D(y,|x-y|/2) \setminus D(y,r_y) \}) \\
& \leq C \frac{\log \frac{1}{R}}{\log \frac{1}{r_x}} \times \frac{\log \frac{1}{|x-y|}}{\log \frac{1}{r_y}} C \frac{\log \frac{1}{R} \log \frac{R}{|x-y|^2}}{\log \frac{1}{r_x} \log \frac{1}{r_y}}
\end{align*}
using \eqref{E:L_loopmeasure_annulus_quantitative1}.

$\bullet$ One loop crosses $D(y,|x-y|/2) \setminus D(y,r_x)$ and $D(x,R) \setminus (D(x,|x-y|/2) \cup D(y,|x-y|/2))$ and another loop crosses $D(x,|x-y|/2) \setminus D(x,r_y)$. The probability of that event can be bounded in a similar way as the previous one.

$\bullet$ The three crossings are made by three different loops. The probability of that event is at most
\begin{align*}
& \theta^3 \loopmeasure_D (\{ \wp \text{~crossing~} D(x,|x-y|/2) \setminus D(x,r_x) \})
\loopmeasure_D (\{ \wp \text{~crossing~} D(y,|x-y|/2) \setminus D(y,r_y) \}) \\
& \times \loopmeasure_D (\{ \wp \text{~crossing~} D(x,R) \setminus D(x,3|x-y|/2) \} )
\end{align*}
which is again bounded by the same expression.
\end{proof}

\subsection{Proof of Proposition \ref{P:two-point}}\label{SS:p}

\begin{proof}[Proof of Proposition \ref{P:two-point}]
Let us first notice that \eqref{E:p1} is a consequence of \eqref{E:p2}. Indeed, by FKG inequality the left hand side of \eqref{E:p2} is at least the left hand side of \eqref{E:p1} times the probability that $D(x,|x-y|/4) \overset{\Lc_\D^\theta}{\longleftrightarrow} \partial D(x,r_0)$ and that there is a loop included in $D(x,|x-y|/2)$ surrounding $D(x,|x-y|/4)$. By Theorem \ref{T:large_crossing}, this shows that the left hand side of \eqref{E:p2} is at least the left hand side of \eqref{E:p1} times $|\log |x-y||^{-1+\theta+o(1)}$.

We now focus on proving \eqref{E:p2}.
Let $\delta >0$ small. First of all, if $|x-y| \leq \exp \left( - |\log \eps|^{1-\delta} \right)$, then we simply bound the left hand side of \eqref{E:p2} by
\[
\frac{1}{Z_\eps^2} \Prob{ D(x,\eps) \overset{\Lc_\D^\theta}{\longleftrightarrow} \partial D(x,r_0) } \leq \frac{C}{Z_\eps} \leq |\log \eps|^{1-\theta+o(1)} \leq |\log |x-y||^{(1-\theta)/(1-2\delta)}.
\]
In the remaining of the proof, we focus on the case $|x-y| \geq \exp \left( - |\log \eps|^{1-\delta} \right)$.
For $z \in \{x,y\}$, let $D_z$ be the disc $D(z,|x-y|/2)$, and $\Lc_{x,y}$ the collection of loops in $\Lc_\D^\theta$ that intersect $\partial D_x \cup \partial D_y$. Define also
\[
r_z = \inf \{ r >0: \exists \wp \in \Lc_{x,y} \mathrm{~intersecting~} \partial D(z,r) \}, \quad \quad z = x,y
\]
and
\[
R = \sup \{ r >0: \exists \wp \in \Lc_{x,y} \mathrm{~intersecting~} \partial D(x,r) \}.
\]
Conditioned on $r_x, r_y$ and $R$, in order to have $D(x,\eps) \overset{\Lc_\D^\theta}{\longleftrightarrow} D(y,\eps)$ and $D(x,\eps) \overset{\Lc_\D^\theta}{\longleftrightarrow} \partial D(x,r_0)$, there must be a cluster of loops included in $D_z$ that crosses $D(z,r_z) \setminus D(z,\eps)$, $z = x,y$ and a cluster of loops included in $D \setminus (\bar{D}_x \cup \bar{D}_y)$ that crosses $D(x,r_0) \setminus D(x,R)$.
By independence of the different involved sets of loops, we therefore obtain that the left hand side of \eqref{E:p2} is at most
\begin{align*}
& \frac{1}{Z_\eps^2} \Expect{ \prod_{z=x,y} \Prob{ D(z,\eps) \overset{\Lc_{D_z}^\theta}{\longleftrightarrow} \partial D(z,r_z) \Big\vert r_z } \Prob{ D(x,R) \overset{\Lc_{\D \setminus (\bar{D}_x \cup \bar{D}_y)}^\theta}{\longleftrightarrow} \partial D(x,r_0) \Big\vert R } }.
\end{align*}
Let $z \in \{x,y\}$.  By scaling,
\begin{align*}
\frac{1}{Z_\eps} \Prob{ D(z,\eps) \overset{\Lc_{D_z}^\theta}{\longleftrightarrow} \partial D(z,r_z) \Big\vert r_z }
= \frac{1}{Z_\eps} \Prob{ \frac{2\eps}{|x-y|} \D \overset{\Lc_{\D}^\theta}{\longleftrightarrow} \frac{2r_z}{|x-y|} \partial \D \Big\vert r_z }.
\end{align*}
Let us assume for a moment that $r_z \leq e^{-1} |x-y| / 2$. By definition of $Z_\eps$, the above right hand side term is the inverse of the probability that $\eps \D$ is connected to $e^{-1} \partial \D$, conditionally on the fact that $2 \eps/ |x-y| \D$ is connected to $2r_z/|x-y|$. By FKG inequality, this is at most
\begin{equation}
\label{E:p4}
\Prob{ \eps \D \overset{\Lc_{\D}^\theta}{\longleftrightarrow} 2 \eps/ |x-y| \partial \D}^{-1}
\Prob{ \frac{2r_z}{|x-y|} \D \overset{\Lc_{\D}^\theta}{\longleftrightarrow} e^{-1} \partial \D}^{-1}
\end{equation}
times the inverse of two welding events that happen with high probability which ensure that the connections have been made at $2 \eps/ |x-y| \partial \D$ and at $\frac{2r_z}{|x-y|}$ (similarly to some previous detailed cases). Because $|x-y| \geq \exp \left( - |\log \eps|^{1-\delta} \right)$, the first probability in \eqref{E:p4} is very close to 1. The second probability is at least $c ( \log \frac{|x-y|}{r_z} )^{-1+\theta-\eta}$ by Theorem \ref{T:large_crossing}. Hence,
\[
\frac{1}{Z_\eps} \Prob{ D(z,\eps) \overset{\Lc_{D_z}^\theta}{\longleftrightarrow} \partial D(z,r_z) \Big\vert r_z } \leq C \left( \log \frac{|x-y|}{r_z} \right)^{1-\theta + \eta}.
\]
Note that this inequality still holds if $r_z \geq e^{-1}|x-y|/2$, with $\log \frac{|x-y|}{r_z}$ replaced by $\log \frac{10|x-y|}{r_z}$. To ease the presentation, we will keep writing $\log \frac{|x-y|}{r_z}$ in the rest of the proof. Alternatively, we will always assume that the logarithms we write are uniformly bounded from below by some positive constant.
Moreover, by Theorem \ref{T:large_crossing},
\[
\Prob{ D(x,R) \overset{\Lc_{D \setminus (\bar{D}_x \cup \bar{D}_y)}^\theta}{\longleftrightarrow} \partial D(x,r_0) \vert R } \leq \Prob{ D(x,R) \overset{\Lc_{D}^\theta}{\longleftrightarrow} \partial D(x,r_0) \vert R }
\leq C | \log R |^{-1+\theta + \eta}.
\]
Putting everything together, we have obtained that the left hand side of \eqref{E:p2} is at most
\begin{equation}
\label{E:p5}
C \Expect{ \prod_{z=x,y} \left( \log \frac{|x-y|}{r_z} \right)^{1-\theta + \eta}
| \log R |^{-1+\theta + \eta} }.
\end{equation}

The rest of the proof is dedicated to bounding the above expression. We have already laid the groundwork for this computation in Section \ref{SS:two_annuli}.
For $k_x, k_y, k_R \geq 0$ and $j_x, j_y, j_R \in \{1, +\infty\}$, let us denote by
\[
p_{k_x,k_y,k_R}^{t_x,t_y,t_R} = \Prob{ \forall z = x,y, \frac{2r_z}{|x-y|} \in (e^{-k_z-t_z},e^{-k_z}], \frac{2R}{|x-y|} \in [e^{k_R}, e^{k_R+t_R}) }.
\]
We will also denote by $p_{k_x,k_y}^{t_x,t_y}$ (and any other type of combination) the probability of the above event where nothing is required for $R$, i.e. $p_{k_x,k_y}^{t_x,t_y} = p_{k_x,k_y,0}^{t_x,t_y,\infty}$. With these notations, we can bound \eqref{E:p5} by
\begin{align*}
C \sum_{k_x,k_y,k_R} (k_xk_y)^{1-\theta+\eta} ( |\log |x-y|| - k_R)^{-1+\theta + \eta}
p_{k_x,k_y,k_R}^{1,1,1}
\end{align*}
where the sum runs over $k_x, k_y \geq 0$ and $k_R \in \{0, \dots, \ceil{|\log |x-y||}\}$. We are going to perform a discrete integration by parts with respect to each of the three variables $k_x, k_y, k_R$. We first integrate by parts the $k_R$ variable ``integrating'' $p_{k_x,k_y,k_R}^{1,1,1} = p_{k_x,k_y,k_R}^{1,1,\infty} - p_{k_x,k_y,k_R+1}^{1,1,\infty}$. We will ``differentiate'' $( |\log |x-y|| - k_R)^{-1+\theta + \eta}$ and bound:
\[
( |\log |x-y|| - k_R)^{-1+\theta + \eta} - ( |\log |x-y|| - k_R+1)^{-1+\theta + \eta} \leq C ( |\log |x-y|| - k_R)^{-2+\theta + \eta}.
\]
This leads to the following upper bound for \eqref{E:p5}:
\begin{align*}
& C \sum_{k_x,k_y,k_R} (k_xk_y)^{1-\theta+\eta} ( |\log |x-y|| - k_R )^{-2+\theta + \eta} p_{k_x,k_y,k_R}^{1,1,\infty} \\
& + C |\log|x-y||^{-1+\theta+\eta} \sum_{k_x,k_y} (k_xk_y)^{1-\theta+\eta} p_{k_x,k_y}^{1,1}.
\end{align*}
Similarly, integrating by parts with respect to $k_y$ shows that this is further bounded by
\begin{align*}
& C \sum_{k_x,k_y,k_R} k_x^{1-\theta+\eta} k_y^{-\theta+\eta} ( |\log |x-y|| - k_R )^{-2+\theta + \eta} p_{k_x,k_y,k_R}^{1,\infty,\infty} \\
& + C \sum_{k_x,k_R} k_x^{1-\theta+\eta} ( |\log |x-y|| - k_R )^{-2+\theta + \eta} p_{k_x,k_R}^{1,\infty} \\
& + C |\log|x-y||^{-1+\theta+\eta} \sum_{k_x,k_y} k_x^{1-\theta+\eta} k_y^{-\theta+\eta} p_{k_x,k_y}^{1,\infty} + C |\log|x-y||^{-1+\theta+\eta} \sum_{k_x} k_x^{1-\theta+\eta} p_{k_x}^{1}.
\end{align*}
Finally, we integrate by parts with respect to $k_x$ to find that \eqref{E:p5} is at most
\begin{align*}
& C \sum_{k_x,k_y,k_R} k_x^{-\theta+\eta} k_y^{-\theta+\eta} ( |\log |x-y|| - k_R )^{-2+\theta + \eta} p_{k_x,k_y,k_R}^{\infty,\infty,\infty} \\
& + C \sum_{k_y,k_R} k_y^{-\theta+\eta} ( |\log |x-y|| - k_R )^{-2+\theta + \eta} p_{k_y,k_R}^{\infty,\infty} \\
& + C \sum_{k_x,k_R} k_x^{-\theta+\eta} ( |\log |x-y|| - k_R )^{-2+\theta + \eta} p_{k_x,k_R}^{\infty,\infty} \\
& + C \sum_{k_R} ( |\log |x-y|| - k_R )^{-2+\theta + \eta} p_{k_R}^{\infty} \\
& + C |\log|x-y||^{-1+\theta+\eta} \sum_{k_x,k_y} k_x^{-\theta+\eta} k_y^{-\theta+\eta} p_{k_x,k_y}^{\infty,\infty}
+ C |\log|x-y||^{-1+\theta+\eta} \sum_{k_y} k_y^{-\theta+\eta} p_{k_y}^{\infty} \\
& + C |\log|x-y||^{-1+\theta+\eta} \sum_{k_x} k_x^{-\theta+\eta} p_{k_x}^{\infty}
+ C |\log|x-y||^{-1+\theta+\eta}.
\end{align*}
The proof of \eqref{E:p2} then follows from the following claim:

\begin{claim}\label{claim}
There exists $C>0$ such that
\begin{gather}
\label{E:c1}
\sum_{k_z} k_z^{-\theta+\eta} p_{k_z}^{\infty} \leq C |\log|x-y||^{1-\theta+\eta}, \quad z = x,y; \\
\label{E:c2}
\sum_{k_z,k_R} k_z^{-\theta+\eta} ( |\log |x-y|| - k_R )^{-2+\theta + \eta} p_{k_z,k_R}^{\infty,\infty} \leq C |\log |x-y||^{2\eta}, \quad z=x,y; \\
\label{E:c3}
\sum_{k_x,k_y} k_x^{-\theta+\eta} k_y^{-\theta+\eta} p_{k_x,k_y}^{\infty,\infty}
\leq C |\log|x-y||^{2-2\theta+2\eta}; \\
\label{E:c4}
\sum_{k_x,k_y,k_R} (k_x k_y)^{-\theta+\eta} ( |\log |x-y|| - k_R )^{-2+\theta + \eta} p_{k_x,k_y,k_R}^{\infty,\infty,\infty} \leq C |\log|x-y||^{1-\theta+3\eta}.
\end{gather}
\end{claim}

\begin{proof}[Proof of Claim \ref{claim}]
We start with \eqref{E:c1}. Let us prove it in the case $z=y$, the case $z=x$ being similar.
$p_{k_y}^\infty$ is bounded by the probability that there is a loop in $\Lc_{D(y,10)}^\theta$ crossing $D(y,|x-y|/2) \setminus D(y,e^{-k_y}|x-y|/2)$ which is by Lemma \ref{L:loopmeasure_annulus_quantitative} at most $C \frac{|\log|x-y||}{|\log|x-y||+k_y}$. Hence
\[
\sum_{k_y} k_y^{-\theta+\eta} p_{k_y}^{\infty} \leq C \sum_{k_y} k_y^{-\theta+\eta} \frac{|\log|x-y||}{|\log|x-y||+k_y} \leq C |\log|x-y||^{1-\theta+\eta}
\]
proving \eqref{E:c1}.

We now prove \eqref{E:c2}, again in the case $z = y$.
Let $A>0$ be large.
When $R \leq A$, we simply bound $( |\log |x-y|| - k_R )^{-2+\theta + \eta} p_{k_y,k_R}^{\infty,\infty}$ by $C|\log |x-y||^{-2+\theta + \eta} p_{k_y}^{\infty}$ and use \eqref{E:c1}:
\begin{align*}
& \sum_{k_y} \sum_{k_R \leq A} k_y^{-\theta+\eta} ( |\log |x-y|| - k_R )^{-2+\theta + \eta} p_{k_y,k_R}^{\infty,\infty} \\
& \leq C |\log |x-y||^{-2+\theta + \eta} \sum_{k_y} k_y^{-\theta+\eta} p_{k_y}^{\infty} \leq C |\log |x-y||^{-1 + 2\eta}.
\end{align*}
Let $k_R \geq A$. Since the disc $D(x,e^{k_R}|x-y|/2)$ contains the disc $D(y,(e^{k_R}-2)|x-y|/2)$ and since the disc $D(y, 3|x-y|/2)$ contains $D(x,|x-y|/2) \cup D(y,|x-y|/2)$, $p_{k_y,k_R}^{\infty,\infty}$ is bounded by the probability that there is a loop crossing $D(y,(e^{k_R}-2)|x-y|/2) \setminus  D(y, 3|x-y|/2)$ and a loop crossing $D(y,|x-y|/2) \setminus D(y,e^{-k_y}|x-y|/2)$. Because $k_R$ is large, $e^{k_R} - 2$ is very close to $e^{k_R}$. Using Lemma \ref{L:loopmeasure_annulus_quantitative} and dividing the above event according to whether these crossings are made by the same loop or two different loops (see the proof of Corollary \ref{C:3crossings} for more details), we obtain that
\begin{align*}
p_{k_y,k_R}^{\infty,\infty} \leq C \frac{|\log|x-y|| - k_R}{|\log|x-y||+k_y}.
\end{align*}
This shows that 
\begin{align*}
& \sum_{k_y} \sum_{k_R \geq A} k_y^{-\theta+\eta} ( |\log |x-y|| - k_R )^{-2+\theta + \eta} p_{k_y,k_R}^{\infty,\infty} \\
& \leq C \sum_{k_y} \sum_{k_R \geq A} \frac{k_y^{-\theta+\eta}}{|\log|x-y||+k_y} ( |\log |x-y|| - k_R )^{-1+\theta + \eta}
\leq C |\log |x-y||^{2\eta}
\end{align*}
finishing the proof of \eqref{E:c2}.

The proof of \eqref{E:c3} is similar to the proof of \eqref{E:c2} dividing the sum into two terms depending on whether $k_x \geq A |\log |x-y||$ or not. The probability that a single loop hits both $D(x,e^{-k_x}|x-y|/2)$ and $D(y,e^{-k_y}|x-y|/2)$ is handled by \eqref{E:L_polar00}.

Finally, we move on to the proof of \eqref{E:c4}. The contribution of integers $k_x \leq A|\log|x-y||$ and $k_y,k_R \geq 0$ to the sum \eqref{E:c4} can be bounded by the left hand side of \eqref{E:c2} for $z=y$, times
\[
\sum_{k_x}^{A|\log |x-y||} k_x^{-\theta+\eta} \leq C |\log |x-y||^{1-\theta+\eta}.
\]
This contribution is therefore at most $C|\log |x-y||^{1-\theta+3\eta}$. We can now assume that $k_x \geq A|\log|x-y||$. Similarly, using \eqref{E:c2} for $z=x$ and \eqref{E:c3}, we can assume that $k_y \geq A|\log|x-y||$ and $k_R \geq A$. Using Corollary \ref{C:3crossings},  we deduce that the contribution of these integers to the sum \eqref{E:c4} is bounded by
\begin{align*}
|\log|x-y|| \sum_{k_x,k_y \geq A|\log |x-y||} \sum_{k_R = A}^{|\log|x-y||} \frac{k_x^{-\theta+\eta}}{|\log|x-y|| + k_x} \frac{k_y^{-\theta+\eta}}{|\log|x-y|| + k_y} ( |\log |x-y|| - k_R )^{-1+\theta + \eta}
\end{align*}
which is at most $C |\log |x-y||^{1-\theta+3\eta}$ as claimed in \eqref{E:c4}. This concludes the proof of Claim \ref{claim} and of Proposition \ref{P:two-point}.
\end{proof}
\end{proof}

\begin{remark}
In the proof above, we showed that for all $\eta>0$ and $x \neq y \in \D$, \eqref{E:p5} is at most $C(\eta) |\log|x-y||^{1-\theta + 3\eta}$. The very same proof would also show that for all $\eta>0$ and $x \neq y \in \D$,
\begin{equation}
\Expect{ \prod_{z=x,y} \left( \log \frac{|x-y|}{r_z} \right)^{1-\theta + \eta} } \leq C(\eta) |\log |x-y||^{2(1-\theta) + 2\eta}
\end{equation}
and for all $A>0$ large,
\begin{equation}
\Expect{ \prod_{z=x,y} \left( \log \frac{|x-y|}{r_z} \right)^{1-\theta + \eta} \indic{\exists z \in \{x,y\}: r_z \leq |x-y|^A} } \leq c(\eta,A) |\log |x-y||^{2(1-\theta) + 2\eta}
\end{equation}
where $c(\eta,A) \to 0$ as $A \to \infty$.
We will not use these estimates in the current paper, but they will be needed in the companion paper \cite{JLQ23b}.
\end{remark}

\subsection{Proof of Theorem \ref{T:polar}}\label{SS:polar}

\begin{proof}[Proof of Theorem \ref{T:polar}]
Let $A \subset D$ be a closed set and assume that $A$ is not polar, i.e. assume that the probability that the closure of some cluster hits $A$ is positive. 
Let $\alpha < 1-\theta$. We want to show that $\Cap_{\log^\alpha}(A) > 0$.
Let $\Cc_i$, $i \geq 1$, be the clusters of $\Lc_\D^\theta$ ordered in decreasing diameters. By a union bound, there exists $i \geq 1$ such that $\Prob{\bar{\Cc}_i \cap A \neq \varnothing}>0$. Since $\Cc_i$ has a positive diameter almost surely, there is also $\eta >0$ such that $\Prob{\bar{\Cc}_i \cap A \neq \varnothing, \text{diam}(\Cc_i) > \eta }>0$. On the event that $\bar{\Cc}_i$ intersects $A$ and $\text{diam}(\Cc_i) > \eta$, let $X$ be any point of $\bar{\Cc}_i \cap A$ defined in a measurable way (for instance, choose the one that has minimal imaginary part among points of $\bar{\Cc}_i \cap A$ with minimal real part).
Let $\mu$ be the probability measure on $A$ defined for all Borel set $B$ by
\[
\mu(B) := \Prob{ X \in B, \bar{\Cc}_i \cap A \neq \varnothing, \text{diam}(\Cc_i)>\eta } / \Prob{ \bar{\Cc}_i \cap A \neq \varnothing, \text{diam}(\Cc_i)>\eta }.
\]
To show that $\Cap_{\log^\alpha}(A) >0$, it is enough to show that $\mu$ has a finite $\log^\alpha$-energy:
\begin{equation}
\label{E:polar1}
\int |\log |x-y||^\alpha \mu(\d x) \mu(\d y) < \infty.
\end{equation}
Let $y \in A$. Without loss of generality, we can assume that the diameter of $D$ is smaller than 1 so that $D \subset D(y,1)$. We can bound
\begin{align*}
\int |\log |x-y||^\alpha \mu(\d x)
\leq \sum_{k=0}^\infty \mu( D(y,e^{-k}) \setminus D(y,e^{-k-1}) ) (k+1)^{\alpha}.
\end{align*}
For any large $N \geq 1$, we perform the following discrete integration by parts
\begin{align*}
\sum_{k=0}^N \mu( D(y,e^{-k}) \setminus D(y,e^{-k-1}) ) (k+1)^{\alpha}
= \sum_{k=0}^N \left( \mu( D(y,e^{-k}) ) - \mu ( D(y,e^{-k-1}) ) \right) (k+1)^{\alpha} \\
= \mu( D(y,1) ) - \mu (D(y,e^{-N-1})) (N+1)^\alpha + \sum_{k=1}^N \mu( D(y,e^{-k})) ( (k+1)^\alpha - k^\alpha ).
\end{align*}
Using the bound $(k+1)^\alpha - k^\alpha \leq C k^{\alpha-1}$ for all $k \geq 1$, we obtain that
\[
\int |\log |x-y||^\alpha \mu(\d x)
\leq 1+ C \sum_{k=1}^\infty k^{\alpha-1} \mu ( D(y,e^{-k})).
\]
On the event that the diameter of $\Cc_i$ is larger than $\eta$, $\Cc_i$ has to intersect the circle $\partial D(y,\eta)$.
By definition of $\mu$, $\mu ( D(y,e^{-k}))$ can therefore be bounded by
\[
\frac{\P \Big( D(y,e^{-k}) \overset{\Lc_\D^\theta}{\longleftrightarrow} \partial D(y,\eta) \Big) }{ \Prob{ \bar{\Cc}_i \cap A \neq \varnothing, \text{diam}(\Cc_i)>\eta } }
\leq \frac{\P \Big( D(y,e^{-k}) \overset{\Lc_{D(y,1)}^\theta}{\longleftrightarrow} \partial D(y,\eta) \Big) }{ \Prob{ \bar{\Cc}_i \cap A \neq \varnothing, \text{diam}(\Cc_i)>\eta } } \leq k^{-1+\theta +o(1)}
\]
where in the last inequality we used Theorem \ref{T:large_crossing}. Putting things together, we obtain that for all $y \in A$,
\[
\int |\log |x-y||^\alpha \mu(\d x) \leq 1 + C \sum_{k=1}^\infty k^{\alpha-1 - 1+\theta + o(1)}.
\]
Since $\alpha < 1-\theta$ the above sum is finite. Integrating with respect to $y$, this proves \eqref{E:polar1} as desired.

We now turn to the other direction, i.e. we suppose that $\Cap_{\log^\alpha}(A) >0$ for some $\alpha > 1-\theta$ and we are going to show that $A$ is hit by the closure of a cluster with positive probability. Since $\Cap_{\log^\alpha}(A) >0$, there exists a probability measure $\mu$ on $A$ such that
\begin{equation}
\label{E:energy}
\int |\log |x-y||^\alpha \mu(\d x) \mu(\d y) < \infty.
\end{equation}
Let $\Cc$ be the outermost cluster surrounding some given point of $D \setminus A$. Let $d >0$ be the distance between that point and $A$ (which is positive since $A$ is closed).
For all $\eps >0$, we define the following random variable
\[
J_\eps := \frac{1}{Z_\eps} \int_A \indic{\Cc \cap D(x,\eps) \neq \varnothing} \mu(\d x)
\quad \text{where} \quad
Z_\eps = \P \Big( e^{-1} \partial \D \overset{\Lc_{\D}^\theta}{\longleftrightarrow} \eps \D \Big).
\]
Because $A$ and $\bar{\Cc}$ are closed, if they do not intersect each other, they have to stay at a positive distance from each other. Hence,
\[
\Prob{ \bar{\Cc} \cap A \neq \varnothing } \geq \liminf_{\eps \to 0} \Prob{J_\eps >0} \geq \liminf_{\eps \to 0} \Expect{J_\eps}^2 / ~\Expect{J_\eps^2}
\]
where in the last inequality we used Paley--Zygmund inequality. Thanks to the normalising constant $Z_\eps$, the expectation of $J_\eps$ remains uniformly bounded away from 0. Proposition \ref{P:two-point} allows us to bound the second moment of $J_\eps$:
\begin{align*}
\Expect{J_\eps^2} & = \int_{A \times A} \mu(\d x) \mu(\d y) \frac{1}{Z_\eps^2} \Prob{ \Cc \cap D(x,\eps) \neq \varnothing, \Cc \cap D(y,\eps) \neq \varnothing } \\
& \leq \int_{A \times A} \mu(\d x) \mu(\d y) |\log |x-y||^{1-\theta+o(1)}
\end{align*}
which is finite by \eqref{E:energy}. This proves that $\Prob{ \bar{\Cc} \cap A \neq \varnothing } >0$ as desired.
\end{proof}

\appendix

\section{Lower bound on \texorpdfstring{$f_\infty$}{f infinity}}\label{S:app_lower}
Let $\theta \in (0,1)$ and $f_\infty$ be the unique fixed point of $T$ \eqref{E:mapT} that belongs to $\bigcap_{\alpha \in (0,1-\theta)} \Fc_\alpha$ \eqref{E:spaceF_alpha}; see Lemma \ref{L:fixed-points}. Let $\xi_1$ be the exponent describing the decay of $f_\infty$: $f_\infty(s) = s^{-\xi_1+o(1)}$ as $s \to \infty$; see Corollary \ref{C:submultiplicative}.
As a direct consequence of Lemma \ref{L:fixed-points}, $\xi_1 \geq 1-\theta$.
As we saw in Section \ref{S:Bessel}, $f_\infty$ possesses a simple explicit expression, implying in particular that $\xi_1 = 1-\theta$.
We give here an independent short proof of this fact that does not rely on this expression:

\begin{lemma}\label{L:upper_bound_exponent}
The function $f_\infty$ satisfies
$
\int_1^\infty f_\infty(s) s^{-\theta} \d s = \infty.
$
In particular, $\xi_1 \leq 1-\theta$.
\end{lemma}

\begin{proof}
Let $u : s \in (1,\infty) \mapsto (s-1)^{-\theta} (1-f_\infty(s)) \in [0,\infty)$.
The fact that $Tf_\infty = f_\infty$ implies that for all $t > 1$,
\begin{equation}
\label{E:proof_infinity}
u(t) = \theta \int_1^\infty (s-1)^\theta (s+t-1)^{-\theta -1} u(s) \d s.
\end{equation}
This in particular implies that $u(s)$ has a finite limit as $s \to 1^+$.
We now compute the integral of $u$ times $f_\infty$. All the terms involved below are nonnegative, so we can safely exchange integrals and we obtain by definition of $T$:
\begin{align*}
& \int_1^\infty u(s) f_\infty(s) \d s
= \int_1^\infty u(s) T f_\infty (s) \d s \\
& = \int_1^\infty u(s) \Big( 1 - \Big(1-\frac1{s}\Big)^\theta \Big) \d s
+ \int_1^\infty f_\infty(t) \left( \theta \int_1^\infty (s-1)^\theta (s+t-1)^{-\theta-1} u(s) \d s \right) \d t.
\end{align*}
Using \eqref{E:proof_infinity} we deduce that
\begin{align*}
& \int_1^\infty u(s) f_\infty(s) \d s
= \int_1^\infty u(s) \Big( 1 - \Big(1-\frac1{s}\Big)^\theta \Big) \d s
+ \int_1^\infty u(t) f_\infty(t) \d t.
\end{align*}
The first integral on the right hand side being positive, the integral of $u$ times $f_\infty$ has to be infinite. This concludes the proof since $u(s)$ has a finite limit as $s \to 1^+$ and $u(s) \sim s^{-\theta}$ as $s \to \infty$.
\end{proof}

\section{Hypergeometric function}

In this section we recall the definition and a few properties of the hypergeometric function ${}_2 F_1$.
For $a,b \in \R$ and $c \in \R \setminus \{0, -1, -2, \dots \}$,
\begin{equation}
\label{E:Hypergeometric}
{}_2 F_1(a,b,c,z) = 1 + \sum_{n \geq 1} \frac{(a)_n (b)_n}{(c)_n} \frac{z^n}{n!},
\quad \quad |z| < 1
\end{equation}
where $(q)_n = q(q+1) \dots (q+n-1)$ for all $n \geq 1$.
${}_2 F_1(a,b,c,\cdot)$ can be analytically continued to the whole complex plane.
When $b<c$ and $z \in \C$ is not a real number larger or equal than 1, ${}_2 F_1(a,b,c,z)$ has the following integral representation \cite[15.3.1]{special}
\begin{equation}
\label{E:Hypergeometric_integral}
\frac{\Gamma(b) \Gamma(c-b)}{\Gamma(c)} {}_2 F_1(a,b,c,z) = \int_0^1 x^{b-1} (1-x)^{c-b-1} (1-zx)^a \d x.
\end{equation}
The following Pfaff transformation relates the behaviour at infinity of ${}_2 F_1(a,b,c,\cdot)$ in terms of the behaviour at $z=1$ of a different hypergeometric function \cite[15.3.5]{special}:
\begin{equation}
\label{E:Hypergeometric_Pfaff}
{}_2 F_1(a,b,c,z) = (1-z)^{-b} {}_2 F_1(c-a,b,c,z/(z-1))
\end{equation}
and the following transformation relates the behaviour at $z=1$ to the one at zero \cite[15.3.6]{special}:
\begin{align}
\label{E:Hypergeometric_01}
{}_2 F_1(a,b,c,z)
& = \frac{\Gamma(c) \Gamma(c-a-b)}{\Gamma(c-a) \Gamma(c-b)} {}_2 F_1 (a,b,a+b+1-c,1-z) \\
& + \frac{\Gamma(c) \Gamma(a+b-c)}{\Gamma(a) \Gamma(b)} (1-z)^{c-a-b} {}_2 F_1 (c-a,c-b,1+c-a-b,1-z), \nonumber
\end{align}
provided all the terms above make sense. In particular, \eqref{E:Hypergeometric_01} implies that when $c>a+b$, the value at $z=1$ can be expressed in terms of the Gamma function \cite[15.1.20]{special}
\begin{equation}
\label{E:Hypergeometric_z=1}
{}_2 F_1(a,b,c,1) = \frac{\Gamma(c) \Gamma(c-a-b)}{\Gamma(c-a) \Gamma(c-b)}.
\end{equation}

\bibliographystyle{plain}
\bibliography{bibliography}

\end{document}